\documentclass[fontsize=11pt]{amsart}

\newcounter{defcounter}
\numberwithin{defcounter}{section}
\numberwithin{equation}{section}

\usepackage{appendix}
\usepackage{amsmath}
\usepackage{amsthm}
\usepackage{amssymb}
\usepackage{calligra}
\usepackage{mathtools}
\usepackage{mathrsfs}
\usepackage{physics}
\usepackage[margin=1in]{geometry}
\usepackage{scrextend}
\usepackage{tikz-cd}
\usepackage[shortlabels]{enumitem}
\usepackage{colonequals}
\usepackage{hyperref}
\usepackage[all]{xy}
\usepackage{pdfsync}
\usepackage[mathscr]{euscript}
\usepackage[notcolon, notquote]{hanging}
\usepackage{verbatim}
\usepackage{tikz}
\usepackage{stackengine}
\usepackage{adjustbox}
\usepackage[nopostdot,nogroupskip,sort=use]{glossaries}

\renewcommand{\P}{\mathbb{P}}
\newcommand{\cI}{\mathcal{I}}

\renewcommand{\deg}[1]{\textnormal{deg}({#1})}

\newcommand{\lra}{\longrightarrow}

\newcommand{\OO}{\mathcal{O}}

\DeclareMathOperator{\bl}{\textnormal{bl}}

\DeclareMathOperator{\res}{\textnormal{res}}

\DeclareMathOperator{\pr}{\textnormal{pr}}

\DeclareMathOperator{\Sym}{\textnormal{Sym}}

\DeclareMathOperator{\sm}{\textnormal{sm}}
\DeclareMathOperator{\Gor}{\textnormal{Gor}}
\DeclareMathOperator{\Sing}{\textnormal{Sing}}
\let\ker\relax
\DeclareMathOperator{\ker}{\textnormal{ker}}

\DeclareMathOperator{\SHom}{\mathcal{H}\hspace*{-1.5pt}\textit{om}}

\DeclareMathOperator{\SExt}{\mathcal{E}\hspace*{-1.5pt}\textit{xt}}
\DeclareMathOperator{\STor}{\mathcal{T}\hspace*{-2.5pt}\textit{or}}

\makeatletter
\renewenvironment{proof}[1][\proofname]{\par
  \pushQED{\qed}%
  \normalfont
  \topsep-5\p@\@plus6\p@\relax 
  \trivlist
  \item[\hskip\labelsep
        \upshape
    \emph{#1}\@addpunct{.}]\ignorespaces
}{%
  \popQED\endtrivlist\@endpefalse
}
\makeatother

\theoremstyle{plain}
\newtheoremstyle{mytheoremstyle}{5pt}{-5pt}{\itshape}{}{\bfseries}{.}{.5em}{} 
\theoremstyle{mytheoremstyle}
\newtheorem{theorem}{Theorem}[section]

\newtheorem{theoremalpha}{Theorem}

\theoremstyle{definition}
\newtheoremstyle{mydefinitionstyle}{5pt}{-5pt}{}{}{\bfseries}{.}{.5em}{}
\theoremstyle{mydefinitionstyle}
\newtheorem{definition}[theorem]{Definition}

\newtheorem{corollary}[theorem]{Corollary}

\newtheorem{remark}[theorem]{Remark}

\newtheorem{proposition}[theorem]{Proposition}

\newtheorem{lemma}[theorem]{Lemma}

\newtheorem{question}[theorem]{Question}

\newtheorem{corollaryalpha}[theoremalpha]{Corollary}

\tikzset{
  symbol/.style={
    draw=none,
    every to/.append style={
      edge node={node [sloped, allow upside down, auto=false]{$#1$}}}
  }
}
\tikzcdset{scale cd/.style={every label/.append style={scale=#1},
    cells={nodes={scale=#1}}}}

\author{Doyoung Choi}
\address{Department of Mathematical Sciences, KAIST, Daejeon 34141, Republic of Korea}
\email{cdy4019@kaist.ac.kr}

\author{Justin Lacini}
\address{Department of Mathematics, Princeton University, Princeton, New Jersey  08544}
\email{{\tt jlacini@princeton.edu}}

\author{Jinhyung Park}
\address{Department of Mathematical Sciences, KAIST, Daejeon 34141, Republic of Korea}
\email{parkjh13@kaist.ac.kr}

\author{John Sheridan}
\address{Department of Mathematics, Princeton University, Princeton, New Jersey  08544}
\email{{\tt john.sheridan@princeton.edu}}

\newglossaryentry{mathcalZk}
{
        name=\ensuremath{\mathcal{Z}_k},
        description={Universal family of length-$k$ schemes,}
}
\newglossaryentry{tilde-mathcalZk}
{
        name=\ensuremath{\widetilde{\mathcal{Z}}_k},
        description={Fiber product of $\pr_2 : \mathcal{Z}_k \lra X_*^{[k]}$ and $\widetilde{q}_k : \widetilde{X}_*^k \lra X_*^{[k]}$,}
}
\newglossaryentry{Deltak}
{
        name=\ensuremath{\Delta_k},
        description={Union of all diagonals of $X^k$,}
}
\newglossaryentry{tautilde}
{
        name=\ensuremath{\widetilde{\tau}},
        description={Base change of $\tau$ via $\pi_m$,}
}
\newglossaryentry{r-tilde}
{
        name=\ensuremath{\widetilde{\rho}},
        description={Base change of $\rho$ via $\pi_m$,}
}
\newglossaryentry{Sigma-k}
{
        name=\ensuremath{\Sigma_k},
        description={Variety of secant $k$-planes relative to $L$,}
}
\newglossaryentry{Hilb-k}
{
        name=\ensuremath{X^{[k]}},
        description={Hilbert scheme of $0$-dimensional schemes of length $k$,}
}
\newglossaryentry{Hilb-k-sm}
{
        name=\ensuremath{X_{\operatorname{sm}}^{[k]}},
        description={Smoothable component of the Hilbert scheme of $0$-dimensional schemes of length $k$,}
}
\newglossaryentry{Hilb-k-gor}
{
        name=\ensuremath{X_{\operatorname{Gor}}^{[k]}},
        description={Hilbert scheme of $0$-dimensional Gorenstein schemes of length $k$,}
}

\newglossaryentry{Hilb-k-sm-gor}
{
        name=\ensuremath{X_{\operatorname{sm},\operatorname{Gor}}^{[k]}},
        description={Smoothable component of the Hilbert scheme of $0$-dimensional Gorenstein schemes of length $k$,}
}
\newglossaryentry{Hilb-k-lci}
{
        name=\ensuremath{X_{\operatorname{lci}} ^{[k]}},
        description={Hilbert scheme of $0$-dimensional local complete intersection schemes of length $k$,}
}
\newglossaryentry{Hilb-k-*}
{
        name=\ensuremath{X_* ^{[k]}},
        description={Hilbert scheme of $0$-dimensional schemes of length $k$ and support of cardinality at least $k-1$,}
}
\newglossaryentry{Hilbert--Chow}
{
        name=\ensuremath{h},
        description={Hilbert--Chow morphism,}
}
\newglossaryentry{HC-exc}
{
        name=\ensuremath{E_k},
        description={Exceptional divisor of the Hilbert--Chow morphism,}
}
\newglossaryentry{tilde-h-k}
{
        name=\ensuremath{\widetilde{h}},
        description={Base-change of $h$ along $q$,}
}
\newglossaryentry{tilde-q-k}
{
        name=\ensuremath{\widetilde{q}},
        description={Base-change of $q$ along $h$,}
}
\newglossaryentry{tilde-X-*-k}
{
        name=\ensuremath{\widetilde{X}_*^k},
        description={Blow up of $X_* ^k$ along $\Delta_k$,}
}
\newglossaryentry{X-*-k}
{
        name=\ensuremath{X_*^k},
        description={Set $X^k \setminus \Delta_k$,}
}
\newglossaryentry{X-k}
{
        name=\ensuremath{X_k},
        description={Isospectral Hilbert scheme,}
}
\newglossaryentry{Sym-X-*-k}
{
        name=\ensuremath{\operatorname{Sym}_*^k(X)},
        description={Open subset of $\operatorname{Sym}^k(X)$ consisting of $0$-cycles supported at least at $k-1$ points,}
}
\newglossaryentry{Gamma-k}
{
        name=\ensuremath{\Gamma_k},
        description={Exceptional divisor of $\widetilde{h}$,}
}
\newglossaryentry{quotient-to-sym}
{
        name=\ensuremath{q},
        description={Quotient morphism $X^k \lra \operatorname{Sym}^k(X)$,}
}
\newglossaryentry{delta-k}
{
        name=\ensuremath{\delta_k},
        description={Vandermonde line bundle on the Hilbert scheme of points,}
}
\newglossaryentry{S-k}
{
        name=\ensuremath{S_{k,L}},
        description={Line bundle descending from $L^{\boxtimes k}$ to $\operatorname{Sym}^k(X)$,}
}
\newglossaryentry{T-k}
{
        name=\ensuremath{T_{k,L}},
        description={Pullback of $S_{k,L}$ via the Hilbert--Chow morphism,}
}
\newglossaryentry{D-k}
{
        name=\ensuremath{D_k},
        description={Differential map $\mathcal{I}_{\mathcal{Z}_k}/ \mathcal{I}_{\mathcal{Z}_k}^2\lra \operatorname{pr}_2 ^* \Omega_{X^{[k]}}^1 \vert_{\mathcal{Z}_k}$,}
}
\newglossaryentry{Nested-Hilb}
{
        name=\ensuremath{X^{[m,k]}},
        description={Nested Hilbert scheme of $m$-subschemes of $k$-schemes,}
}
\newglossaryentry{tau-m-k}
{
        name=\ensuremath{\tau},
        description={Morphism $X^{[m,k]}\lra X^{[m]}$,}
}
\newglossaryentry{a-k}
{
        name=\ensuremath{A_{k,L}},
        description={Line bundle $N_{k,L}\otimes \delta_{k}^{-1}$,}
}
\newglossaryentry{a}
{
        name=\ensuremath{A_L},
        description={Line bundle $N_L\otimes \delta^{-1}$,}
}
\newglossaryentry{r-m-k}
{
        name=\ensuremath{\rho},
        description={Morphism $X^{[m,k]}\lra X^{[k]}$,}
}
\newglossaryentry{res}
{
       name=\ensuremath{\operatorname{res}},
        description={Morphism $X^{[k-1,k]}\lra X$,}
}
\newglossaryentry{relative-family-small}
{
        name=\ensuremath{\mathcal{V}_{m,k}},
        description={Pullback of $\mathcal{Z}_m$ via $X^{[m,k]}\lra X^{[m]}$,}
}
\newglossaryentry{relative-family-big}
{
        name=\ensuremath{\mathcal{W}_{m,k}},
        description={Pullback of $\mathcal{Z}_k$ via $X^{[m,k]}\lra X^{[k]}$,}
}
\newglossaryentry{blow-up-of-universal-family}
{
        name=\ensuremath{\operatorname{bl}},
        description={Morphism $\operatorname{Bl}_{\mathcal{Z}_{k-1}}X\times X^{[k-1]}\lra X\times X^{[k-1]}$ --- it may be identified with $\operatorname{res}\times \tau$,}
}
\newglossaryentry{exceptional-divisor-F}
{
        name=\ensuremath{F_{k-1}},
        description={Exceptional divisor of $\operatorname{bl}_k$,}
}
\newglossaryentry{sigma-k}
{
        name=\ensuremath{\sigma_k},
        description={Secant variety of $k$-secant planes relative to $L$,}
}
\newglossaryentry{sigma-k-X-L}
{
        name=\ensuremath{\sigma_k(X,L)},
        description={Secant variety of $k$-secant planes relative to $L$,}
}
\newglossaryentry{sigma-m-k}
{
        name=\ensuremath{\sigma_{m,k}},
        description={Relative secant variety,}
}
\newglossaryentry{kappa-m-k}
{
        name=\ensuremath{\kappa_{m,k}},
        description={Relative cactus scheme,}
}
\newglossaryentry{secant-bundle}
{
        name=\ensuremath{B^k},
        description={Bundle of $k$-secant planes $\mathbb{P}(E_{k,L})$,}
}
\newglossaryentry{pi-k}
{
        name=\ensuremath{\pi_k},
        description={Morphism $B^k\lra X^{[k]}$,}
}
\newglossaryentry{relative-secant-bundle}
{
        name=\ensuremath{B^{m,k}},
        description={Relative secant bundle of $m$-secant planes contained in $k$-secant planes over $X^{[m,k]}$,}
}
\newglossaryentry{alpha-k}
{
        name=\ensuremath{\alpha_k},
        description={Morphism $B^k\lra \sigma_k$,}
}
\newglossaryentry{alpha-m-k}
{
        name=\ensuremath{\alpha_{m,k}},
        description={Morphism $B^{m,k}\lra \sigma_{m,k}$,}
}
\newglossaryentry{cactus-k}
{
        name=\ensuremath{\kappa_k},
        description={Cactus scheme of $X$ relative to $L$,}
}
\newglossaryentry{u-k}
{
        name=\ensuremath{U^k},
        description={Open subset $B^k\setminus \sigma_{k-1,k}$,}
}
\newglossaryentry{u-m-k}
{
        name=\ensuremath{U^{m,k}},
        description={Pullback of $U^m$ via $\tau$,}
}
\newglossaryentry{e-k}
{
        name=\ensuremath{e_{k}},
        description={Morphism $\pi_k^* \operatorname{pr}_{2,*}(\mathcal{I}_{\mathcal{Z}_k}/\mathcal{I}_{\mathcal{Z}_k}^2)\otimes \mathcal{O}_{B^k}(-1)\lra \pi_k^* \Omega_{X^{[k]}}$,}
}
\newglossaryentry{E-k-L}
{
        name=\ensuremath{E_{k,L}},
        description={Sheaf of $k$-secant planes (associated to the line bundle $L$) on $X^{[k]}$,}
}
\newglossaryentry{E-m-k-L}
{
        name=\ensuremath{E_{m,k,L}},
        description={Relative sheaf of $m$-secant planes (associated to the line bundle $L$) on $X^{[m,k]}$,}
}
\newglossaryentry{M-k-L}
{
        name=\ensuremath{M_{k,L}},
        description={Syzygy bundle of $E_{k,L}$,}
}
\newglossaryentry{M-m-k-L}
{
        name=\ensuremath{M_{m,k,L}},
        description={Syzygy bundle of $E_{m,k,L}$,}
}
\newglossaryentry{N-k-L}
{
        name=\ensuremath{N_{k,L}},
        description={Determinant line bundle $\operatorname{det}(E_{k,L})$,}
}
\newglossaryentry{Theta-k}
{
        name=\ensuremath{\Theta_k},
        description={Incidence subscheme $\{(x,\eta)\vert x \in \operatorname{Supp }\eta\}$ in $X\times \operatorname{Sym}^k(X)$,}
}
\newglossaryentry{con-x-y}
{
        name=\ensuremath{\mathcal{C}_{X/Y}},
        description={Conormal sheaf (of $X$ in $Y$),}
}
\newglossaryentry{gamma-k-L}
{
        name=\ensuremath{\gamma_{k}},
        description={Composition of homomorphisms $\mathcal{C}_{B^k/\mathbb{P}^r\times X^{[k]}}\lra i_k^*\Omega_{\mathbb{P}^r\times X^{[k]}}^1$ and $i_k^*\Omega_{\mathbb{P}^r\times X^{[k]}}^1 \lra \pi_k^*\Omega_{X^{[k]}}^1$,}
}
\newglossaryentry{i-k}
{
        name=\ensuremath{i_k},
        description={Morphism $B^k \lra \mathbb{P}H^0(X,L)\times X^{[k]}$,}
}
\newglossaryentry{Q-m-k}
{
        name=\ensuremath{Q_{m,k}},
        description={Relative cotangent sheaf of $r : X^{[m,k]}\lra X^{[k]}$,}
}
\newglossaryentry{i-m-k}
{
        name=\ensuremath{i_{m,k}},
        description={Morphism $B^{m,k} \lra P^{m,k}$,}
}
\newglossaryentry{gamma-m-k-L}
{
        name=\ensuremath{\gamma_{m,k}},
        description={Composition of homomorphisms $\mathcal{C}_{B^{m,k}/P^{m,k}}\lra i_{m,k}^*\Omega_{P^{m,k}}^1$ and $i_{m,k}^*\Omega_{P^{m,k}}^1 \lra \pi_{m,k}^*Q_{m,k}$,}
}
\newglossaryentry{z-k}
{
        name=\ensuremath{Z_k},
        description={Exceptional divisor of $\alpha_k$,}
}
\newglossaryentry{bl-j-k}
{
        name=\ensuremath{\operatorname{bl}_k ^j},
        description={Morphism $X(j,k)\lra X^j \times X^{[k]}$ blowing up the union of the universal families,}
}
\newglossaryentry{x-j-k}
{
        name=\ensuremath{X(j,k)},
        description={Blow up variety of $X^j\times X^{[k]}$ along the union of the universal families,}
}
\newglossaryentry{pi-m-k}
{
        name=\ensuremath{\pi_{m,k}},
        description={Morphism $B^{m,k}\lra X^{[m,k]}$,}
}
\newglossaryentry{dim-k}
{
        name=\ensuremath{d_k},
        description={dimension of $B^k$,}
}
\newglossaryentry{larger-relative-secant-bundle}
{
        name=\ensuremath{P^{m,k}},
        description={Relative secant bundle of $k$-secant planes over $X^{[m,k]}$,}
}
\newglossaryentry{p-m-k}
{
        name=\ensuremath{p_{m,k}},
        description={Structure morphism of $P^{m,k}$ as a projective bundle on $X^{[m,k]}$,}
}
\newglossaryentry{r-k}
{
        name=\ensuremath{r_k},
        description={Rank of $M_{k,L}$,}
}

\newglossaryentry{h-k}
{
        name=\ensuremath{H_k},
        description={Tautological divisor on $B^k$,}
}

\setglossarystyle{tree}

\makeglossaries

\begin{document}

\title[Singularities and syzygies of secant varieties]{Singularities and syzygies of secant varieties of\\ smooth projective varieties}

\begin{abstract}
We study the higher secant varieties of a smooth projective variety embedded in projective space. We prove that when the variety is a surface and the embedding line bundle is sufficiently positive, these varieties are normal with Du Bois singularities and the syzygies of their defining ideals are linear to the expected order. 
We show that the cohomology of the structure sheaf of the surface completely determines whether the singularities of its secant varieties are Cohen--Macaulay or rational.
We also prove analogous results when the dimension of the original variety is higher and the secant order is low, and by contrast we prove a result that strongly implies these statements do not generalize to higher dimensional varieties when the secant order is high. 
Finally, we deduce a complementary result characterizing the ideal of secant varieties of a surface in terms of the symbolic powers of the ideal of the surface itself, and we include a theorem concerning the weight one syzygies of an embedded surface --- analogous to the gonality conjecture for curves --- which we discovered as a natural application of our techniques. 
\end{abstract}

\maketitle

\vspace{-30pt}\section*{Introduction}

Let $X \subseteq \mathbb{P}^r$ be a linearly normal, non-degenerate embedding of a smooth projective variety of dimension $n$. We study both the singularities and the syzygies of the higher secant varieties of $X$. We mainly consider the case when $L := \OO_X(1)$ is sufficiently positive (see item (7) in the section ``Notation and Conventions" for the precise notion). Throughout the paper, we work over the field $\mathbb{C}$ of complex numbers.

Classically, the secant variety $\Sigma(X)$ of $X \subseteq \mathbb{P}^r$ is defined to be the closure of the union of all lines in $\mathbb{P}^r$ joining a distinct pair of points on $X$.  More generally the $k^{\text{th}}$ secant variety $\Sigma_k = \Sigma_k(X,L)$ of $X \subseteq \mathbb{P}^r$ is defined to be the closure of the union of all $k$-planes in $\mathbb{P}^r$ joining $k+1$ distinct points on $X$, so $X = \Sigma_0$ and $\Sigma(X) = \Sigma_1$. 
These higher secant varieties are easily seen to be irreducible and under relatively weak positivity assumptions their dimensions are as expected (see \cite{Griffiths.Harris.79} and Zak's Theorems \cite{Zak.93}), but a fundamental question is whether they have mild singularities. Ullery \cite{Ullery.16} proved that $\Sigma_1$ is normal when $L \cong \omega_X\otimes A^m \otimes P$ for $A$ very ample, $P$ nef, and $m \geq 2n+2$. With the same assumptions on $L$, Chou and Song \cite{Chou.Song.18} proved that $\Sigma_1$ has Du Bois singularities (a notion generalizing rational singularities) which are in fact rational precisely when $\OO_X$ has trivial higher cohomology. More recently Ein, Niu, and Park, the third author of the present paper, showed in \cite{Ein.Niu.Park.20} that if $X = C$ is a curve of genus $g$ and $\operatorname{deg}(L) \geq 2g + 2k + 1$, then $\Sigma_k$ has normal Du Bois singularities.

In this paper, we develop a new and robust framework to study higher secant varieties of higher dimensional varieties. Our results include asymptotic analogues of Ullery, Chou--Song and Ein--Niu--Park: that $\Sigma_k$ is normal with mild singularities when either $n$ or $k$ is small and $L$ is sufficiently positive. However our results point strongly towards a non-normality result when both $n$ and $k$ are large.

\begin{theoremalpha}\label{theoremA}\textnormal{(Singularity Theorem)}
Fix a positive integer $k$ and assume that either $n \leq 2$ or $k \leq 2$ and that $L$ is sufficiently positive. Then $\Sigma_k$ has normal Du Bois singularities which are
\begin{itemize}[topsep=0pt]
    \item Cohen--Macaulay if and only if $H^i(X,\OO_X) = 0$ for $1 \leq i \leq n-1$, and
    \item rational if and only if $H^i(X,\OO_X) = 0$ for $1\leq i \leq n$.
\end{itemize}
Moreover, $\Sigma_k\subseteq \mathbb{P}^r$ is projectively normal (in fact it is arithmetically Cohen--Macaulay if $\Sigma_k$ itself is Cohen--Macaulay) with regularity $\operatorname{reg}(\OO_{\Sigma_k})\leq (n+1)(k+1)$.
\end{theoremalpha}
The importance of understanding higher secant varieties derives both from algebra and from geometry. On the algebraic side, secant varieties (particularly of Grassmann, Segre and Veronese varieties) parametrize tensors of bounded rank and thus feature prominently in the tensor decomposition problem as well as algebraic statistics and algebraic complexity theory more broadly --- see \cite{Landsberg.Tensors,Landsberg.Complexity}. On the geometric side, secant varieties are simply defined yet fine invariants of embeddings and thus can be naturally expected to reflect much of the geometry of the original embedding $X \subseteq \mathbb{P}^r$. Among many other applications, secant varieties have been studied in relation to moduli problems (see \cite{Bertram.92, Thaddeus.94}) and flips (see \cite{Vermeire.01, Vermeire.02}).

The results of Theorem \ref{theoremA} in the case $n = 2$ warrant specific mention. When $X=S$ is a surface (and $L$ is sufficiently positive), $\Sigma_k$ is Cohen--Macaulay if and only if $q(X) = 0$, and $\Sigma_k$ has rational singularities if and only if $q(X)=p_g(X)= 0$ (here $q(X)$ and $p_g(X)$ denote the irregularity and geometric genus of $X$, respectively). In particular if $X$ is a rational surface, then $\Sigma_k$ has rational singularities. Yet $\Sigma_k$ may still have rational singularities when $X$ has nonnegative, or even positive, Kodaira dimension. On the other hand, as the question raised by Chou--Song \cite[Question 1.6]{Chou.Song.18} asks for $\Sigma_1$, it is natural to ask when $\Sigma_k$ has singularities of the sort studied in the minimal model program --- i.e. log terminal or log canonical. We show that if $n=2$ and $k \geq 9$ (respectively, $k \geq 10$), then $\Sigma_k$ does not have log terminal (respectively, log canonical) singularities --- see Proposition \ref{not-klt-lc} for details.

There has also been a great deal of interest in the equations defining secant varieties (see for example \cite{BB,BBF24,BK24,BGL}). Let $I_X \subseteq S := \mathbb{C}[x_0,\ldots,x_r]$ denote the homogeneous ideal of $X \subseteq \mathbb{P}^r$. Mumford \cite{Mumford.70} proved that $X \subseteq \mathbb{P}^r$ is projectively normal and the homogeneous ideal $I_X$ is generated by quadrics whenever $L$ is sufficiently positive. Green \cite{Green1, Green2} realized that these and similar classical results on equations defining algebraic varieties should be seen as the first cases of a much more general picture involving higher syzygies of $I_X$. The guiding principle is that as the positivity of the embedding line bundle $L$ grows, the first few steps of the minimal graded free resolution
\begin{center}
    \begin{tikzcd}
        \cdots \ar[r] & E_1 \ar[r] & E_0 \ar[r] & S/I_X \ar[r] & 0
    \end{tikzcd}
\end{center}
should become as simple as possible. Specifically, if $E_0 = S$ and $E_i$ is generated in degree $\leq i+1$ for $1\leq i \leq p$ then (the coordinate ring of) $X$ is said to have \emph{linear syzygies to order $p$}, or to have \emph{property $N_p$}. In this way, having property $N_0$ means $X \subseteq \mathbb{P}^r$ is projectively normal, $N_1$ means we have $N_0$ and $I_X$ is generated by quadrics, $N_2$ means we have $N_1$ and the relations between the quadrics are linear, and so on.\\
\\
\indent We expect $X$ to have property $N_p$ for higher values of $p$ as the positivity of $L$ grows. However, it is well-known that forms of degree $\leq k+1$ cannot cut out the $k^{\text{th}}$ secant variety $\Sigma_k$ in $\mathbb{P}^r$ whenever $\Sigma_k \neq \mathbb{P}^r$. The best we can hope for the resolution of $I_{\Sigma_k}$ is that it is generated by degree $k+2$ forms (that is, $E_1$ is generated in degree $\leq k+2$) and that it has linear syzygies for a few steps thereafter (that is, $E_i$ is generated in degree $\leq (i-1)+k+2$). In general we say that a variety whose homogeneous ideal has a resolution like this up to the $p^{\text{th}}$ step satisfies property $N_{k+2,p}$ (see \cite{eisenbud.green.hulek.popescu.05}). Along these lines, Sidman--Vermeire \cite{Sidman.Vermeire.09} conjectured that $\Sigma_k$ satisfies property $N_{k+2,p}$ when $X=C$ is a curve and $L$ has sufficiently high degree, and Ein--Niu--Park \cite[Theorem 1.2]{Ein.Niu.Park.20} verified this conjecture. Eisenbud also speculated \cite[Question 1.5]{BGL} that $\Sigma_k$ satisfies property $N_{k+2,p}$ in general.

Our second main result gives an asymptotic analogue of that of Ein--Niu--Park mentioned above, concerning syzygies of secant varieties. Together with Theorem \ref{theoremA}, this gives a partial answer to both \cite[Question 1.5]{BGL} and \cite[Problem 6.3]{Ein.Niu.Park.20}. 

\begin{theoremalpha}\label{thm:N_{k+2,p}}\textnormal{(Syzygy Theorem)}
Fix integers $k \geq 1$ and $p\geq 0$ and assume either $n \leq 2$ or $k \leq 2$. If $L$ is sufficiently positive relative to $p$, then $\Sigma_k \subseteq \mathbb{P}^r$ satisfies property $N_{k+2,p}$.
\end{theoremalpha}

As a direct consequence of Theorem \ref{thm:N_{k+2,p}} and a result of Sidman--Sullivant \cite[Theorem 1.2]{Sidman.Sullivant}, we obtain a description of the generators of the homogeneous ideal $I_{\Sigma_k}$ in terms of the symbolic powers of the homogeneous ideal $I_X$. Recall that the $t^{\text{th}}$ symbolic power $I_X^{(t)}$ coincides, by a well-known theorem of Nagata and Zariski (\cite[(38.3) Theorem]{Nagata.62} and \cite[Main Lemma, p.189]{Zariski.49}, respectively), with the ideal of all homogeneous polynomials on $\mathbb{P}^r$ vanishing to order at least $t$ at a general point of $X$.

\begin{corollaryalpha}\label{cor:idealofsecants}\textnormal{(Ideal Generators)}
Fix a positive integer $k$ and assume that either $n\leq 2$ or $k\leq 2$ and that $L$ is sufficiently positive. Then $I_{\Sigma_k}$ is generated by the $(k+2)^{\text{th}}$ graded piece of $I_X^{(k+1)}$.
\end{corollaryalpha}

One of the fundamental ideas guiding our study of the higher secant varieties is the secant bundle construction. When $L$ is $k$-very ample (that is, the restriction map $H^0(X, L) \lra H^0(\xi, L|_{\xi})$ is surjective for every length $k+1$ zero-dimensional subscheme of $X$), the $k^{\text{th}}$ secant variety $\Sigma_k$ of $X \subseteq \mathbb{P}^r$ is the union of the projective spans of smoothable length $k+1$ zero-dimensional subschemes of $X$. This indicates that $\Sigma_k$ is dominated by a projective bundle on the so-called \emph{smoothable} (or \emph{principal}) \emph{component} $X_{\operatorname{sm}}^{[k+1]}$ of the \emph{Hilbert scheme of $k+1$ points} $X^{[k+1]}$ of $X$. Because of the unfortunate index discrepancy between these basic objects, we modify our notation slightly --- we define
$$
\sigma_k \coloneqq \Sigma_{k-1}.
$$
Then the projective bundle mentioned above (for the case of $\sigma_k$ now) is a bundle on $X^{[k]}_{\operatorname{sm}}$ which we will denote by $B_{\operatorname{sm}}^k$ and refer to as the \emph{$k$-secant bundle}. 
We have the following diagram
\begin{center}
    \begin{tikzcd}
        B_{\operatorname{sm}}^k \ar[r,"\alpha"]\ar[d, "\pi"']& \sigma_k\\
        X_{\operatorname{sm}}^{[k]}&
    \end{tikzcd}
\end{center}
where $\pi \colon B_{\operatorname{sm}}^k \lra X^{[k]}_{\operatorname{sm}}$ is the structure morphism of the bundle and $\alpha \colon B_{\operatorname{sm}}^k \lra \sigma_k \subseteq \mathbb{P}^r$ is the surjective morphism given by $\lvert \mathcal{O}_{B_{\operatorname{sm}}^k}(1) \rvert$. With this perspective, we have now reduced the study of $\sigma_k$ to the study of $X^{[k]}_{\operatorname{sm}}$, a vector bundle $E_{k,L}$ on it (the \emph{tautological bundle associated to $L$}, for which $B_{\operatorname{sm}}^k = \mathbb{P}_{X^{[k]}_{\operatorname{sm}}}(E_{k,L})$), and the map $\alpha$. When $n\leq 2$ or $k\leq 3$ the scheme $X^{[k]}_{\operatorname{sm}}$ is smooth (and coincides with $X^{[k]}$), hence in these cases $\alpha$ is a resolution of singularities for $\sigma_k$.\\
\\
\indent In the case that $X = C$ is a curve, Bertram \cite[Sections 1 and 2]{Bertram.92} undertook an in-depth study of the map $\alpha$ and its behavior relative to the natural stratification
\begin{equation}\label{eq:stratificationintro}
X = \sigma_1 \subseteq \sigma_2 \subseteq \cdots \subseteq \sigma_{k-1} \subseteq \sigma_k \subseteq \mathbb{P}^r.\tag{$\star$}
\end{equation}
In particular, he sharpened the classical Terracini lemma (concerning general tangent planes of $\sigma_k$) and introduced a relative version of the same, useful for comparing different strata in ($\star$). Ein--Niu--Park's work \cite{Ein.Niu.Park.20} is based on Bertram's results. 

Generalizing Bertram's work to higher dimensions was the first obstacle in the study of higher secant varieties of higher dimensional smooth projective varieties, as was pointed out in \cite[page 664]{Ein.Niu.Park.20}. Here we successfully generalize the precise Terracini lemma \cite[Lemma 1.4]{Bertram.92} via a careful study of the cotangent sheaf of $X^{[k]}$, which then leads us to interpret the differential of $\alpha$ in terms of a natural pairing on Artinian Gorenstein algebras. 
One of the consequences is the following embedding theorem for the open subset $U^k \coloneqq B^k\setminus \alpha^{-1}(\kappa_{k-1} \cap \sigma_k)$, where $\kappa_k \subseteq \mathbb{P}^r$ is the \emph{$k$-cactus variety}  --- that is, the union of all $k$-secant $(k-1)$-planes (see \cite{BB} and \S \ref{subsec:secant bundles}). Notice that we do not impose any conditions on $n$ and $k$ and we only assume that $L$ is $(2k-1)$-very ample.

\begin{theoremalpha}\label{embedding-theorem}\textnormal{(Embedding Theorem)}
Fix a positive integer $k$ and suppose that $L$ is $(2k-1)$-very ample on $X$. Then the morphism
$$
\alpha\vert_{U^k} \colon U^k \lra \mathbb{P}^r \setminus \kappa_{k-1}
$$
is a closed embedding, i.e., $U^k \cong \sigma_k\setminus \kappa_{k-1}$. 
\end{theoremalpha}

A striking consequence of this theorem is that it gives a description of the singular locus of the secant variety $\sigma_k$ when $L$ is $(2k-1)$-very ample. It has long been clear that $\sigma_{k-1} \subseteq \operatorname{Sing} (\sigma_k)$, so the key issue is whether there are ``extra" singularities beyond those in $\sigma_{k-1}$. The extra singularities would live in the open stratum $\sigma_k \setminus \sigma_{k-1}$ of (\ref{eq:stratificationintro}), which contains $U^k$ by the Embedding Theorem. But $U^k$ lies over the Gorenstein locus in $X^{[k]}_{\operatorname{sm}}$ (see \S\ref{distinguished-subschemes}), that is,
$$
U^k \subseteq \pi^{-1}(X^{[k]}_{\operatorname{sm}, \operatorname{Gor}}).
$$
Thus we see that $U^k$ has the same type of singularities as $X^{[k]}_{\operatorname{sm}, \operatorname{Gor}}$. In particular, if $X^{[k]}_{\operatorname{sm}, \operatorname{Gor}}$ has non-normal singularities, then so does $\sigma_k$.\footnote{We are not aware of any such example. It is even unknown whether $X^{[k]}_{\sm}$ may have non-normal singularities. The best results currently known in this direction are Kyungyong Lee's result \cite{Lee.10} that $X^{[k]}_{\sm}$ may not be Cohen--Macaulay and Joachim Jelisiejew's result \cite{Jelisiejew.20} that $X^{[k]}$ satisfies Murphy's law, i.e., $X^{[k]}$ may have arbitrarily bad singularities.} On the other hand, by the series of work \cite{Casnati.Notari.07, Casnati.Notari.09, Casnati.Jelisiejew.Notari.15}, 
if $k \leq 13$ or $k=14$ and $\dim(X) \leq 5$, then $X^{[k]}_{\Gor}$ is irreducible, and $X^{[k]}_{\textnormal{sm,Gor}}$ is smooth if and only if $n \leq 3$ or $k \leq 5$ (see \S\ref{subsubsection-smoothness}). Taking these observations and results into account, we arrive at the following consequence of the Embedding Theorem.\footnote{While we were writing the present paper, we learned from Jaros{\l}aw Buczy\'{n}ski that he had obtained an asymptotic analogue of this result with different methods from a series of joint work with Hanieh Keneshlou, Weronika Buczy\'{n}ska, and {\L}ucja Farnik (cf. \cite{BBF24}, \cite{BK24}).}

\begin{corollaryalpha}\label{cor:singularlocusofSigma}\textnormal{(Singular Locus)}
Fix a positive integer $k$ and suppose that $L$ is $(2k-1)$-very ample on $X$. Then we have inclusions
$$
 \sigma_{k-1}\cup \alpha(\pi^{-1}(\operatorname{Sing}(X_{\operatorname{sm},\operatorname{Gor}}^{[k]}))) \subseteq \operatorname{Sing}(\sigma_k) \subseteq  (\kappa_{k-1} \cap \sigma_k)\cup \alpha(\pi^{-1}(\operatorname{Sing}(X_{\operatorname{sm},\operatorname{Gor}}^{[k]}))).
$$
In particular, the following hold:
\begin{enumerate}[topsep=0pt]
    \item If $k \leq 14$ or $k=15$ and $\dim(X) \leq 5$, then $\operatorname{Sing}(\sigma_k)= \sigma_{k-1}\cup \alpha(\pi^{-1}(\operatorname{Sing}(X_{\operatorname{sm},\operatorname{Gor}}^{[k]})))$.
    \item $\operatorname{Sing}(\sigma_k) = \sigma_{k-1}$ if and only if $n \leq 3$ or $k \leq 5$.
\end{enumerate}
\end{corollaryalpha}

\noindent There are two points to be made here. First, the $(2k-1)$-very ampleness condition for $L$ in Theorem \ref{embedding-theorem} and Corollary \ref{cor:singularlocusofSigma} is sharp --- for any $n \geq 1$ and $k \geq 2$, there are choices of $X$ and $L$ for which $L$ is $(2k-2)$-very ample but $\sigma_k\setminus\sigma_{k-1}$ has singularities, indeed non-normal ones (see \S\ref{subsection-not-normal}). We also refer to \cite[Theorem 2.1]{Han18} and \cite[Example 6]{FH21} for other examples. Second, in the setting of Theorem \ref{theoremA}, we have $\operatorname{Sing}(\sigma_k) = \sigma_{k-1}$ --- hence in this case all points of $\sigma_k\setminus \sigma_{k-1}$ are determined locally analytically by $n = \dim(X)$ only, regardless of the geometry of $X$, while by contrast the singularities of $\sigma_k$ along $\sigma_{k-1}$ actually reflect the geometry of $X$. For instance, if $X$ is an Enriques surface, a K3 surface, or an abelian surface, then $\sigma_k$ has rational singularities, Cohen-Macaulay singularities, or non-Cohen-Macaulay singularities, respectively.

Now, we briefly outline how Theorems \ref{theoremA} and \ref{thm:N_{k+2,p}} are proven. We assume $n \leq 2$ or $k \leq 3$ (note the change in $k$ bound because we switched to $\sigma_k$ notation). When $n=1$, based on Bertram's careful analysis of the fibers of $\alpha$ \cite{Bertram.92}, Ein--Niu--Park \cite{Ein.Niu.Park.20} apply the theorem on formal functions to reduce the normality of $\sigma_k$ to the projective normality of $\sigma_m$ for $1 \leq m \leq k-1$. The theorem on formal functions is also a key tool for the case of $k=2$ as well in \cite{Chou.Song.18}, \cite{Ullery.16}. While we have been directly motivated by these approaches, our results in higher dimensions require a significantly more robust framework as well as a new strategy for the following important reason. The relative Terracini lemma (Theorem \ref{relative-terracini} -- see also Proposition \ref{fiber-product}) is needed to compare strata $\sigma_m \subseteq \sigma_k$ and for this one needs to introduce the nested Hilbert scheme $X^{[m,k]}$, together with its forgetful maps $X^{[m,k]} \lra X^{[m]}$ and $X^{[m,k]} \lra X^{[k]}$ (see \S \ref{nestedHilb}). In the case $n=1$ or $k=2$, these forgetful maps are, respectively, smooth and finite. However neither property continues to hold when $n \geq 2$ and $k\geq 3$. 
Therefore, the fibers of $\alpha$ along lower strata of ($\star$) are not smooth.
Moreover, apart from when $X$ is a curve, the nested Hilbert scheme $X^{[m,k]}$ itself is smooth essentially only when $n = 2$ and $m = k-1$ (see \S\ref{hilbert-key-facts} for the remaining isolated cases). When $X^{[m,k]}$ is not smooth, the preimages of lower strata of ($\star$) under the resolution $\alpha$ are not smooth.
Taken together, these facts suggest that one cannot hope to apply the theorem on formal functions in a useful way.

Instead, we take a more cohomological approach --- what follows is an outline of the technique we use to transform problems concerning a higher secant variety into problems concerning the corresponding Hilbert scheme of points, as well as quick alternative proofs of the main results of \cite{Ein.Niu.Park.20} and \cite{Ullery.16}.
At the heart of the issue is the normality of $\sigma_k$ --- to establish this, it suffices to prove that the natural injective map
\begin{equation*}
\begin{tikzcd}[row sep = tiny]
\mathcal{O}_{\sigma_k} \ar[hookrightarrow,r] & \alpha_* \mathcal{O}_{B^k}
\end{tikzcd}
\end{equation*}
is an isomorphism. We derive both the normality of $\sigma_k$ and the projective normality of $\sigma_k \subseteq \mathbb{P}^r$ at once by demonstrating surjectivity of the natural maps
$$
S^t H^0(X, L) \lra H^0(B^k, \mathcal{O}_{B^k}(t))
$$ for all $t$. Here $H^0(X, L) \cong H^0(\sigma_k, \mathcal{O}_{\sigma_k}(1)) \cong H^0(B^k, \mathcal{O}_{B^k}(1))$.
This being the case, we let $Z_k := B^k\setminus U^k$ and consider the short exact sequence
$$
0 \lra \mathcal{O}_{B^k}(-Z_k) \lra \mathcal{O}_{B^k} \lra \mathcal{O}_{Z_k} \lra 0.
$$
By a careful study of the divisor $Z_k$ (see \S\ref{subsection-boundary-divisor}) we find that
$$
\OO_{B^k}(-Z_k) \cong \OO_{B^k}(-k)\otimes \pi^*A_{k,L},
$$
where $A_{k,L}$ is a specific line bundle on $X^{[k]}$ (see \S\ref{subsubsec-line-bundles}). Then via an induction on $k$ and our generalization of Bertram's study of the map $\alpha$, we find that
$$
\alpha_*\OO_{Z_k} \cong \OO_{\sigma_{k-1}}
$$
which, together with the previous isomorphism and a theorem of Danila (Theorem \ref{Danila}), allows us to reduce the problem to surjectivity of the maps
$$
S^\ell H^0(X, L) \otimes H^0(X^{[k]}, A_{k, L}) \lra H^0(X^{[k]}, S^\ell E_{k, L} \otimes A_{k, L})
$$
for all $\ell$. This can be deduced from the $i=j$ cases of our main vanishing theorem by considering the terms in the Eagon--Northcott complex associated to the evaluation map $H^0(X, L) \otimes \mathcal{O}_{X^{[k]}} \to E_{k, L}$, whose kernel bundle is denoted by $M_{k,L}$.

\begin{theoremalpha}\label{intro-vanishing}\textnormal{(Main Vanishing Theorem)}
Let $n$ and $k$ be as in Theorem \ref{theoremA} and fix integers $i \geq 1$ and $j \geq 0$. If $L$ is sufficiently positive relative to $i$ and $j$, then
$$
H^i(X^{[k]}, \wedge^jM_{k,L}\otimes A_{k,L}) = 0.
$$
\end{theoremalpha}

To prove Theorem \ref{thm:N_{k+2,p}}, we need to convert the cohomology computation on $\sigma_k$ into a cohomology computation on $X^{[k]}$. However, note that $\sigma_k$ may not have rational singularities, as shown in Theorem \ref{theoremA}; that is, $R^i \alpha_* \mathcal{O}_{B^k} \neq 0$ in general. This causes serious difficulties, which can be deflected by proving the \emph{Du Bois-type condition}:
$$
R^i\alpha_* \mathcal{O}_{B^k}(-Z_k) = \begin{cases} \mathcal{I}_{\sigma_{k-1}/\sigma_k} & \text{if $i=0$} \\
0 & \text{if $i > 0$}
\end{cases}.
$$
This key condition allows us to deduce property $N_{k+1,p}$ for $\sigma_k \subseteq \mathbb{P}^r$ fairly quickly from the main vanishing theorem. In the cases $n=1$ or $k=2$, the Du Bois-type condition was also shown by the formal function theorem in \cite{Chou.Song.18} and \cite{Ein.Niu.Park.20}. However, we directly derive it from the main vanishing theorem. As the name suggests, this condition also implies that $\sigma_k$ has Du Bois singularities.

The central idea to establish the main vanishing theorem is to reduce it to an analogous statement on the symmetric product $\operatorname{Sym}^k(X)$ and then apply Serre--Fujita vanishing. As we are working in characteristic zero, it is sufficient to show that  
\begin{equation}\label{eq:desiredvanintro}\tag{$\sharp$}
H^i(X^{[k]}, M_{k, L}^{\otimes j} \otimes A_{k, L})=0~~\text{ for $i>0$ and $j \geq 0$}.
\end{equation}
Let $\mathcal{Z}_{k}$ be the universal family over $X^{[k]}$ with ideal sheaf $\mathcal{I}_{\mathcal{Z}_{k}}:=\mathcal{I}_{\mathcal{Z}_{k}/X \times X^{[k]}}$. On $X^j\times X^{[k]}$ we let $\operatorname{pr}_i$ denote the $i^{\text{th}}$ projection to $X\times X^{[k]}$ and $\operatorname{pr}_{X^{[k]}}$ the projection to $X^{[k]}$. We then define a sheaf
$$
\mathcal{F} := \operatorname{pr}_1^*\mathcal{I}_{\mathcal{Z}_k}\otimes \cdots \otimes \operatorname{pr}_j^*\mathcal{I}_{\mathcal{Z}_k}\otimes (L^{\boxtimes j}\boxtimes A_{k,L})
$$
and find that
$$
R^i \operatorname{pr}_{X^{[k]},*}\mathcal{F}=\begin{cases} M_{k,L}^{\otimes j} \otimes A_{k, L} & \text{if $i=0$} \\ 0 & \text{if $i>0$}\end{cases}.
$$
Then (\ref{eq:desiredvanintro}) is equivalent to
\begin{equation}\tag{$\flat$}
H^i(X^j \times X^{[k]}, \mathcal{F}) = 0~~\text{ for $i>0$ and $j \geq 0$}.
\end{equation}
When working with $X = C$ a curve, the Hilbert scheme of points coincides with the symmetric product. In this case, $A_{k, L}$ is a sufficiently positive line bundle, so the desired vanishing (\ref{eq:desiredvanintro}) holds by Serre--Fujita vanishing. When $n \geq 2$, however, this is no longer true, and therefore, a number of results that leverage the quotient structure of $X^{[k]}$ in the curve case are no longer available. Looking at the problem from another angle, there is a well-known method to prove ($\flat$) for $k=1$ as well: one may blow up the union of the universal families (in this case, diagonals) and simply use the Kawamata--Viehweg vanishing theorem. 
This method fails for $k\geq 2$ for two important reasons. First, the blow-up is singular in general. Second, $\mathcal{F}$ is not globally generated, even when $L$ is sufficiently positive. 
What is more, the desired vanishing (\ref{eq:desiredvanintro}) does not hold in general even though $L$ is sufficiently positive. This is perhaps the most crucial difficulty in the higher dimensional case -- see Remark \ref{remark:hard-problem} for a more detailed discussion of these issues. Instead, considering the \emph{Hilbert--Chow morphism} $h \colon X^{[k]} \to \operatorname{Sym}^k(X)$, we show that
$$
R^i h_* (\wedge^j M_{k, L} \otimes A_{k, L}) = 0~~\text{ for $i>0$},
$$
which in turn is obtained by considering the terms of the form $\bigoplus S^{\ell} E_{k, L}^{\vee} \otimes \delta_{k}^{-1}$ in the resolution of $\wedge^j M_{k, L} \otimes A_{k, L}$ and proving the relative vanishing statement for those terms via a rather intricate induction on $k$. This allows us to apply Serre--Fujita vanishing on the symmetric product $\operatorname{Sym}^k(X)$ and finishes the proof of the main vanishing theorem. Given the extent to which the geometry of the Hilbert scheme of points controls the geometry of secant varieties, we believe that the subtleties highlighted above are an inherent and fundamental aspect of higher dimensions, and mark a sharp contrast with the case of curves.

On a different note, we point out here that for $k=2$ we are able to prove an effective version of Theorem \ref{intro-vanishing} by leveraging the fact that the triple nested Hilbert scheme $X^{[1,2,3]}$ has canonical singularities. 

\begin{theoremalpha}\label{intro-effective}\textnormal{(Effective Projective Normality)}
Set $L:=\omega_X \otimes A^m \otimes B$, where $A$ is very ample and $B$ is nef. If $m \geq 4n$, then $\sigma_2 \subseteq \mathbb{P}^r$ is projectively normal and $H^i(\sigma_2, \OO_{\sigma_2}(\ell))=0$ for $i \geq 1$ and $\ell \geq 1$.
\end{theoremalpha}

The ideas from the proof of the main vanishing theorem lead to a surprisingly quick proof of the following theorem on weight-one syzygies of algebraic surfaces (see item (13) in the section ``Notation and Conventions" for the definition of the Koszul cohomology groups $K_{p,1}(S,B;L)$).

\begin{theoremalpha}\label{thm:weight-one-syzygies}\textnormal{(Weight One Syzygy Vanishing)}
Let $S$ be a smooth projective surface, let $B$ and $L$ be line bundles on $S$, and fix integers $k\geq 1$ and $p \geq 0$. If $L$ is sufficiently positive relative to $p$, then
$$
B \text{ is $p$-very ample } \iff K_{p,1}(S,B;L) = 0.
$$
\end{theoremalpha}

For example, the $p = 0$ case of this theorem says that $B$ is globally generated if and only if the multiplication map \[\mu_{B,L} \colon  H^0(S,B)\otimes H^0(S,L) \lra H^0(S,B\otimes L)\] is surjective, and the $p = 1$ case says that $B$ is very ample if and only if the natural map 
\[H^0(S,B)\otimes \wedge^2H^0(S,L) \lra \operatorname{ker}(\mu_{B\otimes L,L})\] 
is surjective. We note that Theorem \ref{thm:weight-one-syzygies} for the cases $0 \leq p \leq 3$ was previously shown by Agostini \cite[Theorem B]{Agostini}.

One can view this theorem as a natural generalization of the Green--Lazarsfeld gonality conjecture for algebraic curves \cite[\S3.7, Conjecture]{green.lazarsfeld.86} to algebraic surfaces --- here follows a short explanation of this perspective. The gonality conjecture predicts that if $L$ is a sufficiently positive line bundle on a smooth projective curve $C$ of genus $g$, then $K_{p,1}(C;L)=0$ for $p \geq h^0(C, L)-\operatorname{gon}(C)$. By the duality theorem for Koszul cohomology \cite[Theorem 2.c.6]{Green1} we have
$$
K_{p,1}(C;L) \cong K_{r-p,1}(C,\omega_C;L)^{\vee},
$$
where $r = h^0(C,L) - 1$. More generally, Ein and Lazarsfeld proved in \cite{Ein.Lazarsfeld.Gonality} that a line bundle $B$ on $C$ is $p$-very ample if and only if $K_{p,1}(C, B; L)=0$ (see also \cite{Rathmann}, \cite{Niu.Park}), hence in particular resolving the conjecture (since $\operatorname{gon}(C) \geq p + 2$ if and only if $\omega_C$ is $p$-very ample).

A vast body of work examining questions related to the asymptotic behavior of Koszul cohomology (particularly including \cite{Ein.Lazarsfeld.Asymptotic, Yang, Ein.Lazarsfeld.Yang, Ein.Lazarsfeld.survey,Agostini,Park1, Park2}) suggested that vanishing of $K_{p,1}(X,B;L)$ for higher dimensional $X$ would be a priori stronger than $p$-very ampleness of $B$. Most notably, Ein--Lazarsfeld--Yang showed in \cite{Ein.Lazarsfeld.Yang} that $p$-jet-very ampleness of $B$ implied the vanishing and they suspected that the two conditions would be equivalent. What we have shown though, perhaps surprisingly, is that at least in the surface case it is still merely $p$-very ampleness that is equivalent to the vanishing, and $p$-jet-very ampleness is a strictly stronger notion.
We refer to Subsection \ref{subsection-weight-one-syzygies}, and in particular Remark \ref{rmk:detailed-discussion-weight-one}, for a more comprehensive discussion.

The structure of the paper is as follows: in \S\ref{section-hilbert-scheme} we establish notation and key facts concerning the geometry of Hilbert schemes and nested Hilbert schemes. In \S\ref{section-secant-sheaves} we introduce the secant variety together with certain relevant vector bundles and projective bundles associated to it on the Hilbert scheme. In this section we also conduct an important study of the boundary divisor $Z_k \subseteq B^k$ to be used in the proof of our Singularity Theorem. In \S\ref{section-embedding-theorems} we outline some local algebraic results and use these to prove our Embedding Theorem (Theorem \ref{terracini}). Then in \S\ref{section-cohomology-vanishing} we establish some preliminary lemmas before proving the Main Vanishing Theorem (Theorem \ref{main-vanishing-serre}). In \S\ref{main-results-section} we proceed to prove our Singularity Theorem in the following order: we begin with an inductive proof of normality and projective normality of $\sigma_k$ (Theorem \ref{thm:normalsing}), then proceed to prove that the singularities are Du Bois (Theorem \ref{thm:dubois}), then that they are Cohen--Macaulay (Theorem \ref{thm:CM}) when there is sufficient vanishing (from which a bit more work shows that the embedding $\sigma_k\subseteq \mathbb{P}^r$ is arithmetically Cohen--Macaulay --- Theorem \ref{thm:arithmetically-CM}), then that they are rational (Theorem \ref{thm:rational-sing}) given a little more vanishing. In this section we also include some examples demonstrating the sharpness of these results. In Section \S\ref{section-syzygies} we prove our Syzygy and Weight One Vanishing Theorems (Theorems \ref{thm:syz} and \ref{thm:wt1syz}, respectively).  Finally in \S\ref{section-questions} we pose a number of natural questions that arose in the preparation of this work and which we hope will provide both further context for our results as well as an indication of important future research directions. For the sake of completeness, we have also included an appendix in which we give a detailed construction of the cotangent sheaf on nested Quot schemes which, though the result itself (Theorem \ref{cotan-nested}) is used in the paper (see \S\ref{hilbert-key-facts}), is otherwise entirely extraneous to our work. 

\noindent \textbf{Acknowledgements.}
We would like to thank Jaros\l{}aw Buczy\'{n}ski both
for sharing with us his manuscript \cite{BK24} and for providing detailed and valuable feedback on the present work. We have also benefitted from discussions with Joachim Jelisiejew concerning normality of secant varieties and we are grateful for those and for his comments on our manuscript. Finally, we thank Daniele Agostini, Lawrence Ein, Kangjin Han, J.M. Landsberg, Robert Lazarsfeld, Wenbo Niu, Claudiu Raicu, and Lei Song for valuable discussions. The first author has been supported by the Institute for Basic Science, grant no. IBS-R032-D1. The second author has been supported by the Simons Investigator in Mathematics grant, no. $10015267$.

\section*{Notation and Conventions}\label{notation-and-conventions}

\begin{enumerate}[topsep=0pt]
    \item Throughout the paper, we work over the field $\mathbb{C}$ of complex numbers, and $X$ denotes a smooth projective variety of dimension 
    $n$.
    \item By \emph{scheme} we shall mean a separated scheme of finite type. A reduced and irreducible scheme will be a \emph{variety}.
    \item When $V$ is a vector space or an algebraic coherent sheaf on a scheme $Y$, we denote by $S^kV$ the $k^{\text{th}}$ \textit{symmetric power} of $V$, and by $\operatorname{Sym }V := \bigoplus_k S^kV$ the symmetric algebra on $V$.
    \item When $V$ is a vector space or an algebraic coherent sheaf on a scheme $Y$, we denote by $\mathbb{P}V = \mathbb{P}_YV$ the projective space or bundle of one-dimensional \textit{quotients} of $V$. That is, $\mathbb{P}V := \textnormal{Proj}_Y(\operatorname{Sym }V)$.
    \item For a coherent sheaf $\mathcal{F}$ on a scheme $Y$, the \emph{support of $\mathcal{F}$} will mean the subscheme $W \subseteq Y$ defined by the annihilator ideal of $\mathcal{F}$
    $$
    \mathcal{I}_W := \operatorname{Ann}(\mathcal{F}).
    $$
    \item If $p_1,\ldots,p_k \in Y$ are distinct closed points of a scheme $Y$, then we will say that an effective $0$-cycle $Z := \sum_in_ip_i$ (for $n_i \in \mathbb{N}$) has \emph{support}
    $$
    \operatorname{Supp}(Z) := \bigcup_i\{p_i\}.
    $$
    In other words the support $W$ is the scheme cut out by the ideal
    $$
    \mathcal{I}_W := \mathfrak{m}_{p_1}\cap\cdots\cap \mathfrak{m}_{p_k}.
    $$
    \item If $Y$ is a projective scheme, $\mathcal{L}$ a line bundle on $Y$ and $\mathcal{P}$ a property of $Y$ depending on $\mathcal{L}$ (and possibly on a further parameter $t$), we will say that $\mathcal{P}$ holds \emph{when $\mathcal{L}$ is sufficiently positive (relative to $t$)} if there is a line bundle $\mathcal{L}_t$ such that $\mathcal{P}$ holds whenever $\mathcal{L}\otimes \mathcal{L}_t^{-1}$ is ample. This definition is adapted from \cite[Definition 3.1]{Green2}, where it is referred to as ``sufficiently ample''.
    
    \item On a quasi-projective scheme $Y$ we will denote by $\omega_Y^{\bullet}$ the \emph{dualizing complex} of $Y$. If $Y$ is Cohen-Macaulay (so that $\omega_Y^{\bullet}$ is quasi-isomorphic to a shift of a coherent sheaf) we will simply write $\omega_Y$ and call it the \emph{dualizing sheaf} (or the \emph{canonical bundle} if $Y$ is Gorenstein). If $Y$ is normal then we will denote by $K_Y$ any canonical divisor.

    \item Let $Y$ be a normal variety, and $f \colon Z \lra Y$ a resolution of singularities. We say that $Y$ has \emph{rational singularities} if $R^i f_* \mathcal{O}_{Z}=0$ for $i\geq 1$. It is equivalent that $Y$ is Cohen--Macaulay and $f_* \omega_{Z} = \omega_Y$ (see \cite[Theorem 5.10]{Kollar-Mori}).

    \item Let $Y$ be a variety (not necessarily normal), and $\underline{\Omega}_Y^{\bullet}$ the Deligne--Du Bois complex (which is a generalization of the de Rham complex for smooth varieties). We say that $Y$ has \emph{Du Bois singularities} if the natural map $\mathcal{O}_Y \lra \underline{\Omega}_Y^0$ is a quasi-isomorphism. We refer to \cite[Chapter 6]{Kollar} for details.

    \item Let $Y$ be a normal projective variety, and $\Gamma$ an effective $\mathbb{Q}$-divisor on $Y$ such that $K_Y + \Gamma$ is $\mathbb{Q}$-Cartier. We say that $(Y,\Gamma)$ has \emph{Kawamata log terminal} (or \emph{klt}) singularities if $\lfloor \Gamma\rfloor=0$ and for all birational morphisms $f \colon Z \lra Y$ 
    we have
    $$
    K_Z + f_*^{-1}\Gamma \sim_{\text{lin}} f^*(K_Y + \Gamma) + \sum_{E \subseteq Z \textnormal{ exceptional}}a(E; Y,\Gamma)E
    $$
    with rational numbers $a(E;Y,\Gamma) > -1$. We say that $(Y,\Gamma)$ has \emph{log canonical} (or \emph{lc}) singularities if for all $f$ as above we have $a(E; Y,\Gamma)\geq -1$. We refer to \cite{Kollar-Mori} and \cite{Kollar} for details.
    
    \item Let $Y$ be a projective variety, $\mathcal{F}$ a coherent sheaf on $Y$, and $\mathcal{L}$ a line bundle on $Y$. The \emph{section module of $\mathcal{F}$ relative to $\mathcal{L}$} is defined as $R(Y,\mathcal{F}; \mathcal{L}):=\bigoplus_{m \geq 0}H^0(Y,\mathcal{F} \otimes \mathcal{L}^m)$. When $\mathcal{L}$ is ample, then $R(Y,\mathcal{F}; \mathcal{L})$ is a finitely generated graded module over the symmetric algebra $S := \operatorname{Sym} H^0(Y,\mathcal{L})$. 
    
    \item By Hilbert's syzygy theorem, if $\mathcal{L}$ is very ample then $R(Y,\mathcal{F}; \mathcal{L})$ admits a finite minimal graded free resolution
    $$
    0 \lra E_r \lra \cdots \lra E_2 \lra E_1 \lra E_0 \lra R(Y, V; \mathcal{L}) \lra 0
    $$
    over $S$, where $r := h^0(Y,\mathcal{L})-1$ and
    $$
    E_p\cong\bigoplus_{q \in \mathbb{Z}} K_{p,q}(Y, \mathcal{F}; \mathcal{L}) \otimes_{\mathbb{C}} S(-p-q).
    $$
    Here the $K_{p,q}(Y, \mathcal{F}; \mathcal{L})$ terms are the \emph{Koszul cohomology groups} which can be regarded as the spaces of \emph{$p^{\text{th}}$ syzygies of weight $q$}. 
    
    \item Let $\mathcal{L}$ be a basepoint free line bundle and let $M_{\mathcal{L}}$ be the kernel bundle of the evaluation map $H^0(Y, \mathcal{L}) \otimes \mathcal{O}_Y \lra \mathcal{L}$. It is well-known that if $H^i(Y, \mathcal{F} \otimes \mathcal{L}^m) = 0$ for all $i\geq 1$ and $m\geq 1$, then
    $$
    K_{p,q}(Y, \mathcal{F}; \mathcal{L}) \cong H^1(Y, \wedge^{p+1} M_{\mathcal{L}} \otimes \mathcal{F} \otimes \mathcal{L}^{q-1})~~\text{ for $p \geq 0$ and $q \geq 2$}.
    $$
    We refer to \cite[\S 3]{Ein.Lazarsfeld.Asymptotic} and \cite[\S 2]{Park2} for more details. 

    \item The \emph{Castelnuovo--Mumford regularity (with respect to a very ample line bundle $\mathcal{L}$)}, denoted by $\operatorname{reg}(\mathcal{F})$, is the minimum integer $m$ such that $K_{i,j}(Y, \mathcal{F}; \mathcal{L})=0$ for all $i \geq 0$ and $j \geq m+1$. It is the same as the minimum integer $m$ such that $H^i(Y, \mathcal{F} \otimes \mathcal{L}^{m-i})=0$ for $i>0$. See \cite[\S 1.8]{pag1} for more details. 

    \item When $\mathcal{F}=\mathcal{O}_X$, we simply write $R(Y, \mathcal{L}) := R(Y, \mathcal{O}_Y; \mathcal{L})$ and $K_{p,q}(Y; \mathcal{L}):=K_{p,q}(Y, \mathcal{O}_Y; \mathcal{L})$. Following \cite{eisenbud.green.hulek.popescu.05}, we say that $Y \subseteq \mathbb{P}H^0(Y,\mathcal{L})$ satisfies \emph{property $N_{k+1,p}$} for an integer $k \geq 1$ if $K_{i,j}(Y; \mathcal{L})=0$ for $0 \leq i \leq p$ and $j \geq k+1$.
\end{enumerate}

\section{Preliminaries on Hilbert schemes}\label{section-hilbert-scheme}

In this section, we introduce our framework for studying (nested) Hilbert schemes of points, including setting notation and terminology. 

\subsection{Hilbert scheme of points} For $k \geq 1$ we denote by $\gls{Hilb-k}$ the \emph{Hilbert scheme of $0$-dimensional subschemes of $X$ of length $k$}. We will refer to $0$-dimensional subschemes of length $k$ more briefly as \textit{$k$-schemes}. Let 
$$
\gls{mathcalZk}:=\{(x, \xi) \mid x \in \operatorname{Supp}(\xi)\} \subseteq X\times X^{[k]}
$$
be the \emph{universal family} of $X^{[k]}$. The components of the following commutative diagram
\begin{center}
    \begin{tikzcd}
        & X\times X^{[k]}\ar[dl,"\pr_1"']\ar[dr,"\pr_2"]\\
        X & \mathcal{Z}_k \ar[r] \ar[l] \ar[u,symbol=\subseteq] & X^{[k]}
    \end{tikzcd}
\end{center}
will feature heavily throughout this paper.

\subsubsection{Distinguished subschemes}\label{distinguished-subschemes} Within the Hilbert scheme $X^{[k]}$ of points, we will be variously interested in the following distinguished subschemes:
\begin{enumerate}[topsep=0pt]
    \item $\gls{Hilb-k-sm}$, which denotes the \emph{smoothable component} (also often called the \emph{main component} or the \emph{principal component}) --- it is the closure of the locus consisting of reduced $k$-schemes,
    \item $\gls{Hilb-k-gor}$, which consists of the Gorenstein $k$-schemes,
    \item $\gls{Hilb-k-sm-gor}:=X_{\operatorname{sm}} ^{[k]}\cap 
    X_{\operatorname{Gor}}^{[k]}$, which consists of smoothable Gorenstein $k$-schemes,
    \item $\gls{Hilb-k-lci}$, which consists of the local complete intersection $k$-schemes, and
    \item $\gls{Hilb-k-*}$, which consists of $k$-schemes whose reduction has at least $k-1$ points.
\end{enumerate}

\noindent We of course have inclusions $X_* ^{[k]}\subseteq X_{\operatorname{lci}}^{[k]}\subseteq X_{\operatorname{sm},\operatorname{Gor}}^{[k]}$ (one sees that an lci scheme $\xi$ is smoothable by taking a general deformation of a regular sequence defining it). Note that $X_* ^{[k]}$, $X_{\operatorname{lci}}^{[k]}$ and $X_{\operatorname{Gor}}^{[k]}$ are open (the latter two are the loci of subschemes $\xi \subseteq X$ where the conormal sheaf and the dualizing sheaf, respectively, of $\xi$ are locally free) and in fact $X_{\operatorname{lci}}^{[k]}$ is smooth (see \cite[Theorem 3.10]{huneke.ulrich.88}). In general though, $X_{\operatorname{sm}}^{[k]}$
and $X_{\operatorname{Gor}}^{[k]}$ are not smooth. While $X^{[k]}$ may be non-reduced when $\operatorname{dim}(X)$ and $k$ are large (see \cite{Jelisiejew.20}), the scheme $X^{[k]}_{\operatorname{sm}}$ is always reduced since it is the closure of a smooth variety. We note that when $X^{[k]}$ is smooth we have $\operatorname{codim}(X^{[k]}\setminus X_* ^{[k]})\geq 2$.

\subsubsection{The Hilbert--Chow morphism}
A $k$-scheme $[\xi] \in X^{[k]}$ naturally determines an effective $0$-cycle of degree $k$ on $X$. This globalizes to a morphism from $X^{[k]}$ to the $k^{\text{th}}$ symmetric product $\operatorname{Sym}^k(X)$ called the \emph{Hilbert--Chow morphism}
\begin{center}
    \begin{tikzcd}[row sep = tiny]
        \gls{Hilbert--Chow} \colon X^{[k]} \ar[r] & \operatorname{Sym}^k(X)\\
        \xi \ar[r, mapsto] & \sum\textnormal{length}(\OO_{\xi,p})\cdot p\quad 
    \end{tikzcd}
\end{center}
\noindent (see for example \cite[Theorem 7.3.1]{Brion.Kumar}). 
The morphism $h$ is birational when $X^{[k]}$ is irreducible and it is an isomorphism when $\operatorname{dim}(X) = 1$. When $\operatorname{dim}(X) > 1$ its exceptional locus $\operatorname{Exc}(h)$ parametrizes the $k$-schemes in $X$ which are non-reduced --- we will denote the divisorial part of $\operatorname{Exc}(h)\cap X_{\operatorname{sm}}^{[k]}$ by $\gls{HC-exc}$. If $\operatorname{dim}(X)=2$, the fibers of $h$ have dimension at most $k-1$ by \cite{Briancon}.

\subsubsection{The symmetric product} Let $\gls{quotient-to-sym} \colon X^k\lra \operatorname{Sym}^{k} (X)$ be the quotient map.
Consider the incidence subscheme
\[
\gls{Theta-k}:=\{(x,\eta)|x\in\operatorname{Supp}(\eta)\}\subseteq X\times \operatorname{Sym}^k (X)
\]
with the reduced scheme structure.
We have an obvious identification $\Theta_k\cong X\times \operatorname{Sym}^{k-1}(X)$.
Furthermore, we have closed embeddings $\mathcal{Z}_k\lra(\operatorname{id}\times h)^{-1} \Theta_k$ and $\Delta_k ^0\lra(\operatorname{id}\times q)^{-1}\Theta_k$, where $\Delta_k ^0$ is the union of the diagonals $\Delta_{X\times X^k}^{0,1}\cup \cdots \cup \Delta_{X\times X^k} ^{0,k}$ (here $\Delta_{X\times X^k} ^{0,j}=\{(x,x_1,\cdots, x_k) \text{ such that } x=x_j\}$).
Define $\gls{Sym-X-*-k}$ to be the open subset of $\Sym^k (X)$ consisting of $0$-cycles $\eta$ such that $|\operatorname{Supp}(\eta)|\geq k-1$. Let $\gls{X-*-k} :=q^{-1}( \operatorname{Sym}_* ^k (X))$. Finally, let $\gls{tilde-X-*-k}$ be the blowing up of $X_* ^k$ along the union of all pairwise diagonals, denoted as \gls{Deltak}. By \cite[Lemma 4.4]{Fogarty.73}, there is a Cartesian diagram
\begin{center}
\begin{tikzcd}
\widetilde{X}_* ^k\ar[r, "\gls{tilde-h-k}"]\ar[d, "\gls{tilde-q-k}"]& X_* ^k\ar[d,"q"]\\
X_* ^{[k]}\ar[r, "h"] &\operatorname{Sym}_* ^k (X)
\end{tikzcd}.
\end{center}
Let $\Gamma_{i,j}$ be the $\widetilde{h}$-exceptional divisors mapping to $\Delta_{X^k} ^{i,j}$, and let $\gls{Gamma-k}$ be their sum. We may complete the above diagram as
\begin{center}
\begin{tikzcd}
\widetilde{\mathcal{Z}}_k\ar[r,"p"]\ar[d]& \widetilde{X}_* ^k\ar[r, "\widetilde{h}"]\ar[d, "\widetilde{q}"]& X_* ^k\ar[d,"q"]\\
\mathcal{Z}_k \cap X\times X_* ^{[k]}\ar[r,"\operatorname{pr}_2"]& X_* ^{[k]}\ar[r, "h"] &\operatorname{Sym}_* ^k (X)
\end{tikzcd}.
\end{center}
Note that there are $k$ morphisms $f_i \colon \widetilde{X}_* ^k \lra \mathcal{Z}_k$ coming from the $k$ projections $\widetilde{X}_* ^k \lra X$. Since $p$ is flat, we have that $\widetilde{\mathcal{Z}}_k = \bigcup_i \operatorname{Graph}(f_i)$.

\subsubsection{The isospectral Hilbert scheme}\label{isospectral} We define the \emph{isospectral Hilbert scheme} $X_k$ by taking the reduced fiber product
\begin{center}
\begin{tikzcd}
\gls{X-k} \ar[r, "H"]\ar[d,  "Q"]& X^k\ar[d,"q"]\\
X ^{[k]}\ar[r,  "h"] &\operatorname{Sym} ^k (X)
\end{tikzcd}.
\end{center}
By \cite[Theorem 3.1]{Haiman.01}, we have that $X_k$ is normal, Cohen--Macaulay and Gorenstein when $X$ is a smooth projective surface. 
Note that taking the reduction is essential, since the fiber product is not reduced for $\operatorname{dim}(X)\geq 2$ and $k\geq 2$. A subtle but fundamental consequence of this fact is that the closed embedding $\mathcal{Z}_k\lra (\operatorname{id}\times h_k)^{-1} \Theta_k$ is not an isomorphism in general.
By the above discussion, we have that $\widetilde{X}_* ^k$ is an open subset of $X_k$ and when $X^{[k]}$ is smooth then $\operatorname{codim}(X_k\setminus \widetilde{X}_* ^k)\geq 2$. 

\subsubsection{The discriminant and its root}\label{discriminant}
Define the \emph{Vandermonde line bundle} of the finite morphism $\pr_2\vert_{\mathcal{Z}_k}$ 
to be
$$
\gls{delta-k}:=\operatorname{det} (\operatorname{pr}_{2,*} \mathcal{O}_{\mathcal{Z}_k})^{-1}
$$
on $X^{[k]}$. 
This terminology is used by analogy with the discriminant of a finite extension of algebraic number fields --- the corresponding ideas for a finite morphism of smooth varieties $f \colon Y \lra Z$ are as follows. Firstly, the ramification locus of $f$ is scheme-theoretically the rank-drop scheme of the differential $f^*\omega_Z \lra \omega_Y$.  Equivalently, it is the support of the cokernel of the induced $\OO_Y$-homomorphism
$$
D \colon \OO_Y \lra \omega_{Y/Z}.
$$
Similarly, the branch locus is the rank-drop scheme of $f_*(D) \colon f_*\OO_Y \lra f_*\omega_{Y/Z}$, that is
$$
B := \textnormal{Supp}(\textnormal{coker}(f_*(D))).
$$
Notice that $f_*\omega_{Y/Z} \cong (f_*\OO_Y)^{\vee}$ by Grothendieck duality and therefore $f_*(D)$ equivalently comes from composing the multiplication $f_*\OO_Y\otimes_{\OO_Z}f_*\OO_Y \lra f_*\OO_Y$ of the $\OO_X$-algebra $f_*\OO_Y$ with the trace map $f_*\OO_Y \lra \OO_Z$. From the exact sequence
\begin{center}
\begin{tikzcd}
    0 \ar[r] & f_*\OO_Y \ar[r,"f_*(D)"] & (f_*\OO_Y)^{\vee} \ar[r] & \textnormal{coker}(f_*(D)) \ar[r] & 0
\end{tikzcd}
\end{center}
we see that $\textnormal{det}(f_*\OO_Y)^{-1} \cong \textnormal{det}(f_*\OO_Y)\otimes \OO_Z(B)$, or  in other words,
$$
\OO_Z(-B) \cong \textnormal{det}(f_*\OO_Y)^{\otimes 2}.
$$

\noindent When $X^{[k]}$ is smooth the branch locus $\operatorname{Br}_k$ of the finite morphism $\operatorname{pr}_2\vert_{\mathcal{Z}_k}:\mathcal{Z}_k\lra X^{[k]}$ is a Cartier divisor, the line bundle $\OO_{X^{[k]}}(-\operatorname{Br}_k)$ is the \emph{discriminant (ideal)} of the that morphism and therefore $\delta_k^{-1}$ is its square root:
$$
\delta_{k}^{\otimes 2}=\mathcal{O}_{X^{[k]}}(\operatorname{Br}_k).
$$
Moreover, when $\operatorname{dim}(X) \geq 2$, this branch locus coincides with the exceptional divisor $E_k$.

\begin{lemma}\label{pullback-delta}
We have $\widetilde{q} ^* \delta_k \cong \mathcal{O}_{\widetilde{X}_* ^k}(\Gamma_k)$.
\end{lemma}
\begin{proof}
There is a short exact sequence
\[
0\lra \mathcal{O}_{\widetilde{\mathcal{Z}}_k}\lra
\bigoplus_{i=1} ^k \mathcal{O}_{\operatorname{Graph}(f_i)}\lra
\bigoplus_{1\leq i<j\leq k}
\mathcal{O}_{\operatorname{Graph}(f_i)\cap \operatorname{Graph}(f_j)}\lra 0.
\]
Since $p$ is finite, we obtain a short exact sequence
\[
0\lra p_* \mathcal{O}_{\widetilde{\mathcal{Z}}_k}
\lra \mathcal{O}_{\widetilde{X}_* ^k} ^{\oplus k}
\lra \mathcal{O}_{\Gamma_k}\lra 0.
\]
The statement follows by taking the determinant of the left map. 
\end{proof} 

\subsubsection{Smoothness and irreducibility}\label{subsubsection-smoothness}
The Hilbert scheme of points on a connected projective scheme is itself connected (see \cite{Hartshorne.Connectedness} and \cite[Proposition 2.3]{Fogarty.68}), so smoothness will imply irreducibility. We have that
$$
\text{$X^{[k]}$ is smooth}~\Longleftrightarrow~\text{$\dim(X) \leq 2$ or $k \leq 3$.}
$$
The first case was originally proven in \cite[Theorem 2.4]{Fogarty.68} 
and both cases are encompassed by \cite[\S 0.2 and Theorem 3.2.2]{Cheah.98}.

The Gorenstein locus $X_{\textnormal{Gor}}^{[k]}$ is of particular interest in the literature ---  it is known that
$$
\text{$X_{\textnormal{Gor}}^{[k]}$ is irreducible if either $k \leq 13$ or $k=14$ and $\dim (X) \leq 5$.}
$$
Here, both results are \cite[Theorems A and B]{Casnati.Jelisiejew.Notari.15}. Keeping in mind the above statement about the irreducibility of the Gorenstein locus, we see by \cite[Proposition 2.1]{Casnati.Notari.07} and \cite[Theorem B]{Casnati.Notari.09} that 
$$
\text{$X_{\operatorname{Gor}} ^{[k]}$ is smooth}~\Longleftrightarrow~\text{$X_{\operatorname{sm,Gor}}^{[k]}$ is smooth}~\Longleftrightarrow~\text{$\dim(X) \leq 3$ or $k\leq 5$}.
$$

\subsubsection{The canonical bundle}\label{hilbert-canonical}
Assume that $X^{[k]}$ is smooth (that is, either $n\leq 2$ or $k\leq 3$, by the previous section). The canonical bundle of the Hilbert scheme can be expressed in terms of the canonical bundle of $X$ as follows. For any line bundle $L$ on $X$, the line bundle $L^{\boxtimes k}$ on $X^k$ descends equivariantly to a line bundle $\gls{S-k}$ on $\operatorname{Sym}^k(X)$. We denote by $\gls{T-k}$ the pullback $h^* S_{k,L}$ of this bundle to $X^{[k]}$ via the Hilbert--Chow morphism. Then we have 
$$
\omega_{X^{[k]}}\cong T_{k,\omega_X}\otimes \delta_k ^{n-2}.
$$

\noindent One may consult \cite[Theorem 7.4.6]{Brion.Kumar} for the case $n=2$ --- the same proof gives the result for any $n$ provided that $k\leq 3$. For the reader's convenience, we will give a short proof  in Lemma \ref{canonical-bundle}. 

\subsubsection{The tangent and cotangent sheaves} The description of the cotangent sheaf of Quot in Appendix \ref{Section-cotangent} gives
the following two theorems. 

\begin{theorem}\label{cotangent-sheaf-hilbert-scheme}
There are isomorphisms
\[
\Omega_{X^{[k]}} ^1 \cong \SExt_{\operatorname{pr}_2} ^n(\mathcal{O}_{\mathcal{Z}_k}, \mathcal{I}_{\mathcal{Z}_k}\otimes \operatorname{\omega}_{\operatorname{pr}_2})
~~\text{ and }~~
T_{X^{[k]}} \cong \operatorname{pr}_{2,*} \SHom (\mathcal{I}_{\mathcal{Z}_k}, \mathcal{O}_{\mathcal{Z}_k}).
\]
\end{theorem}

\begin{theorem}\label{conormal-cotangent}
The natural pairing
\[
(\mathcal{I}_{\mathcal{Z}_k}/\mathcal{I}_{\mathcal{Z}_k}^2)\otimes
\operatorname{pr}_2 ^*
\operatorname{pr}_{2,*} \SHom (\mathcal{I}_{\mathcal{Z}_k}, \mathcal{O}_{\mathcal{Z}_k})\lra
\mathcal{O}_{\mathcal{Z}_k}
\]
is compatible with the natural map
\[
\gls{D-k} \colon \mathcal{I}_{\mathcal{Z}_k}/\mathcal{I}_{\mathcal{Z}_k}^2\lra
\operatorname{pr}_2 ^* \Omega_{X^{[k]}} ^1 \vert_{\mathcal{Z}_k}.
\]
\end{theorem}

\subsection{The nested Hilbert scheme of points}\label{nestedHilb}
For $m\leq k$, the \emph{nested Hilbert scheme of $m$-subschemes of $k$-schemes} $\gls{Nested-Hilb}\subseteq X^{[m]}\times X^{[k]}$ consists of pairs $(\xi, \eta)$ such that $\xi\subseteq \eta$. 
There are natural maps $\gls{tau-m-k} \colon X^{[m,k]}\lra X^{[m]}$ and
$\gls{r-m-k} \colon X^{[m,k]}\lra X^{[k]}$ giving rise to the diagram
\begin{center}
    \begin{tikzcd}
    & X^{[m,k]} \ar[r,symbol=\subseteq]\ar[dl,"\tau"] \ar[dr,"\rho"'] & X^{[m]}\times X^{[k]}\\
    X^{[m]}& & X^{[k]}
    \end{tikzcd}.
\end{center}
The universal families 
$$
\gls{relative-family-small}:=\{(x, \xi \subseteq \eta) \mid x \in \xi \} \subseteq \gls{relative-family-big}:=\{(x, \xi \subseteq \eta) \mid x \in \eta\} \subseteq X\times X^{[m,k]}
$$
give rise to the commutative diagram 
\begin{center}
    \begin{tikzcd}
    & X\times X^{[m,k]} \ar[r,symbol=\subseteq]\ar[ddl,"\pr_1"'] \ar[ddr,"\pr_2"] &  X\times X^{[m]}\times X^{[k]}\\
    & \mathcal{W}_{m,k}\ar[u,symbol=\subseteq] \ar[dl]\ar[dr]\\
    X & \mathcal{V}_{m,k} \ar[u,symbol=\subseteq]\ar[l]\ar[r] & X^{[m,k]}
    \end{tikzcd}.
\end{center}

\subsubsection{The cotangent sheaves}
\noindent We have an analog to Theorem \ref{cotangent-sheaf-hilbert-scheme} (see Appendix \ref{Section-cotangent} for the proof).

\begin{theorem}\label{cotangent-sheaf-nested-hilbert-scheme}
The cotangent sheaf of the nested Hilbert scheme $X^{[m,k]}$
may be identified as 
\[
\Omega_{X^{[m,k]}}^1\cong
\operatorname{coker}\left(\SExt_{\operatorname{pr}_2} ^n(\mathcal{O}_{\mathcal{V}_{m,k}}, \mathcal{I}_{\mathcal{W}_{m,k}})
\lra \SExt_{\operatorname{pr}_1}^n (\mathcal{O}_{\mathcal{W}_{m,k}}, \mathcal{I}_{\mathcal{W}_{m,k}})\oplus 
\SExt_{\operatorname{pr}_2}^n (\mathcal{O}_{\mathcal{V}_{m,k}}, \mathcal{I}_{\mathcal{V}_{m,k}})\right)\otimes \omega_{\operatorname{pr}_2}.
\]
\end{theorem}

\subsubsection{Smoothness}\label{hilbert-key-facts}
\noindent By \cite[\S 0.2]{Cheah.98}, the nested Hilbert scheme $X^{[m,k]}$ of points is smooth precisely in one of the cases:
$$
\text{(1) $\dim(X) = 1$, (2) $\dim(X) = 2$ and $m=k-1$, (3) $m=1$ and $k=2$, or (4) $m=2$ and $k=3$.}
$$
Note in particular that if $X^{[m,k]}$ is smooth, then both $X^{[m]}$ and $X^{[k]}$ are smooth.

\subsubsection{The residual morphism}\label{residue-map}
If $\xi \subseteq \eta$ is a nested pair and $m = k-1$ then the ideal quotient $(\mathcal{I}_{\eta} : \mathcal{I}_{\xi})$ always has colength $1$ and thus is the maximal ideal of a closed point $p \in \textnormal{Supp}(\eta)$ that is residual to $\xi$ in $\eta$. This globalizes to a morphism from $X^{[k-1,k]}$ to $X$ called the \textit{residual morphism}
\begin{center}
\begin{tikzcd}[row sep = tiny]
\gls{res} \colon  X^{[k-1,k]} \ar[r] & X\\
(\xi \subseteq \eta) \ar[r,mapsto] & (\mathcal{I}_{\eta}:\mathcal{I}_{\xi})
\end{tikzcd}.
\end{center}
We denote by $\Gamma_{\textnormal{res}} \subseteq X\times X^{[k-1,k]}$ the graph of $\textnormal{res}$.

When $\operatorname{dim}(X) = 1$ the colength of $(\mathcal{I}_{\eta}:\mathcal{I}_{\xi})$ is always $k-m$ for all $1 \leq m \leq k$ so the residual morphism is also defined for $m < k-1$ and its fibers are naturally identified with $X^{[k-m]}$ (in fact in that case $X^{[m,k]} \cong X^{[m]}\times X^{[k-m]}$ and the residual morphism is then just projection to the second factor).

When $\operatorname{dim}(X) = 2$ or $k \leq 3$ we have that $X^{[k-1,k]}$ is smooth by the above Subsection \ref{hilbert-key-facts}. Hence $\operatorname{res}$ is flat by miracle flatness. Moreover, the general fiber of $\operatorname{res}$ is smooth by generic smoothness. By comparing the singularities of the fibers, we see in fact that $\operatorname{res}$ is a smooth morphism. Recall that we have an isomorphism
$$
X^{[k-1,k]}\cong \operatorname{Bl}_{\mathcal{Z}_{k-1}} X\times X^{[k-1]},
$$
where the blowup map $\gls{blow-up-of-universal-family}:X^{[k-1,k]}\lra X\times X^{[k-1]}$ is identified with the product of the residue morphism and $\tau$:
$$
\textnormal{bl} \cong \textnormal{res}\times \tau
$$
(see for example \cite[Proposition 2.2]{Ellingsrud.Stromme} and \cite[Proposition 2.5.8]{goettsche.94}). 
Therefore, there is a
distinguished Cartier divisor $\gls{exceptional-divisor-F} \subseteq X^{[k-1,k]}$ which is the exceptional divisor of $\textnormal{bl}$.
This identification allows for a useful computation of the canonical bundle of $X^{[k-1,k]}$:
$$
\omega_{X^{[k-1,k]}}=\operatorname{bl}^* (\omega_X \boxtimes \omega_{X^{[k-1]}})((n-1)F_{k-1}).
$$
Recall the graph of the residual morphism $\Gamma_{\textnormal{res}} \cong X^{[k-1,k]}$. Notice that if $\xi \subseteq \eta$ is a nested pair (still with $m = k-1$) and $p \in \textnormal{Supp}(\eta)$, then the only possibilities for $p$ are that it is supported in $\xi$ or it is residual to $\xi$ in $\eta$ --- this observation globalizes to the following coincidence of subschemes in $X\times X^{[k-1,k]}$:
$$
\mathcal{W}_{k-1,k} = \mathcal{V}_{k-1,k}\cup \Gamma_{\textnormal{res}}.
$$

It is possible that $p$ is both residual to $\xi$ in $\eta$ and yet also contained in $\textnormal{Supp}(\xi)$ --- this of course means that $\textnormal{Supp}(\xi) = \textnormal{Supp}(\eta)$, which implies that the differential of $\textnormal{res}\times \tau$ drops rank at $(\xi \subseteq \eta) \in X^{[k-1,k]}$. Since $\textnormal{res}\times \tau$ coincides with $\textnormal{bl}$, and since $\mathcal{V}_{k-1,k}$ and $\mathcal{W}_{k-1,k}$ are flat over $X^{[k-1,k]}$, this observation globalizes to the following commutative diagram:
\begin{center}
\begin{tikzcd}
\mathcal{V}_{k-1,k}\cap \Gamma_{\textnormal{res}}\ar[r,symbol=\cong] \ar[d,symbol=\subseteq] & F_{k-1}\ar[d,symbol=\subseteq]\\
\Gamma_{\textnormal{res}} \ar[r,symbol=\cong] & X^{[k-1,k]}
\end{tikzcd}
\end{center}
(here the horizontal isomorphisms are given by the projection $X\times X^{[k-1,k]} \lra X^{[k-1,k]}$). 
Hence we have
$$
\mathcal{I}_{\mathcal{V}_{k-1,k}/\mathcal{W}_{k-1,k}} \cong \iota_*\mathcal{O}_{X^{[k-1,k]}}(-F_{k-1})
$$
for $\iota$ the inclusion map $X^{[k-1,k]}\cong \Gamma_{\textnormal{res}} \subseteq \mathcal{W}_{k-1,k}$.

We conclude by some remarks on the morphism $\tau$. 
The fiber of $\operatorname{bl}$ above a point $(p,[\xi])\in X\times X^{[k-1]}$ is a projective space of dimension one less than the minimal number of generators of $\mathcal{I}_{\xi}$ at $p$ by \cite[page 2550]{Ellingsrud.Stromme}. 
As a consequence,
the restriction to $\tau^{-1} ([\xi])\lra X$ of the residual morphism for $[\xi]\in X^{[k-1]}$ may be identified with the blowing up of $\operatorname{Bl}_\xi X\lra X$ if $\xi$ is a local complete intersection scheme. In general, however, there is only a closed embedding $\operatorname{Bl}_\xi X\subseteq \tau ^{-1} ([\xi])$ and the other irreducible components of $\tau^{-1}([\xi])$ are exceptional over $X$ with respect
to $\operatorname{res}$. In particular, $\tau$ is flat if and only if $k\leq 3$.
Another consequence is the following.

\begin{lemma}\label{rinverseXstar}
Let $X$ be a smooth projective variety. If $\operatorname{dim}(X)=2$ or $k\leq 3$ then we have $\operatorname{codim}(X^{[k-1,k]}\setminus \rho^{-1}(X_* ^{[k]}))\geq 2$.
\end{lemma}
\begin{proof}
Let $X_{\leq \ell} ^{[k]}$ be the set of $k$-schemes such that the reduction has length at most $\ell$. We want to show
    \[
    \operatorname{codim}(\rho^{-1} (X_{\leq k-2} ^{[k]}))\geq 2.
    \]
If $k\leq 3$ this is obvious, so assume $\operatorname{dim}(X)=2$. Note $\tau(\rho^{-1} (X_{\leq k-2} ^{[k]}))=X_{\leq k-2} ^{[k-1]}$. Furthermore we have $\dim(X_{\leq \ell} ^{[k-1]})\leq 2(k-1) + \ell - k+1$ by \cite{Briancon}. Pick a point $[\xi]$ in $X_{\leq \ell} ^{[k-1]} \setminus X_{\leq \ell -1}^{[k-1]}$ for some $\ell\leq k-2$. By the above discussion on the fibers of $\tau$, we have
\[
\operatorname{dim}(\tau^{-1} ([\xi]))\leq \operatorname{max}_{p\in \xi}\{2, \operatorname{length}_p(\xi)-1\}.
\]
If $\ell=k-2$, then $\xi$ is local complete intersection and $\tau^{-1}([\xi])\cong \operatorname{Bl}_\xi X$. Therefore, we have
\[
\operatorname{dim}(\tau^{-1}([\xi])\cap \rho^{-1} (X_{\leq k-2} ^{[k]}))\leq 1.
\]
Since $\operatorname{length}_p(\xi)\leq k+1-\ell$, we have that $\tau^{-1}(X_{\leq \ell} ^{[k-1]} \setminus X_{\leq \ell-1} ^{[k-1]})\cap \rho^{-1}(X_{\leq k-2} ^{[k]})$ has codimension at least 2 for all $\ell$, hence we are done.
\end{proof}

\subsubsection{Properties of the universal families}\label{subsection-uni-family}
By \cite{Fogarty.73} we have that $\mathcal{Z}_k$ is Cohen--Macaulay and satisfies property $R_3$ (hence it is normal).
There is a canonical isomorphism $\tau\times \rho \colon X^{[1,k]}\lra \mathcal{Z}_k$. Similarly, there is a canonical isomorphism $X^{[1,k-1,k]}\cong \mathcal{V}_{k-1,k}$. If $\operatorname{dim}(X)\leq 2$ or $k\leq 3$, there is a resolution of singularities $\operatorname{res}\times \rho \colon X^{[k-1,k]}\lra \mathcal{Z}_k$.
If $\operatorname{dim}(X)\leq 2$, Song \cite{Song.16} has shown that $\mathcal{Z}_k$ is has rational singularities. The same proof works for $k \leq 3$, so $\mathcal{Z}_k$ has rational singularities whenever $\operatorname{dim}(X) \leq 2$ or $k \leq 3$.

\section{Secant varieties, secant sheaves, and secant bundles}\label{section-secant-sheaves}

The main purpose of this section is to introduce absolute and relative secant varieties. In order to resolve singularities, we study secant sheaves and their projectivization, secant bundles. 
Along the way, we establish basic properties of some useful line bundles on the Hilbert scheme of points.

\subsection{Secant varieties}\label{section-geometry}
For an integer $k \geq 1$, we rely on the embedding $X \subseteq \mathbb{P} H^0(X, L) = \mathbb{P}^r$ by $L$ to define the 
\emph{$k^{\text{th}}$ secant variety} 
$$
\gls{Sigma-k}:= \overline{\bigcup\text{Span}_{\mathbb{P}^r}(p_0,\ldots,p_k)},
$$
\noindent where the union is taken over $(k+1)$-tuples of distinct points of $X$. The secant variety $\Sigma_k$ is equipped with the reduced scheme structure. 

While use of the notation $\Sigma_k$ seems to be classical for the secant variety, it is only one of at least two common conventions within the literature. On its own, it poses no problem, but when studying the secant varieties as they relate to corresponding Hilbert schemes of points (as we shall do) we find it greatly clarifies notation to adopt the other common convention --- to this end, we therefore define (and from now on use) the notation:
$$
\gls{sigma-k} := \Sigma_{k-1}.
$$
We say that $\sigma_k$ is the \emph{$k$-secant variety}.
In this way, $\sigma_k$ will be the image of a projective bundle over $X^{[k]}$, hence the convention. One might think of the $k^{\text{th}}$ secant variety $\Sigma_k$ as denoting the \textit{variety of $k$-dimensional secant planes}, and of the $k$-secant variety $\sigma_k$ as denoting the \textit{variety of $k$-secant planes}.

We have a natural stratification
\begin{equation*}
X=\sigma_1\subseteq \sigma_2 \subseteq \cdots \subseteq \sigma_{k-1} \subseteq \sigma_k.
\end{equation*}
Understanding the geometry of the natural stratification is key for our study of secant varieties. 

\subsection{Secant sheaves}
In studying $k$-secant planes, it is natural to study the cases where the span in $\mathbb{P}^r$ of every
length-$k$ subscheme $\xi\subseteq X$ has the maximal dimension, $k-1$. In this spirit, we recall the following well-known definition.

\begin{definition}
For a line bundle $\mathcal{L}$ on a smooth projective variety $Y$ and an integer $p \geq 0$, we say that $\mathcal{L}$ \emph{separates $(p+1)$-schemes} (or is \emph{$p$-very-ample}) if for any length $p+1$ subscheme $\xi \subseteq Y$, the restriction map \[H^0(Y,\mathcal{L}) \lra H^0(Y,\mathcal{L}\otimes \OO_{\xi})\] is surjective.
\end{definition}

Take the short exact sequence
\vspace{-10pt}
\begin{center}
\begin{tikzcd}
0\ar[r] & \pr_1^* L \otimes \mathcal{I}_{\mathcal{Z}_k}\ar[r] & \pr_1^* L \ar[r] &
\pr_1^*L \otimes \mathcal{O}_{\mathcal{Z}_k}\ar[r] & 0
\end{tikzcd}
\end{center}
\vspace{-10pt}
on $X\times X^{[k]}$, and define the \textit{sheaf of $k$-secant planes $E_{k,L}$} (also known as the \emph{tautological bundle} of $L$ on $X^{[k]}$) and its corresponding \textit{syzygy bundle $M_{k,L}$} via the pushed-forward sequence on $X^{[k]}$, as follows:
\begin{equation}\label{secSES}
\begin{tikzcd}
0 \ar[r] & \pr_{2,*} (\pr_1^* L \otimes \mathcal{I}_{\mathcal{Z}_k}) \ar[r] & \pr_{2,*} (\pr_1^* L) \ar[r] &
\pr_{2,*} (\pr_1^*L \otimes \mathcal{O}_{\mathcal{Z}_k})\ar[r] & 
0.\\
& \gls{M-k-L} \ar[u,symbol=\coloneq] & & \gls{E-k-L} \ar[u,symbol=\coloneq]
\end{tikzcd}
\end{equation}
\vspace{-20pt}

\noindent The fiber of $E_{k,L}$ over $\xi \in X^{[k]}$ is isomorphic to $H^0(\xi, L|_{\xi})$, and therefore $E_{k,L}$ is locally free of rank $k$.
If $L$ separates $k$-schemes, Sequence (\ref{secSES}) is exact, and therefore $E_{k,L}$, $M_{k,L}$ and $\gls{N-k-L} :=\det E_{k,L}$ are all globally generated. 
Note that for $k=1$ we have $E_{1,L}=N_{1,L}=L$ and $M_{1,L}=M_L$. 

The sheaves $M_{k,L}$ and $E_{k,L}$ have useful analogues in the relative setting --- take the short exact sequence
\begin{center}
\begin{tikzcd}
0\ar[r] & \pr_1^*L \otimes \mathcal{I}_{\mathcal{V}_ {m,k}/\mathcal{W}_{m,k}}\ar[r]& \pr_1^* L
\otimes \mathcal{O}_{\mathcal{W}_{m,k}} \ar[r]&
\pr_1^*L \otimes \mathcal{O}_{\mathcal{V}_{m,k}}\ar[r]& 0
\end{tikzcd}
\end{center}
\noindent on $X\times X^{[m,k]}$ and define the \emph{relative secant sheaf of $m$-secant planes in $k$-secant planes} $E_{m,k,L}$ and its corresponding \emph{relative syzygy bundle} $M_{m,k,L}$ via the pushed-forward sequence on $X^{[m,k]}$, as follows:
\begin{equation}\label{relsecSES}
\begin{tikzcd}[scale cd=0.94]
0 \ar[r] & \pr_{2,*} (\pr_1^*L \otimes \mathcal{I}_{\mathcal{V}_{m,k}/\mathcal{W}_{m,k}})\ar[r] & \pr_{2,*} (\pr_1^* L
\otimes \mathcal{O}_{\mathcal{W}_{m,k}}) \ar[r] &
\pr_{2,*}(\pr_1^*L \otimes \mathcal{O}_{\mathcal{V}_{m,k}})\ar[r] & 0.\\
& \gls{M-m-k-L} \ar[u,symbol=\coloneq] & & \gls{E-m-k-L}\ar[u,symbol=\coloneq]
\end{tikzcd}
\end{equation}
Note that $E_{m,k,L}=\tau^* E_{m,L}$, and (2) above may be rewritten as
\vspace{-10pt}
\begin{center}
\begin{tikzcd}
0\ar[r]& M_{m,k,L}\ar[r]&
\rho^* E_{k,L}\ar[r]& \tau^* E_{m,L}\ar[r]& 0.
\end{tikzcd}
\end{center}
\vspace{-10pt}

\noindent If $L$ separates $k$-schemes, the snake lemma gives also a dual sequence
\vspace{-10pt}
\begin{center}
\begin{tikzcd}
0\ar[r]& \rho^* M_{k,L}\ar[r]& \tau^* M_{m,L}\ar[r]& M_{m,k,L}\ar[r]& 0.
\end{tikzcd}
\end{center}
\vspace{-10pt}

\subsection{Secant bundles}\label{subsec:secant bundles}
Here we establish the connection between secant sheaves and secant varieties.
We let
$$
\gls{secant-bundle} := \mathbb{P}(E_{k,L}) \overset{\gls{pi-k}} \lra X^{[k]}
$$
denote the \textit{bundle of $k$-secant planes over $X^{[k]}$} and we let $\gls{dim-k}:=k\cdot \dim(X)+k-1$ denote its dimension.
If $L$ separates $k$-schemes, then Sequence (\ref{secSES}) determines a morphism
\[
\gls{alpha-k} \colon B^k\overset{\gls{i-k}}{\lra} \mathbb{P}^r\times X^{[k]}\lra  \mathbb{P}^r,
\]
where $i_k$ is a closed embedding. For $\xi \in X^{[k]}$, we have 
$$
i_k(\pi_k^{-1}(\xi)) = \{(x, \xi) \in \mathbb{P}^r \times X^{[k]} \mid x \in \operatorname{Span}_{\mathbb{P}^r}(\xi)\},
$$
so $\alpha_k(\pi_k^{-1}(\xi)) = \operatorname{Span}_{\mathbb{P}^r}(\xi)$. 
Let $B_{\operatorname{sm}} ^k\subseteq B^k$ be the subscheme induced by the embedding $X_{\operatorname{sm}} ^{[k]}\subseteq X^{[k]}$. The point of the construction is of course
the equality
\[
\alpha_k (B_{\operatorname{sm}}^k)=\sigma_k \subseteq \mathbb{P}^r.
\]
On the other hand, we have $\alpha_k(B^k)=\kappa_k$, where $\gls{cactus-k}$ is the \emph{$k$-cactus scheme} --- that is, the union of all $k$-secant $(k-1)$-planes (see \cite{BB}). Note that while $X_{\operatorname{sm}} ^{[k]}$, and hence $\sigma_k$, is always reduced and irreducible, $X^{[k]}$ may in general be both reducible and non-reduced (see \cite{Jelisiejew.20}). Therefore, the natural scheme structure on $\kappa_k$ is not the reduced structure, but the one induced by the morphism $\alpha_k$. 

The above construction may be generalized to the relative case. The \emph{relative secant bundle of
$m$-secant planes contained in $k$-secant planes over $X^{[m,k]}$} is 
\[
\gls{relative-secant-bundle}:= \mathbb{P}(E_{m,k,L})
\overset{\gls{pi-m-k}}{\lra} X^{[m,k]}
\]
and the \emph{relative secant bundle of $k$-secant planes over $X^{[m,k]}$} is
\[
\gls{larger-relative-secant-bundle}:= \mathbb{P}(r^* E_{k,L})\overset{\gls{p-m-k}}{\lra} X^{[m,k]}.
\]
The short exact Sequence (\ref{relsecSES}) determines a morphism
\[
\gls{alpha-m-k} \colon B^{m,k}\overset{\gls{i-m-k}}{\lra}P^{m,k}\overset{\gls{r-tilde}}{\lra} B^k,
\]
where $i_{m,k}$ is a closed embedding and $\widetilde{\rho}$ is the base-change of $\rho$ by $\pi_m$. The fiber of $\pi_{m,k}$ over $(\xi, \eta) \in X^{[m,k]}$ is isomorphic to $H^0(\xi, L|_{\xi})$. 
With this in mind, the inclusion $\alpha_{m,k}(\pi_{m,k}^{-1}(\xi, \eta)) \subseteq \pi_k^{-1}(\eta)$ corresponds to the inclusion $H^0(\xi, L|_{\xi}) \subseteq H^0(\eta, L|_{\eta})$.
By analogy with the absolute setting, we define the \emph{relative cactus scheme}
as
$$
\gls{kappa-m-k}:=\alpha_{m,k}(B^{m,k}) \subseteq B^k.
$$

\noindent When $X^{[m,k]}$ is irreducible (which will be the most important case for us), we denote $\kappa_{m,k}$ as $\gls{sigma-m-k}$ and call it the \emph{relative secant variety}.

The morphisms $\alpha_k$ and $\alpha_{m,k}$ are compatible, that is, there is a commutative diagram
\begin{center}
\begin{tikzcd}
B^{m,k}\ar[r, "\alpha_{m,k}"]\ar[d, "\tilde{\tau}"]& B^k\ar[d,"\alpha_k"]\\
B^m\ar[r, "\alpha_m"] & \mathbb{P}^r
\end{tikzcd},
\end{center}
where $\gls{tautilde}$ is the base-change of $\tau$ via $\pi_m$. This diagram is not Cartesian. It turns out, however, that it is Cartesian when restricted to an appropriate open set.
To this end, define 
\[
\gls{u-k}:=B^k\setminus \kappa_{k-1,k}
~~\text{ and }~~
\gls{u-m-k}:=U^m\times_{X^{[m]}} X^{[m,k]}.
\]
We will show in Proposition \ref{fiber-product} that if we replace $\widetilde{\tau} \colon B^{m,k} \to B^m$ by $\widetilde{\tau}|_{U^{m,k}} \colon U^{m,k} \to U^m$, then the above diagram is in fact Cartesian.

We conclude with a couple of remarks. First, if $L$ separates $2k$-schemes, then 
\[
\alpha_k(\pi_k^{-1}(\xi)) \cap \alpha_k(\pi_k^{-1}(\xi')) = 
\operatorname{Span}_{\mathbb{P}^r}(\xi)\cap 
\operatorname{Span}_{\mathbb{P}^r}(\xi')=
\operatorname{Span}_{\mathbb{P}^r}(\xi \cap \xi')
\]
for any two $k$-schemes $\xi$ and $\xi'$.
As a consequence, we obtain $\alpha_k^{-1}(\kappa_m) = \kappa_{m,k}$. Similarly, if $L$ separates $k$-schemes, then as sets we have $\alpha_{m,k}^{-1} (\kappa_{\ell,k})=\widetilde{\tau}^{-1} (\kappa_{\ell, m})$ for $\ell < m < k$.

\subsubsection{Global generation and the universal family}\label{global-generation-universal-family} 
Assume that $L$ separates $k$-schemes. On $X\times X^{[k]}$ we have a surjective homomorphism
\[
L^{-1}\boxtimes M_{k,L}\lra \mathcal{I}_{\mathcal{Z}_k}.
\]
We can obtain this by pulling back Sequence (\ref{secSES}) to $X\times X^{[k]}$ and observing that the composite diagonal morphism in the following diagram drops rank along $\mathcal{Z}_k$ because of the assumption that $L$ separates $k$-schemes:
\vspace{-10pt}
\begin{center}
\begin{tikzcd}
    0 \ar[r] & \pr_2^*M_{k,L} \ar[r]\ar[dr] & H^0(X,L)\otimes \OO_{X\times X^{[k]}} \ar[r] \ar[d] & \pr_2^*E_{k,L} \ar[r] & 0\\
    & & \pr_1^*L
\end{tikzcd}.
\end{center}
\vspace{-10pt}
Note the canonical isomorphism 
$\wedge^j M_{k,L}\otimes N_{k,L}\cong \wedge^{r_k-j} M_{k,L}^{\vee}$ for $j\geq 1$, where $\gls{r-k}:=\operatorname{rank}(M_{k,L})$.
Since our assumptions on $L$ imply that $M_{k,L}^{\vee}$ is globally generated, the case $j=1$ of this isomorphism shows that $M_{k,L}\otimes N_{k,L}$ is globally generated. This discussion indicates that
$\mathcal{I}_{\mathcal{Z}_k}\otimes (L\boxtimes N_{k,L})$ is globally generated on $X\times X^{[k]}$.
 In particular, we also obtain that $\operatorname{res}^* L \otimes \tau^* N_{k-1,L}(-F_{k-1})$ is globally generated on $X^{[k-1,k]}$.

 \begin{remark}\label{no-global-generation}
     Note that $N_{k,L}$ cannot be replaced by $T_{k,L}$. In fact, even though $\mathcal{I}_{\Theta_k}\otimes (L\boxtimes S_{k,L})$ is globally generated, scheme-theoretically we have $\mathcal{Z}_k \neq (\operatorname{id}\times h)^{-1} \Theta_k$
     by the discussion in Subsection \ref{isospectral}.
 \end{remark}

\subsubsection{Line bundles on Hilbert schemes}\label{subsubsec-line-bundles}
Recall that $N_{k,L}:=\det E_{k,L}$. The following formula is well-known but we give a short proof. Here $\delta_k$ and $T_{k,L}$ were defined in \S\ref{discriminant} and \S\ref{hilbert-canonical}, respectively.

\begin{lemma}\label{N&S} If $X^{[k]}$ is smooth, then $N_{k,L}\cong T_{k,L}\otimes \delta_k ^{-1}$.
\end{lemma}
\begin{proof}
Since $X^{[k]}$ is normal, it suffices to argue on $X_* ^{[k]}$. 
By Lemma \ref{pullback-delta}, it suffices to show that $\widetilde{q} ^* N_{k,L}\cong \widetilde{h} ^* L^{\boxtimes k} \otimes \mathcal{O}_{\widetilde{X}_* ^k}(-\Gamma_k)$. To see this, notice that the map
\vspace{-10pt}
\[
\widetilde{q}\hspace{2pt}^* E_{k,L}\lra \widetilde{h} ^* L^{\boxplus k}
\]
\noindent drops rank along $\Gamma_k$. 
\end{proof}

The symmetric group $S_k$ acts on the sheaf $L^{\boxtimes k}$ on $X^k$ via permutations. We denote by $L ^{\boxtimes k, \operatorname{sym}}$ the sheaf $L ^{\boxtimes k}$ equipped with this action. If we multiply the above action via the sign of the permutation, we obtain another action, which we denote by $L ^{\boxtimes k, \operatorname{alt}}$ the sheaf $L^{\boxtimes k}$ equipped with this another action.
Then we have 
$h_* T_{k,L}\cong q_* ^{S_k} L ^{\boxtimes k, \operatorname{sym}}$ and $h_* N_{L,k} \cong q_* ^{S_k} L ^{\boxtimes k, \operatorname{alt}}$, where $q_* ^{S_k}$ denotes the equivariant pushforward. One can obtain both isomorphisms using the argument of \cite[Proposition 6.9]{Sheridan.2020} on $\widetilde{X}_* ^k$ (see also \cite[\S 2]{Ein.Lazarsfeld.Yang}). We have
$$
H^0(X^{[k]}, T_{k,L}) \cong S^k H^0(X, L)~~\text{ and }~~
H^0(X^{[k]}, N_{k,L}) \cong \wedge^k H^0(X, L).
$$
Now, we define 
$$
\gls{a-k}:=N_{k,L} \otimes \delta_k^{-1} = T_{k,L} \otimes \delta_k^{-2}.
$$
When $L$ is sufficiently positive, $H^0(X^{[k]}, A_{k,L})$ has a deeper meaning related to the algebraic properties of the embedded secant variety $\sigma_{k-1} \subseteq \mathbb{P}^r$ --- see Proposition \ref{prop:idealofsecants}.

The following lemma indicates how the line bundles introduced in this section, as well as $\delta_k$ from \S\ref{discriminant} and $F_{k-1}$ from \S\ref{residue-map}, interact after appropriate pullbacks to $X^{[k-1,k]}$. 

\begin{lemma}\label{identities} Assume that $X^{[k]}$ is normal. We have the following relations:
\begin{enumerate}[topsep=0pt]
    \item $\rho^* N_{k,L} \otimes \tau^* N_{k-1,L}^{-1}\cong \operatorname{res}^* L (-F_{k-1})$.
    \item $\rho^* \delta_k \otimes \tau^* \delta_{k-1} ^{-1} \cong \mathcal{O}_{X^{[k-1,k]}}(F_{k-1})$.
    \item $\rho^* A_{k,L}\otimes \tau^* A_{k-1,L}^{-1} \cong \operatorname{res}^* L (-2F_{k-1})$. 
    \item $\rho^* T_{k,L} \otimes \tau^* T_{k-1,L}^{-1} \cong \operatorname{res}^* L$.
\end{enumerate}
\end{lemma}
\begin{proof}
Recall from above the short exact sequence
\begin{center}
\begin{tikzcd}
0\ar[r]& M_{m,k,L}\ar[r]& \rho^* E_{k,L}\ar[r]&
\tau^* E_{m,L}\ar[r]& 0
\end{tikzcd}
\end{center}
on $X^{[m,k]}$. When $m = k-1$ we have $\operatorname{pr}_{2,*} \mathcal{I}_{\mathcal{V}_{k-1,k}/\mathcal{W}_{k-1,k}}=\mathcal{O}_{X^{[k-1,k]}}(-F_{k-1})$ by \S\ref{residue-map}. Hence $M_{k-1,k,L}\cong\operatorname{res}^* L (-F_{k-1})$ and therefore (1) follows by taking determinants.

Consider now the short exact sequence
\begin{center}
\begin{tikzcd}
0\ar[r]& \operatorname{pr}_{2,*} \mathcal{I}_{\mathcal{V}_{k-1,k}/\mathcal{W}_{k-1,k}} \ar[r]&
\operatorname{pr}_{2,*} \mathcal{O}_{\mathcal{W}_{k-1,k}}\ar[r]& \operatorname{pr}_{2,*} \mathcal{O}_{\mathcal{V}_{k-1,k}}\ar[r]& 0.
\end{tikzcd}
\end{center}
\noindent Taking determinants in this case yields $(2)$. Then $(3)$ and $(4)$ follow immediately from $(1)$ and $(2)$. 
\end{proof}

\begin{remark}
    Lemma \ref{identities} (1) recovers that $\operatorname{res}^* L \otimes \tau^*N_{k-1,L}(-F_{k-1})$ is globally generated on $X^{[k-1,k]}$, provided that $L$ separates $k$-schemes.
    In fact, it coincides with $\rho^* N_{k,L}$.  
\end{remark}

\begin{lemma}\label{canonical-bundle}
Assume that $X^{[k-1,k]}$ is smooth. We have $\omega_{X^{[k-1,k]}}\cong \rho^* \omega_{X^{[k]}}(F_{k-1})$ and $\omega_{X^{[k]}}\cong T_{k,\omega_X}\otimes \delta_{k}^{n-2}$.
\end{lemma}

\begin{proof}
Note that $\rho^{-1}(X_* ^{[k]})$ has codimension at least 2 by Lemma \ref{rinverseXstar} and that $\rho$ is finite and flat when restricted to this open subset.
Therefore, on this open subset $\rho$ is ramified precisely along $F_{k-1}$ with order 2. Hence we get the first assertion.
For the second, we proceed by induction on $k$, with the case $k=1$ being clear. It suffices to show the isomorphism 
    \[
    \rho^* \omega_{X^{[k]}}\cong \rho^* (T_{k,\omega_X}\otimes \delta_{k}^{n-2}).
    \]
By the description of $X^{[k-1,k]}$ as the blowing up of $X\times X^{[k-1]}$ along $\mathcal{Z}_{k-1}$, we obtain
\[
\omega_{X^{[k-1,k]}}\cong \operatorname{res} ^* \omega_X \otimes \tau^* \omega_{X^{[k-1]}}((n-1)F_{k-1}).
\]
We may therefore conclude by the inductive hypothesis and Lemma \ref{identities}. 
\end{proof}

The following is a generalization of \cite[Lemma 3.5]{Ein.Niu.Park.20}.

\begin{lemma}\label{lem:O_X^[k]directsummand}
Assume that $X^{[k-1,k]}$ is smooth. Then
$\mathcal{O}_{X^{[k]}}$ is a direct summand of $\rho_* \mathcal{O}_{X^{[k-1,k]}}(F_{k-1})$. 
\end{lemma}
\begin{proof}
Immediate from the identification $\omega_{X^{[k-1,k]}/X^{[k]}}\cong \mathcal{O}_{X^{[k-1,k]}}(F_{k-1})$.
\end{proof}

\subsubsection{Sections of secant sheaves} 

The following result of Danila \cite{Dan} will be important for us.
It is stated only for surfaces in \cite{Dan}, but the proof works verbatim if $X^{[k]}$ is smooth.

\begin{theorem}\label{Danila}
 If $X^{[k]}$ is smooth, there is an equivariant isomorphism
    $H^0(X^{[k]}, E_{k,L}^{\otimes \ell})\cong H^0(X,L)^{\otimes \ell}$
for all $\ell\leq k$. In particular, $H^0(X^{[k]}, S^{\ell} E_{k,L})\cong S^{\ell} H^0(X,L)$
for all $\ell\leq k$.
\end{theorem}

\section{The embedding theorems}\label{section-embedding-theorems}

In this section we will establish a precise version of Terracini's lemma, as well as a relative one --- this will require first a brief algebraic digression. After that, we will study what we call the \emph{boundary divisor} on the secant bundle, to be used for inductive purposes in the proofs of our main theorems. We will also be able to conduct a study of the fibers of the morphism from the secant bundle to the secant variety. The geometry here is significantly more complicated than the curve case treated in \cite{Bertram.92} and \cite{Ein.Niu.Park.20}, and for the study of the conormal bundle we will restrict ourselves to the locus of local complete intersection schemes. 

\subsection{Artinian Gorenstein algebras} In this subsection we establish a result (Lemma \ref{pairing-general}) for Noetherian $\mathbb{C}$-algebras which are both Artinian and Gorenstein. Such algebras have received substantial attention in the literature, however we seek only to provide a self-contained argument for a basic fact that may be well-known to experts. See \cite[\S21.1, 21.2]{Eisenbud.95} for further details on the basic theory.\\
\\
Recalling the notation and terminology of the previous chapters, the relevance of Artinian Gorenstein algebras in our context is due to the fact that the differential $d\alpha_k$ may be naturally
expressed in terms of a universal pairing, as will be seen in Subsection \ref{subsection-universal-pairing}.

\begin{definition}
Let $A$ be an Artinian $\mathbb{C}$-algebra. Define
\[
U(A):=\{\varphi\in \operatorname{Hom}_\mathbb{C}(A, \mathbb{C})\mid 
\operatorname{ker}(\varphi) \text{ does not contain any non-zero ideal}\}.
\]
\end{definition}

\noindent Clearly, $U(A)\neq \operatorname{Hom}_\mathbb{C}(A,\mathbb{C})$. To understand $U(A)$ geometrically, note that an element $\varphi\in \operatorname{Hom}_\mathbb{C}(A,\mathbb{C})$ does not belong to $U(A)$ if and only it factors through a surjection $A\lra B$ of $\mathbb{C}$-algebras (which is not an isomorphism). 
If we think of $\operatorname{Hom}_\mathbb{C}(A,\mathbb{C})$ as the linear space spanned by $\eta=\operatorname{Spec}A$, then $\varphi\in \operatorname{Hom}_\mathbb{C}(A,\mathbb{C})\setminus U$ corresponds to a point that is spanned by a proper subscheme $\xi=\operatorname{Spec}B$ of $\eta$.
In particular, we may think of $\operatorname{Hom}_\mathbb{C}(A,\mathbb{C})\setminus U(A)$ as the set of points in the union of spans of smaller schemes $\xi\subset \eta$.

\begin{lemma}\label{pairing-free}
Let $A$ be an Artinian $\mathbb{C}$-algebra. The pairing
\[
A\otimes_\mathbb{C} A \overset{e}{\lra} A \overset{\varphi}{\lra} \mathbb{C}
\]
is non-degenerate if and only if $\varphi\in U(A)$.
\end{lemma}
\begin{proof}
We have that $\varphi\circ e$ is degenerate if and only if there is an element $a\in A$ such that for every $b\in A$ we have
$\varphi(ab)=0$. This happens if and only if $(a)\subseteq \operatorname{ker}(\varphi)$.
\end{proof}

There is an algebraic counterpart to the above geometric interpretation of $U(A)$. To that end, we now recall some ideas concerning the canonical module of $A$.

\begin{remark}
When $A$ is an Artinian $\mathbb{C}$-algebra, its canonical module $\omega_A$ in fact coincides with the vector space $\operatorname{Hom}_\mathbb{C} (A,\mathbb{C})$ equipped with the following $A$-module structure:
\[
a\varphi(b):=\varphi(ab)\text{ for all $a$, $b\in A$ and $\varphi\in\operatorname{Hom}_{\mathbb{C}}(A,\mathbb{C})$}
\]
(this follows, for example, from the case when $A$ is local, which is discussed in \cite[\S21.1]{Eisenbud.95}).
\end{remark}

\begin{definition}
  Let $A$ be an Artinian $\mathbb{C}$-algebra. The algebra $A$ is \emph{Gorenstein} if $\omega_A$ is isomorphic to $A$ as an $A$-module. If $A$ is Gorenstein, then any generator $\varphi\in\omega_A$ is called a \emph{dual generator} --- a choice of such is equivalent to a choice of isomorphism $\omega_A \cong A$.   
\end{definition}

\begin{lemma}\label{U-empty}
    Let $A$ be an Artinian $\mathbb{C}$-algebra. 
    Then $U(A)$ is the set of dual generators.
    In particular, $A$ is Gorenstein if and only if $U(A)\neq \emptyset$.
\end{lemma}
\begin{proof}
    Fix $\varphi\in \operatorname{Hom}_\mathbb{C}(A,\mathbb{C})$ and consider the map
    \[
    f_\varphi \colon A\lra \operatorname{Hom}_\mathbb{C}(A,\mathbb{C})
    \]
    defined by $(f_\varphi (a))(b) :=\varphi(ab)$. Now $\varphi$ is a dual generator if and only if $f_\varphi$ is an isomorphism, and this happens if and only if $\varphi\in U(A)$ by Lemma \ref{pairing-free}.
\end{proof}

\begin{lemma}\label{pairing-general}
Let $A$ be an Artinian $\mathbb{C}$-algebra and let $M$ be a finitely generated $A$-module. Consider the pairing
\[
M\otimes_{\mathbb{C}} \operatorname{Hom}_A(M,A) \overset{e}{\lra} A \overset{\varphi}{\lra} \mathbb{C}.
\]
Then the induced map
\[
M\lra \operatorname{Hom}_\mathbb{C}(\operatorname{Hom}_A(M,A),\mathbb{C})
\]
is surjective for every $\varphi\in U(A)$.
\end{lemma}

\begin{proof}
Take a surjective map $A^s\lra M$. The natural pairing
\[
A^s \otimes_\mathbb{C} \operatorname{Hom}_A(A^s, A)\lra A
\]
induces the commutative diagram
\begin{center}
\begin{tikzcd}
A^s \ar[r]\ar[d] & \operatorname{Hom}_\mathbb{C}(\operatorname{Hom}_A (A^s, A),\mathbb{C})\ar[d]\\
M\ar[r] & \operatorname{Hom}_\mathbb{C}(\operatorname{Hom}_A(M,A),\mathbb{C})
\end{tikzcd}.
\end{center}
The top row is surjective by Lemma \ref{pairing-free} and the right column is surjective by flatness of $A$ over $\mathbb{C}$.
\end{proof}

\subsection{The universal pairing}\label{subsection-universal-pairing}
Here is the connection of the above results with the global geometric setting.

\begin{lemma}\label{pairing-secant}
On $B^{k}$ there is a natural pairing
\[
\pi_{k} ^* \operatorname{pr}_{2,*} (\operatorname{pr}_1 ^* L \otimes \mathcal{I}_{\mathcal{Z}_k}/\mathcal{I}_{\mathcal{Z}_k}^2)
\otimes
\pi_{k}^* \operatorname{pr}_{2,*}\SHom (\mathcal{I}_{\mathcal{Z}_{k}}, \mathcal{O}_{\mathcal{Z}_{k}})
\lra
\mathcal{O}_{B^{k}}(1).
\]
It induces a morphism
\[
\pi_{k} ^* \operatorname{pr}_{2,*} (\operatorname{pr}_1^*L \otimes \mathcal{I}_{\mathcal{Z}_k}/\mathcal{I}_{\mathcal{Z}_k}^2)\otimes \mathcal{O}_{B^{k}}(-1)
\overset{\gls{e-k}}{\lra} \pi_{k}^* \Omega_{X^{[k]}} ^1.
\]
Then the map $e_{k}$ is
surjective when restricted to $U^{k}$.
\end{lemma}
\begin{proof}
The pairing
of Theorem \ref{conormal-cotangent}
on $\mathcal{Z}_k$ induces a pairing
\[
 \operatorname{pr}_{2,*} (\operatorname{pr}_1 ^* L \otimes \mathcal{I}_{\mathcal{Z}_k}/\mathcal{I}_{\mathcal{Z}_k}^2)\otimes \operatorname{pr}_{2,*}\SHom (\mathcal{I}_{\mathcal{Z}_{k}}, \mathcal{O}_{\mathcal{Z}_{k}})\lra E_{k,L}
\]
and a map
\[
D_{k,L} \colon  \operatorname{pr}_{2,*} (\operatorname{pr}_1 ^* L \otimes \mathcal{I}_{\mathcal{Z}_k}/\mathcal{I}_{\mathcal{Z}_k}^2)\lra \Omega_{X^{[k]}}^1 \otimes E_{k,L}
\]
on $X^{[k]}$. These induce the pairing of the statement and the map $e_{k}$ on $B^{k}$ respectively. 
Surjectivity of $e_{k}$ on $U^k$ then follows from Lemma \ref{pairing-general} and Nakayama's Lemma.
\end{proof}

\subsection{Terracini's lemma}
We want to show that, under appropriate conditions, the morphism $\alpha_k \colon B^k\lra \mathbb{P}^r$ restricts to an isomorphism from the open subset $U^k$ onto its image. 
Recall that we also have morphisms $\pi_k \colon B^k \lra X^{[k]}$ and $i_k \colon B^k \lra \mathbb{P}^r\times X^{[k]}$. 

At this point, we introduce the following notation: for schemes $Y\subseteq Z$ we will denote by
$$
\gls{con-x-y} := \mathcal{I}_{Y/Z}/\mathcal{I}_{Y/Z}^2
$$
the \emph{conormal sheaf (of $Y$ in $Z$)}.

\begin{lemma}\label{diagram-differential-secant}
There is a natural commutative diagram
\begin{center}
\begin{tikzcd}
 & & 0\ar[d]& &\\
 & & \mathcal{C}_{B^k/\mathbb{P}^r\times X^{[k]}}\ar[d]\ar[rd,"\gls{gamma-k-L}"]& &\\
 0\ar[r]& \alpha_{k}^* \Omega_{\mathbb{P}^r} ^1\ar[r]\ar[rd, "d\alpha_{k}"']& i_{k} ^* \Omega_{\mathbb{P}^r\times X^{[k]}} ^1\ar[r]\ar[d] & \pi_k ^* \Omega_{X^{[k]}} ^1\ar[r]&0\\
 & & \Omega_{B^{k}} ^1\ar[d]& &\\
 & & ~0& &
\end{tikzcd}.
\end{center}
Furthermore,  $\operatorname{ker}(\gamma_{k}) \cong \operatorname{ker}(d\alpha_{k})$ and  $\operatorname{coker}(\gamma_{k}) \cong \operatorname{coker}(d\alpha_{k})$.
\end{lemma}

This Lemma will now allow us to study the $d\alpha_k$ by studying $\gamma_{k}$ appropriately. The embedding theorem (Theorem \ref{embedding-theorem}) is a special case of the following theorem. 

\begin{theorem}\label{terracini}
Assume that $L$ separates $2k$-schemes.
Then the restricted morphism 
$$
U^k\xrightarrow{~\alpha_k~} \mathbb{P}^r\setminus \kappa_{k-1}
$$
is a closed embedding, i.e., $U^k \cong \kappa_k \setminus \kappa_{k-1}$. Furthermore, if $X^{[k]}$ is smooth, then $\mathcal{C}_{U^k/\mathbb{P}^r}\cong \operatorname{ker}(\gamma_{k})$.
\end{theorem}

\begin{proof}
The morphism $\alpha_k |_{U^k}$ is injective since $L$ separates $2k$-schemes.
By \S \ref{conormal-sheaf-of-sequence}, we have
\[
\mathcal{C}_{B^k/\mathbb{P}^r\times X^{[k]}} \cong
\pi_k^* M_{k,L}\otimes\mathcal{O}_{B^k}(-1).
\]
The induced morphism $\gamma_{k}$
factors through $e_{k}$
by Theorem \ref{conormal-cotangent}.
By Lemma \ref{pairing-secant} we have that $\gamma_{k}$ is surjective on $U^k$, and therefore $d\alpha_k$ is surjective on $U^k$ by Lemma \ref{diagram-differential-secant}. If $X^{[k]}$ is smooth, then so is $U^k$, and therefore $\mathcal{C}_{U^k/\mathbb{P}^r}\cong \operatorname{ker}(d\alpha_k\vert_{U^k})$.
\end{proof}

\begin{remark}\label{terracini-remark}
The above result can be understood as a ``precise Terracini Lemma'' for the following reason. Recall that the classical statement \cite[Lemma 3.4.28]{pag1} says that for points $x_1,\ldots,x_k \in X \subseteq \mathbb{P}(V)$ and $y \in \textnormal{Span}_{\mathbb{P}(V)}(x_1,\ldots,x_k)$, the projective tangent space of the secant variety at $y\in\sigma_k(X)$ has a simple linear algebraic description:
$$
\mathbb{T}_y(\sigma_k) \supseteq \textnormal{Span}_{\mathbb{P}(V)}(\mathbb{T}_{x_1}X,\ldots,\mathbb{T}_{x_k}X)
$$
as subspaces of $\mathbb{P}(V)$, with equality when $y \in \sigma_k(X)$ is general.
Noting that for subspaces $\mathbb{P}(W),\mathbb{P}(U) \subseteq \mathbb{P}(V)$ the span is identified as
$$
\textnormal{Span}_{\mathbb{P}(V)}(\mathbb{P}(U),\mathbb{P}(W)) = \mathbb{P}(V/(\textnormal{ker}(V\lra U)\cap\textnormal{ker}(V\lra W)))
$$
and that for any $x \in X$ we have
$$
\mathbb{T}_xX = \mathbb{P}(\textnormal{im}(V \lra H^0(\OO_X(1)\otimes \OO_X/\mathfrak{m}_x^2))).
$$
Therefore, for general $y \in \sigma_k$, we have
\begin{align*}
\mathbb{T}_y(\sigma_k) & = \mathbb{P}(V/\cap_{i=1}^kH^0(\mathbb{P}V,\mathfrak{m}^2_{x_i/\mathbb{P}V})).
\end{align*}
When $L$ separates $2k$-schemes, this means
$$
\mathbb{T}_y(\sigma_k) = \mathbb{P}(H^0(X,L\otimes \OO_X/\mathcal{I}_{\eta}^2))
$$
for $\eta = x_1+\cdots +x_k$ and $\mathcal{I}_{\eta} = \mathcal{I}_{\eta/X}$, noting that $\mathcal{I}_{\eta}^j \cong \otimes_{i=1}^k\mathfrak{m}_i^j$ since the $x_i$ are distinct.
In our case, if $y \in \sigma_k\setminus \sigma_{k-1}$ is arbitrary and $X^{[k]}$ is assumed to be smooth, then for any $[\eta] \in X^{[k]}$ for which $y \in \textnormal{Span}_{\mathbb{P}(V)}(\eta)$, our Theorem \ref{terracini} shows that the conormal space of $\sigma_k(L)$ at $y$ is 
\begin{equation}\label{terracini-kernel}
\operatorname{ker}\left( H^0(X, L\otimes \mathcal{I}_\eta) \lra \operatorname{Hom}_{\mathbb{C}}(\operatorname{Hom}_{\mathcal{O}_\eta}(L\otimes \mathcal{I}_\eta /\mathcal{I}_\eta^2, \mathcal{O}_\eta), \mathbb{C})\right),
\end{equation}
which contains $H^0(X,L\otimes \mathcal{I}_{\eta}^2)$ but is only equal to it for $[\eta] \in X_{\operatorname{lci}}^{[k]}$. Hence for general $y \in \sigma_k$ the two descriptions coincide.
\end{remark}

\begin{corollary}\label{corollary-singularities-cactus}
Assume that $L$ separates $2k$-schemes. Then we have
$$
\operatorname{Sing}(\kappa_k) = \kappa_{k-1} \cup \alpha_k(\pi_k^{-1}(\operatorname{Sing}(X^{[k]}_{\operatorname{Gor}}))).
$$
\end{corollary}

\begin{proof}
By Theorem \ref{terracini}, we only need to show that $\kappa_{k-1} \subseteq \operatorname{Sing}(\kappa_k)$. We prove that $\mathbb{T}_y(\kappa_k) = \mathbb{P}^r$ for any $y \in \kappa_{k-1}$. By definition, there is $[\eta] \in X^{[k-1]}$ with $y \in \alpha_k (\pi_k ^{-1}(\eta))$. For general $x \in X$, the subspace $\operatorname{Span}_{\mathbb{P}^r}(\alpha_k (\pi_k ^{-1}(\eta)), x)$ is contained in $\kappa_k$. As $X \subseteq \mathbb{P}^r$ is non-degenerate, this yields $\mathbb{T}_y(\kappa_k) = \mathbb{P}^r$.
\end{proof}

The following was observed in \cite[Lemma 3.5]{BGL} and \cite[Lemma 2.4]{BB}.

\begin{lemma}\label{lem:pi^{-1}(Gorensteinlocus)}
For $[\xi] \in X^{[k]}$, we have that $\pi_k^{-1}([\xi]) \subseteq \kappa_{k-1,k}$ if and only if $[\xi] \not\in X^{[k]}_{\text{Gor}}$. In particular, 
$\pi_k^{-1}(X^{[k]}\setminus X^{[k]}_{\text{Gor}}) \subseteq \kappa_{k-1,k}$ and $U^k \subseteq \pi^{-1}(X^{[k]}_{\operatorname{Gor}})$.
\end{lemma}

\begin{proof}
This is an immediate consequence of Lemma \ref{U-empty}.
\end{proof}

By Lemma \ref{lem:pi^{-1}(Gorensteinlocus)}, we have $\sigma_k=\kappa_k$ as a set if and only if $X^{[k]}_{\textnormal{Gor}}$ is irreducible --- this is the case when $k \leq 13$ or $k=14$ and $\dim(X) \leq 5$ (see \S \ref{subsubsection-smoothness}).
The following is Corollary \ref{cor:singularlocusofSigma}.

\begin{corollary}\label{corollary-singularities-secant}
Assume that $L$ separates $2k$-schemes. Then we have
\begin{equation}\label{eq:singular-locus}
\sigma_{k-1} \cup \alpha_k(\pi_k^{-1}(\operatorname{Sing}(X^{[k]}_{\operatorname{sm}, \operatorname{Gor}}))) \subseteq \operatorname{Sing}(\sigma_k) \subseteq (\kappa_{k-1} \cap \sigma_k) \cup \alpha_k(\pi_k^{-1}(\operatorname{Sing}(X^{[k]}_{\operatorname{sm}, \operatorname{Gor}}))).
\end{equation}
In particular, the following hold:
\begin{enumerate}[topsep=0pt]
 \item If $k \leq 14$ or $k=15$ and $\dim(X) \leq 5$, then $\operatorname{Sing}(\sigma_k) = \sigma_{k-1} \cup \alpha_k(\pi_k^{-1}(\operatorname{Sing}(X^{[k]}_{\operatorname{sm}, \operatorname{Gor}})))$.
 \item $\Sing (\sigma_k) = \sigma_{k-1}$ if and only if $\dim (X) \leq 3$ or $k \leq 5$.
\end{enumerate}
\end{corollary}

\begin{proof}
Immediate by Corollary \ref{corollary-singularities-cactus}, Lemma \ref{lem:pi^{-1}(Gorensteinlocus)}, and the discussion in Subsection \ref{subsubsection-smoothness}.
\end{proof}

\begin{remark}\label{rem:singular-locus}
When $k$ and $\dim(X)$ are large enough, we have a strict inclusion $\sigma_{k-1} \subsetneq \kappa_{k-1} \cap \sigma_k$. This happens because there exist non-smoothable schemes $[\eta] \in X^{[k-1]}$ and smoothable schemes $[\xi] \in X^{[k]}_{\sm}$ with $\eta \subseteq \xi$. In this case, both inclusions in (\ref{eq:singular-locus}) may be strict. We do not know which points of $(\kappa_{k-1} \cap \sigma_k) \setminus \sigma_{k-1}$ belong to the singular locus of $\sigma_k$.
\end{remark}

\subsection{Relative Terracini lemma}

Here we want to generalize Theorem \ref{terracini} to the relative setting. Our aim is to show that the morphism $\alpha_{m,k} \colon B^{m,k}\lra B^k$ restricts to an isomorphism from the open subset $U^{m,k}$ onto its image.

\begin{lemma}\label{ideal-relative-secant-bundle}
The conormal sheaf
\[
\mathcal{C}_{B^{m,k}/P^{m,k}}=\pi_{m, k}^* M_{m,k,L} \otimes \mathcal{O}_{B^{m,k}}(-1).
\]
is a locally free sheaf of rank $k-m$. Furthermore, there is a short exact sequence
\begin{multline*}
0 \lra   \pi_{m, k}^* (\pr_{2,*} (\pr_1^*L \otimes \mathcal{O}_{\mathcal{V}_{m,k}}\otimes \mathcal{I}_{\mathcal{W}_{m,k}}))\otimes \mathcal{O}_{B^{m,k}}(-1) \\
\lra 
\pi_{m,k}^* (\pr_{2,*} (\pr_1^*L \otimes \mathcal{O}_{\mathcal{W}_{m,k}} \otimes \mathcal{I}_{\mathcal{V}_{m,k}}))\otimes \mathcal{O}_{B^{m,k}}(-1) \lra \mathcal{C}_{B^{m,k}/P^{m,k}} \lra 0.
\end{multline*}
\end{lemma}

\begin{proof}
The first part of the statement follows from \S \ref{conormal-sheaf-of-sequence}. The second part of the statement follows from applying $\pi_{m,k}^*(\pr_{2,*}(\pr_1^*L\otimes -))$ to the short exact sequence
\[
0\lra \frac{\mathcal{I}_{\mathcal{W}_{m,k}}}{\mathcal{I}_{\mathcal{V}_{m,k}}\cdot \mathcal{I}_{\mathcal{W}_{m,k}}}
    \lra \frac{\mathcal{I}_{\mathcal{V}_{m,k}}}{\mathcal{I}_{\mathcal{V}_{m,k}}\cdot \mathcal{I}_{\mathcal{W}_{m,k}}}
    \lra \frac{\mathcal{I}_{\mathcal{V}_{m,k}}}{\mathcal{I}_{\mathcal{W}_{m,k}}}\lra 0. \qedhere
\]
\end{proof}

\begin{lemma}\label{relative-secant-row}
Let $\gls{Q-m-k}:=\textnormal{coker}(\rho^* \Omega_{X^{[k]}} ^1\lra \Omega_{X^{[m,k]}} ^1)$. Then there is an exact sequence
\[
\alpha_{m,k} ^* \Omega_{B^k} ^1\lra
i_{m,k} ^* \Omega_{P^{m,k}} ^1
\lra \pi_{m, k} ^* Q_{m,k}\lra 0.
\]
If furthermore $X^{[m,k]}$ is smooth, then the above sequence is also exact on the left.
\end{lemma}

\begin{proof}
Consider the following commutative diagram:
\begin{center}
    \begin{tikzcd}
        B^{m,k}\ar[d,"i_{m,k}"] \ar[dr,"\pi_{m,k}"]\\
        P^{m,k} \ar[r,"\pr_2"'] \ar[d,"\widetilde{\rho}"] & X^{[m,k]} \ar[d,"\rho"]\\
        B^k \ar[r,"\pi_k"] & X^{[k]}
    \end{tikzcd}.
\end{center}
Using the notation from this diagram, since the square is Cartesian and we know that
$$
\Omega_{P^{m,k}/X^{[k]}}^1 \cong \widetilde{\rho}^*\Omega_{B^k/X^{[k]}}^1 \oplus \pr_2^*Q_{m,k},
$$
we apply the Snake Lemma to the following commutative diagram (with $\mathscr{K}$ denoting the kernel of the right vertical column)
\begin{center}
    \begin{tikzcd}
        & 0\ar[d] & 0\ar[d]\\& (\rho \circ\pr_2)^*\Omega_{X^{[k]}}^1 \ar[r,twoheadrightarrow] \ar[d] & \mathscr{K}\ar[d]\\
        & \widetilde{\rho}^*\Omega_{B^k}^1 \ar[r]\ar[d] & \Omega_{P^{m,k}}^1\ar[d]\\
        0 \ar[r] & \widetilde{\rho}^*\Omega_{B^k/X^{[k]}}^1 \ar[r]\ar[d] & \Omega_{P^{m,k}/X^{[k]}}^1 \ar[r] \ar[d] & \pr_2^*Q_{m,k} \ar[r] & 0\\
        & 0 & 0
    \end{tikzcd}
\end{center}
and thus obtain that
\begin{align*}
\pr_2^*Q_{m,k} & \cong \textnormal{coker}(\widetilde{\rho}^*\Omega_{B^k}^1 \lra \Omega_{P^{m,k}}^1).
\end{align*}
The first statement of the lemma then follows from the commutativity of the diagram. The second is clear because when $X^{[m,k]}$ is smooth then $P^{m,k}$ is also smooth.
\end{proof}

\begin{lemma}\label{relative-secant-column}
For $Q_{m,k}$ as in Lemma \ref{relative-secant-row}, there is an exact sequence on $X^{[m,k]}$:
\[
\SExt ^n _{\operatorname{pr}_2} (\mathcal{O}_{\mathcal{V}_{m,k}},\mathcal{I}_{\mathcal{W}_{m,k}} \otimes \operatorname{\omega}_{\operatorname{pr}_2})
\lra \tau^* \Omega_{X^{[m]}} ^1 \lra
 Q_{m,k} \lra 0.
\]
\end{lemma}

\begin{proof}
Theorem \ref{cotangent-sheaf-nested-hilbert-scheme} gives the middle row of the following diagram, which is commutative:
\begin{center}
\begin{tikzcd}
& 0\ar[d] \\
& \rho^*\Omega_{X^{[k]}}^1 \ar[d]\\
\SExt^n_{\pr_2}(\mathcal{O}_{\mathcal{V}_{m,k}},\mathcal{I}_{\mathcal{W}_{m,k}}) \otimes \operatorname{\omega}_{\operatorname{pr}_2} \ar[r] & \Omega_{X^{[m]}}^1\boxplus\Omega_{X^{[k]}}^1 \ar[r]\ar[d] & \Omega_{X^{[m,k]}}^1 \ar[r]\ar[d] & 0\\
 & \tau^* \Omega_{X^{[m]}} ^1 \ar[r]\ar[d] &
Q_{m,k} \ar[r] \ar[d]& 0\\
& 0 & 0
\end{tikzcd}.
\end{center}
By the definition of $Q_{m,k}$, there is a surjection from $\rho^*\Omega_{X^{[k]}}^1$ to the kernel of the right vertical map. Hence, by the Snake Lemma, $\SExt^n_{\pr_2}(\OO_{\mathcal{V}_{m,k}},\mathcal{I}_{\mathcal{W}_{m,k}}) \otimes \operatorname{\omega}_{\operatorname{pr}_2}$ surjects to the kernel of the lower horizontal map. 
\end{proof}

\begin{lemma}\label{diagram-differential-relative-secant}
There is a natural commutative diagram
\begin{center}
\begin{tikzcd}
  & 0\ar[d]& &\\
 & \mathcal{C}_{B^{m,k}/P^{m,k}}\ar[d]\ar[rd,"\gls{gamma-m-k-L}"]& &\\
  \alpha_{m,k}^* \Omega_{B^k} ^1\ar[r]\ar[rd, "d\alpha_{m,k}"']& i_{m,k} ^* \Omega_{P^{m,k}} ^1\ar[r]\ar[d] & \pi_{m, k}^* Q_{m,k}\ar[r]&0\\
  & \Omega_{B^{k,m}} ^1\ar[d]& &\\
  & 0& &
\end{tikzcd}.
\end{center}
We have that $\operatorname{ker}(d\alpha_{m,k})$ surjects onto $\operatorname{ker}(\gamma_{m,k})$ and $\operatorname{coker}(\gamma_{m,k})$ surjects onto $\operatorname{coker}(d\alpha_{m,k})$.
If furthermore $X^{[m,k]}$ is smooth, then the horizontal sequence is also exact on the left,
$\operatorname{ker}(\gamma_{m,k}) \cong \operatorname{ker}(d\alpha_{m,k})$, and
$\operatorname{coker}(\gamma_{m,k}) \cong \operatorname{coker}(d\alpha_{m,k})$.
\end{lemma}

\begin{proof}
The middle column is the natural sequence of cotangent sheaves, whereas the middle row is Lemma \ref{relative-secant-row}. Otherwise, the reasoning is exactly as in Lemma \ref{diagram-differential-secant}.
\end{proof}

\begin{lemma}\label{pairing-relative-secant}
On $B^{m,k}$ there is a natural pairing
\[
\pi_{m,k} ^* (\operatorname{pr}_{2,*}(\operatorname{pr}_1 ^* L \otimes \mathcal{O}_{\mathcal{V}_{m,k}} \otimes \mathcal{I}_{\mathcal{W}_{m,k}}))
\otimes
\pi_{m,k}^* (\operatorname{pr}_{2,*}\SHom (\mathcal{I}_{\mathcal{W}_{m,k}}, \mathcal{O}_{\mathcal{V}_{m,k}}))
\lra
\mathcal{O}_{B^{m,k}}(1).
\]
It induces a map
\[
\pi_{m,k} ^* (\operatorname{pr}_{2,*}(\operatorname{pr}_1 ^* L \otimes \mathcal{O}_{\mathcal{V}_{m,k}} \otimes \mathcal{I}_{\mathcal{W}_{m,k}}))\otimes \mathcal{O}_{B^{m,k}(L)}(-1)
\overset{e_{m,k}}{\lra} \pi_{m,k}^* ( \SExt _{\operatorname{pr}_2} ^n (\mathcal{O}_{\mathcal{V}_{m,k}}, \mathcal{I}_{\mathcal{W}_{m,k}}) \otimes \operatorname{\omega}_{\operatorname{pr}_2}).
\]
The map $e_{m,k}$ is surjective when restricted to $U^{m,k}$.
\end{lemma}

\begin{proof}
There is a natural pairing
\[
(\operatorname{pr}_1 ^* L \otimes \mathcal{O}_{\mathcal{V}_{m,k}} \otimes \mathcal{I}_{\mathcal{W}_{m,k}})
\otimes
\SHom (\mathcal{I}_{\mathcal{W}_{m,k}}, \mathcal{O}_{\mathcal{V}_{m,k}})
\lra
\operatorname{pr}_1 ^* L \otimes  \mathcal{O}_{\mathcal{V}_{m,k}}
\]
on $X\times X^{[m,k]}$. It induces the desired pairing on $B^{m,k}$. Surjectivity of $e_{m,k}$ follows from Lemma \ref{pairing-general}.
\end{proof}

\begin{lemma}\label{diagram-relative-evaluation}
On $B^{m,k}$ there is a natural commutative diagram
\begin{center}
\begin{tikzcd}
0\ar[d] & \\
\pi_{m,k} ^* (\operatorname{pr}_{2,*}(\operatorname{pr}_1 ^* L \otimes \mathcal{O}_{\mathcal{V}_{m,k}} \otimes \mathcal{I}_{\mathcal{W}_{m,k}}))\otimes \mathcal{O}(-1)\ar[d]\ar[r, "e_{m,k}"]&\pi_{m,k}^* (\SExt _{\operatorname{pr}_2} ^n (\mathcal{O}_{\mathcal{V}_{m,k}}, \mathcal{I}_{\mathcal{W}_{m,k}}) \otimes \operatorname{\omega}_{\operatorname{pr}_2})
\ar[d]\\ 
\pi_{m, k} ^* (\operatorname{pr}_{2, *} (\operatorname{pr}_1 ^* L \otimes \mathcal{O}_{\mathcal{W}_{m,k}} \otimes \mathcal{I}_{\mathcal{V}_{m,k}}))\otimes \mathcal{O}(-1)\ar[d]\ar[r, "\widetilde{e}_{m,m}"]& \pi_{m,k} ^* \tau ^*  \Omega_{X^{[m]}} ^1
\ar[d]\\
\mathcal{C}_{B^{m,k}/P^{m,k}}\ar[d]\ar[r, "\gamma_{m,k}"]& 
\pi_{m,k}^* Q_{m,k} \ar[d]\\
0 & 0\\
\end{tikzcd}.
\end{center}
\end{lemma}

\begin{proof}
The left column is Lemma \ref{ideal-relative-secant-bundle}, whereas the right column is the pullback of the sequence in Lemma \ref{relative-secant-column} along $\pi_{m,k}$.
The maps $e_{m,k}$ and $\gamma_{m,k}$ are given by Lemma \ref{pairing-relative-secant} and Lemma \ref{diagram-differential-relative-secant} respectively. 
The map $\widetilde{e}_{m,m}$ is induced by the composition
of
\[
\pi_{m, k} ^* (\operatorname{pr}_{2, *} (\operatorname{pr}_1 ^* L \otimes \mathcal{O}_{\mathcal{W}_{m,k}}\otimes \mathcal{I}_{\mathcal{V}_{m,k}}))\otimes \mathcal{O}(-1)
\overset{\mu_{m,k}}{\lra}
\pi_{m, k} ^* (\operatorname{pr}_{2, *} (\operatorname{pr}_1 ^* L \otimes \mathcal{O}_{\mathcal{V}_{m,k}}\otimes \mathcal{I}_{\mathcal{V}_{m,k}}))\otimes \mathcal{O}(-1)
\]
and the pullback of
\[
\pi_{m} ^* (\operatorname{pr}_{2, *} (\operatorname{pr}_1 ^* L \otimes \mathcal{O}_{\mathcal{V}_{m}}\otimes \mathcal{I}_{\mathcal{V}_{m}}))\otimes \mathcal{O}(-1)
\overset{e_{m}}{\lra}
\pi_{m}^* \Omega_{X^{[m]}}^1 . \qedhere
\]
\end{proof}

\begin{proposition}\label{fiber-product}
The homomorphism $d\alpha_{m,k} \colon \alpha_{m,k} ^* \Omega_{B^k} ^1\lra \Omega_{B^{k,m}}^1$ is surjective when restricted to $U^{m,k}$. Moreover, if $L$ separates $2k$-schemes, then the commutative diagram
\begin{center}
\begin{tikzcd}
U^{m,k}\ar[r, "\alpha_{m,k}"]\ar[d, "\widetilde{\tau}"]& B^k\ar[d,"\alpha_k"]\\
U^m\ar[r, "\alpha_m"]& \mathbb{P}^r
\end{tikzcd}
\end{center}
is Cartesian.
\end{proposition}
\begin{proof}
Consider the diagram given by Lemma \ref{diagram-relative-evaluation}.
By Lemma \ref{pairing-secant} the map $e_m$ (and thus $\widetilde{e}_{m,m}$) is surjective. So then the first part of the Proposition follows from Lemma \ref{diagram-differential-relative-secant} together with the fact that $\textnormal{coker}(\widetilde{e}_{m,m}) \lra \textnormal{coker}(\gamma_{m,k})$ is surjective.
For the second part, we need to show that the inclusion $\mathcal{I}_{U^m\times_{\mathbb{P}^r} B^k/B^k}\subseteq \mathcal{I}_{U^{m,k}/B^k}$ is an isomorphism. By Theorem \ref{terracini} and Nakayama's lemma, it suffices to show that the natural map
\[
\Psi \colon \widetilde{\tau} ^* \mathcal{C}_{U^m/\mathbb{P}^r}
\lra 
\mathcal{C}_{U^{m,k}/B^k}
\] 
is surjective. 
To see this, note first that (by part (1) of Lemma \ref{diagram-differential-secant}), we have that $\widetilde{\tau}^*\mathcal{C}_{U^m/\mathbb{P}^r}$ may be identified with
$$
\textnormal{ker}\left(\pi_{m,k}^*(\pr_{2,*}(\pr_1^*L\otimes \mathcal{I}_{\mathcal{V}_{m,k}})) \otimes \OO(-1) \overset{P}{\lra} \pi_{m,k}^*(\pr_{2,*}(\pr_1^*L\otimes\mathcal{O}_{\mathcal{W}_{m,k}}\otimes \mathcal{I}_{\mathcal{V}_{m,k}})) \otimes \OO(-1) \overset{\widetilde{e}_{m,m}}\lra \pi_{m,k}^*\tau^*\Omega_{X^{[m]}}^1\right).
$$ 
Moreover, in the above, the homomorphism $P$ is surjective by the positivity assumptions on $L$. Given that $\mathcal{C}_{U^{m,k}/B^k} \cong \textnormal{ker}(\gamma_{m,k})$, this means that the map $\Psi$ factors through a surjection onto $\textnormal{ker}(\widetilde{e}_{m,m})$. But the latter surjects to $\textnormal{ker}(\gamma_{m,k})$ by the Snake lemma applied to the diagram in Lemma \ref{diagram-relative-evaluation}.
\end{proof}

In the following proposition, we prove the relative version of Theorem \ref{terracini}.

\begin{theorem}\label{relative-terracini}
    Assume $\dim (X)\leq 2$ or $k\leq 3$. Then, recalling that $F_{k-1}$ is the exceptional divisor of $X^{[k-1,k]} \cong \textnormal{Bl}_{\mathcal{Z}_{k-1}}X\times X^{[k-1]}$, we have
    \[
    \mathcal{C}_{U^{k-1,k}/B^k} \cong \pi_{k-1,k}^*(
    \operatorname{res}^* L(-2F_{k-1}))\otimes \mathcal{O}_{B^{k-1,k}}(-1)\vert_{U^{k-1,k}}.
    \]
\end{theorem}
\begin{proof}
Under the stated conditions $X^{[k-1,k]}$ is smooth.
By reasoning analogous to that of the previous Proposition, we have that $\operatorname{ker}(\tilde{e}_{k-1,k-1})\lra \operatorname{ker}(\gamma_{k-1,k})$
is surjective.
Now note that the right-hand side of the statement of the present Proposition is 
$\operatorname{ker}(\mu_{k-1,k})$. On the open subscheme  $\rho^{-1}(X_{\operatorname{lci}}^{[k]}) \subseteq X^{[k-1,k]}$, the pullback of the homomorphism $e_{k-1}$ along $\pi_{k-1,k}$ is an isomorphism, and thus $\textnormal{ker}(\widetilde{e}_{k-1,k-1}) = \textnormal{ker}(\mu_{k-1,k})$. But $\textnormal{ker}(\mu_{k-1,k})$ is a line bundle which surjects to $\textnormal{ker}(\gamma_{k-1,k})$ which is also a line bundle --- thus these two are isomorphic on this open subscheme $\rho^{-1}(X_{\operatorname{lci}}^{[k]}) \subseteq X^{[k-1,k]}$. Since the complement of this open subscheme has codimension 2 and $U^{k-1,k}$ is smooth (hence in particular normal), this means that
$$
\textnormal{ker}(\widetilde{e}_{k-1,k-1}) \cong \textnormal{ker}(\gamma_{k-1,k})
$$
over the whole scheme $X^{[k-1,k]}$, which is the desired result.
\end{proof}

\subsection{The conormal bundle of the fibers of $\alpha_k$}
Here we assume that $L$ separates $2k$-schemes and we study the fibers of $\alpha_k$ along the lines of \cite[Proposition 3.13 (2.d)]{Ein.Niu.Park.20} (see also \cite{Bertram.92} and \cite{Ullery.16}). In higher dimensions, the situation is much more subtle, and we are only able to obtain a partial description. Although this subsection won't be needed for the main results of the paper, we present it here for completeness and for the application in Proposition \ref{not-klt-lc}.

We assume that $\operatorname{dim}(X) \leq 2$ or $k\leq 3$, so that $X^{[k]}$ and $X^{[k-1,k]}$ are smooth. 
Consider the following commutative diagram
\begin{center}
\begin{tikzcd}
F_{\eta} \ar[r,symbol=\subseteq] \ar[dr] & F_{k-1} \ar[r,symbol=\subseteq] & X^{[k-1,k]} \ar[d,"\tau"] & U^{k-1,k} \ar[l,"\pi_{k-1,k}"'] \ar[d,"\widetilde{\tau}"] \ar[r,symbol=\subseteq] & B^k\ar[d,"\alpha_k"]\\
& \left[\eta\right] \ar[r,symbol=\in] & X^{[k-1]} & U^{k-1} \ar[l,"\pi_{k-1}"] \ar[r,symbol=\subseteq] & ~\sigma_k
\end{tikzcd},
\end{center}
where $F_{k-1}$ is the exceptional divisor of $\operatorname{bl}$, and for some $k$-scheme $\eta \subseteq X$ we define
\begin{align*}
F_{\eta} & := F_{k-1}\cap\tau^{-1}([\eta]).
\end{align*}
Note that the square in the middle of the above diagram is Cartesian (by definition), and therefore for any choice of $x \in U^{k-1}$ we have $\tau^{-1}(\pi_{k-1}(x)) \cong \widetilde{\tau}^{-1}(x)$. 

\begin{lemma}\label{fiber-conormal-decomposition}
Pick $x\in \pi_{k-1}^{-1}(X^{[k-1]}_{\textnormal{lci}}) \subseteq U^{k-1}$ and set $[\eta] := \pi_{k-1}(x)$. Then
    \[
    \mathcal{C}_{\widetilde{\tau}^{-1}(x)/ B^k}
    \cong
    \left(\mathcal{O}_{\widetilde{\tau}^{-1}(x)} ^{\oplus (k-1)\operatorname{dim}(X) + k - 2}\right)
    \oplus \left(\operatorname{res}^* L \otimes \mathcal{O}_{\tau ^{-1}(\eta)}(-2F_{\eta})\right).
    \]
\end{lemma}
\begin{proof}
By Proposition \ref{fiber-product} and the fact that $\tau$ is flat on $\tau^{-1}(X^{[k-1]}_{\textnormal{lci}})$, we have
\[
\mathcal{C}_{\widetilde{\tau}^{-1}(x)/U^{k-1,k}} \cong \Omega_{\sigma_{k-1}}^1\vert_x
\otimes \mathcal{O}_{\widetilde{\tau}^{-1}(x)}
\cong \mathcal{O}_{\widetilde{\tau}^{-1}(x)} ^{\oplus (k-1)\operatorname{dim} X + k - 2}
\]
since $U^{k-1,k}$ is smooth at $x$ and has dimension $k\cdot \operatorname{dim} (X) + k - 2$.
Consider the short exact sequence
\[
0\lra \mathcal{C}_{U^{k-1,k}/B^k}|_{\widetilde{\tau}^{-1}(x)}\lra
\mathcal{C}_{\widetilde{\tau}^{-1}(x)/ B^k}\lra
\mathcal{C}_{\widetilde{\tau}^{-1}(x)/U^{k-1,k}}\lra 0.
\]
Since the composition
\[
\Omega^1_{\mathbb{P}^r}\vert_x\lra H^0(\widetilde{\tau}^{-1}(x), \mathcal{C}_{\widetilde{\tau}^{-1}(x)/ B^k})
\lra H^0(\widetilde{\tau}^{-1}(x), \mathcal{C}_{\widetilde{\tau}^{-1}(x)/U^{k-1,k}})=\Omega^1_{\sigma_{k-1}}\vert_x
\]
is surjective, we have that the above sequence splits. Hence we are done by Proposition \ref{relative-terracini}.
\end{proof}

\begin{remark}
Pick $x \in U^{k-1}$ and set $[\eta] := \pi_{k-1}(x) \in X^{[k]}$. One may prove that if $\eta$ is local complete intersection, then the natural morphism
   \[
   \Omega^1_{\mathbb{P}^r}\vert_x\lra 
   H^0(\widetilde{\tau}^{-1}(x), \mathcal{C}_{\widetilde{\tau}^{-1}(x)/ B^k})
   \]
is an isomorphism (see \cite[Proposition 3.13 (2.e)]{Ein.Niu.Park.20} and \cite[Lemma 2.4]{Ullery.16}). This identification was used in \cite{Ein.Niu.Park.20} and \cite{Ullery.16} to study the local structure of $\sigma_k$ by applying the formal function theorem.
\end{remark}

\subsection{The boundary divisor $Z_k$ on $B^k$}\label{subsection-boundary-divisor}
Assume that $X^{[k]}$ and $X^{[k-1,k]}$ are smooth. Recall that the secant bundle $B^k$ has a natural boundary divisor, the exceptional divisor $\gls{z-k}$ of $\alpha_k$, which coincides with the relative secant variety $\sigma_{k-1,k}=\alpha_{k-1,k}(B^{k-1,k})\subseteq B^k$ by Theorem \ref{terracini}. The purpose of this subsection is to study $Z_k$. Given that $\alpha_k(Z_{k})=\sigma_{k-1}$, this study will be useful later on for inductive arguments. 
From now on, $\gls{h-k} \in \lvert \mathcal{O}_{B^k}(1) \rvert$ is a tautological divisor on $B^k$. The following is a counterpart to \cite[Proposition 3.15]{Ein.Niu.Park.20}.

\begin{proposition}\label{A&Z}
Suppose that $X^{[k]}$ is normal and Cohen--Macaulay, and suppose that $X_{\operatorname{Gor}}^{[k]}$ is smooth. Then the following hold:
\begin{enumerate}[topsep=0pt]
    \item $Z_{k}$ is Cohen--Macaulay; it is furthermore flat over $X_{\operatorname{Gor}}^{[k]}$.
    \item $\pi_{k,*}\mathcal{O}_{B^k} (j H_k - Z_{k}) = 0$ for $j<k$.
    \item $\mathcal{O}_{B^k} (k H_k - Z_{k})\cong\pi_k ^* A_{k,L}$.
\end{enumerate}
\end{proposition}

\begin{proof}
$(1)$ The morphism $\alpha_{k-1,k} \colon B^{k-1,k}\lra Z_{k}$ is birational. Therefore $Z_{k}$ is an irreducible divisor in $B^k$, and thus Cohen--Macaulay. Since the fibers of $Z_{k}$ have constant dimension $k-2$ over $X_{\operatorname{Gor}}^{[k]}$, we are done by \cite[Theorem 18.16]{Eisenbud.95}.

$(2)$ Take a general $k$-scheme $\eta\subseteq X$. Since $\eta$ is general, it is reduced. The divisor $Z_{k,\eta}\subseteq B^k_\eta\cong \mathbb{P}^{k-1}$ 
consists of the union of $k$ hyperplanes, hence has degree $k$. Since $Z_{k}$ is flat over $X_{\operatorname{Gor}}^{[k]}$, this holds true for all $[\eta] \in X_{\operatorname{Gor}}^{[k]}$. 
Therefore we have
\[
h^0(B^k, \mathcal{O}_{B^k}(jH_k -Z_{k})\vert_\eta)=0
\]
for all $[\eta] \in X_{\operatorname{Gor}}^{[k]}$ and $j<k$. Since $\operatorname{codim}(X^{[k]}\setminus X_{\operatorname{Gor}}^{[k]})\geq 2$, we have that
\[
\pi_{k,*} \mathcal{O}_{B^k} (j H_k - Z_{k})=0.
\]

$(3)$ By the same reasoning as $(2)$, we see that 
\[
A_k:=\pi_{k,*} \mathcal{O}_{B^k} (k H_k - Z_{k})
\]
is a line bundle. 
We want to show $A_k\cong A_{k,L}$. Recall that $\alpha_{k-1,k}^{-1} (\sigma_{k-2,k}) = \widetilde{\tau}^{-1} (Z_{k-1})$.
By Proposition \ref{fiber-product} we may therefore write
\[
K_{B^{k-1,k}}=\alpha_{k-1,k}^* K_{Z_{k}}
+\widetilde{\tau} ^*(c\cdot Z_{k-1})
\]
for some integer $c$. 
By adjunction, we have 
$$
K_{Z_{k}} = (K_{B^k} + Z_{k})|_{Z_{k}} = \pi_k^*(K_{X^{[k]}}+N_{k,L} - A_{k}).
$$ 
Therefore we get
\begin{equation}\label{eq:computeZ_k1}
K_{B^{k-1,k}}=\pi_{k-1,k}^* \rho^* (K_{X^{[k]}}+N_{k,L} - A_{k})
+\widetilde{\tau} ^*(c\cdot Z_{k-1}).
\end{equation}
As $B^{k-1,k}$ is a projective bundle over $X^{[k-1,k]}$, we have
\[
K_{B^{k-1,k}}=\pi_{k-1,k}^*(K_{X^{[k-1,k]}}+\tau^* N_{k-1,L})-(k-1)\widetilde{\tau}^* H_{k-1}.
\]
By induction, we may assume that $Z_{k-1} = (k-1)H_{k-1} - \pi_{k-1}^*A_{k-1,L}$. Recalling from Lemma \ref{canonical-bundle} that $K_{X^{[k-1,k]}} = \rho ^* K_{X^{[k]}} + F_{k-1}$,
we have
\begin{equation}\label{eq:computeZ_k2}
K_{B^{k-1,k}}=\pi_{k-1,k}^* (\rho^* K_{X^{[k]}}+F_{k-1}+\tau^* N_{k-1,L})-(k-1)\widetilde{\tau}^* H_{k-1}.
\end{equation}
Comparing $(\ref{eq:computeZ_k1})$ and $(\ref{eq:computeZ_k2})$ above, we 
obtain $c=-1$ and
\[
\rho^*(N_{k,L}-A_k)=\tau^* (N_{k-1,L}-A_{k-1,L})+F_{k-1}
\]
By Lemma \ref{identities} we get $A_{k} = A_{k,L}$ as desired.
\end{proof}

Note that in particular, if $X^{[k]}$ is smooth, then by adjunction we have
$
\omega_{Z_{k}}\cong \pi_k ^* (T_{k,\omega_X}\otimes \delta_k ^{n-1}).
$

\section{Vanishing theorems on Hilbert schemes}\label{section-cohomology-vanishing}

In this section, we prove the main vanishing theorem (Theorem \ref{intro-vanishing}), which will be the technical heart of our main results in the next two sections. We let $L$ be a line bundle on $X$, and we assume that $n=\dim(X) \leq 2$ or $k \leq 3$, so that $X^{[k]}$ and $X^{[k-1,k]}$ are smooth.

\subsection{Standard vanishing results for higher direct images}
For the vanishing of higher direct images, we frequently use the following standard facts.

\begin{lemma}[{\cite[Lemma 4.3.10]{pag1}}]\label{lem:standfactforhdi} Let $f\colon Y\lra Z$ be a morphism of irreducible projective varieties. Suppose that $\mathcal{F}$ is a coherent sheaf on $Y$ with the property that $H^i(X, \mathcal{F}\otimes f^* A)=0$
for some sufficiently positive line bundle $A$ and all $i\geq 1$. Then $R^i f_* \mathcal{F}=0$ for all $i\geq 1$. 
\end{lemma}

\begin{lemma}\label{standard-grauert-riemenschneider}
    Let $Z$ be a smooth projective variety, $f\colon Z \lra Y$ be a birational morphism, and $G$ be a divisor on $Z$ such that $\omega_Z^{-1}(G)$ is relatively nef over $Y$.
    Then $R^i f_* \mathcal{O}_Z(G)=0$ for all $i\geq 1$.
\end{lemma}

\begin{proof}
We may pick a sufficiently positive line bundle $H$ on $Y$ such that $f^* H \otimes \omega_Z^{-1}(G)$ is big and nef. Then we conclude by the Kawamata--Viehweg vanishing theorem and Lemma \ref{lem:standfactforhdi}.
\end{proof}

\subsection{Main relative vanishing}
The aim of this subsection is to prove the following crucial ingredient of the proof of the main vanishing theorem. Recall that $h:X^{[k]}\lra \operatorname{Sym}^k (X)$ is the Hilbert--Chow morphism, $M_{k,L}$ is the syzygy bundle and $\delta_k$ is the determinant of the dual of the secant sheaf $E_{k,\mathcal{O}_X}$.

\begin{theorem}\label{main-direct-image-vanishing}
Assume that $L$ separates $k$-schemes. Then
$R^i h_{*} (\wedge^j M_{k,L}\otimes \delta_k ^{-\ell}) =0$ for $i \geq 1$, $j \geq 0$ and $\ell\in\{1,2\}$.
\end{theorem}

This is trivial for $n=1$, since the Hilbert--Chow morphism $h$ is the identity map. In the remainder of this subsection we assume $n \geq 2$. To prove Theorem \ref{main-direct-image-vanishing}, we first need some preliminary results.

\begin{lemma}\label{direct-image-res-r-k-1-k}
Let $B$ be a vector bundle on $X$. Then 
$R^i \rho_{*} \operatorname{res}^* B =0$ and $R^i \rho_{*} (\operatorname{res}^* B(-F_{k-1}))=0$  for all $i\geq 1$.
\end{lemma}

\begin{proof}
Note that by \S\ref{section-hilbert-scheme} we have the following commutative diagram
\begin{center}
    \begin{tikzcd}
        & X^{[k-1,k]}\ar[dl,"\operatorname{res}"]\ar[dr,"\operatorname{res}\times\rho"']\ar[bend left=30,"\rho"]{ddrr}\\
        X & & \mathcal{Z}_k\ar[ll,"\operatorname{pr}_1\vert_{\mathcal{Z}_k}"] \ar[dr,"\operatorname{pr}_2\vert_{\mathcal{Z}_k}"']\\
        & & & X^{[k]}
    \end{tikzcd}.
\end{center}
This shows that for a vector bundle $E$ on $X$,
$$
R^i\rho_*\operatorname{res}^*E = R^i(\operatorname{pr}_2\vert_{\mathcal{Z}_k}\circ(\operatorname{res}\times\rho))_*(\operatorname{res}\times\rho)^*\operatorname{pr}_1\vert_{\mathcal{Z}_k}^*E.
$$
Applying the Leray spectral sequence to the composition, and noting the vanishing resulting from finiteness of $\operatorname{pr}_2$, it suffices to show that $R^i(\res \times \rho)_* \mathcal{O}_{X^{[k-1,k]}}=0$ and $R^i (\res \times \rho)_* \mathcal{O}_{X^{[k-1,k]}}(-F_{k-1})=0$ for $i \geq 1$. The first vanishing follows from the fact that $\mathcal{Z}_k$ has rational singularities (see the discussion in Subsection \ref{subsection-uni-family}).

For the second vanishing, we go by cases. 
Fix a point $\zeta=(x,[\xi])\in\mathcal{Z}_k$, and let $G_\zeta := (\operatorname{res}\times \rho)^{-1} (\zeta)$ be the fiber of $\res \times \rho$ over $\zeta$. Assume first that $n=2$. By \cite[Proposition 2.3]{Song.16} (see also \cite{Ellingsrud.Stromme}), we have $G_\zeta\cong \mathbb{P}^{b_2-1}$, where $b_2$ is the dimension of the socle of $\mathcal{O}_{\mathcal{Z}_k,\zeta}$. By \cite[Page 353]{Song.16}, we have a short exact sequence
\[
0\lra \mathcal{O}_{\mathbb{P}^{b_2-1}} ^{\oplus b_2+1}(-1) \lra \mathcal{O}_{\mathbb{P}^{b_2-1}} ^{\oplus 2k+2}\lra N_{G_\zeta / X^{[k-1,k]}} ^{\vee}\lra 0.
\]
Therefore we have
\[
H^i(\mathbb{P}^{b_2-1}, S^j N_{G_\zeta / X^{[k-1,k]}} ^{\vee} (-1))=0~~\text{ for $i\geq 1$ and $j\geq 0$}.
\]
Note that $\omega_{G_\zeta}\cong\omega_{X^{[k-1,k]}}|_{G_\zeta} \otimes \operatorname{det} N_{G_\zeta/X^{[k-1,k]}}$, so that $\omega_{X^{[k-1,k]}}|_{G_\zeta}\cong \mathcal{O}_{\mathbb{P}^{b_2-1}}(1)$. As a consequence, we have the chain of isomorphisms $\mathcal{O}_{X^{[k-1,k]}}(F_{k-1})|_{G_\zeta} \cong \omega_{X^{[k-1,k]}/X^{[k]}}|_{G_{\zeta}} \cong \mathcal{O}_{\mathbb{P}^{b_2-1}}(1)$.
Let $G_{\zeta,\ell}$ be the $\ell$-th thickening of $G_\zeta$ in $X^{[k-1,k]}$. Consider the short exact sequence
\[
0\lra S^\ell N_{G_\zeta/X^{[k-1,k]}} ^\vee \lra
\mathcal{O}_{G_{\zeta,\ell+1}}\lra \mathcal{O}_{G_{\zeta,\ell}}\lra 0.
\]
Twisting by $\mathcal{O}_{X^{[k-1,k]}}(-F_{k-1})$ and using the above cohomology vanishing for $S^\ell N_{G_\zeta/X^{[k-1,k]}} ^\vee (-1)$, we get $R^i (\operatorname{res}\times \rho)_* \mathcal{O}_{X^{[k-1,k]}}(-F_{k-1})=0$ for $i\geq 1$ by the formal function theorem (\cite[Theorem 11.1, Chapter III]{Hartshorne.77}).

Next, assume that $n \geq 3$ so that $k = 2$ or $3$. If $k=2$, then $\operatorname{res}\times \rho$ is the identity map, and hence, the statement is trivial. Assume therefore that $k=3$. Using the same notation as above, it is immediate to see that 
$G_\zeta\cong \mathbb{P}^1$ or a point. Suppose that $G_{\zeta} \cong \mathbb{P}^1$. Since in any case the embedding dimension of $\xi \subseteq X$ is at most 2, we see by arguing locally as above that $\mathcal{O}_{X^{[k-1,k]}}(-F_{k-1})|_{G_\zeta}\cong \mathcal{O}_{\mathbb{P}^1}(-1)$ and $N_{G_\zeta/X^{[k-1,k]}}^\vee$ is globally generated. Then the same argument using the formal function theorem gives the desired vanishing $R^i (\operatorname{res}\times \rho)_* \mathcal{O}_{X^{[k-1,k]}}(-F_{k-1})=0$ for $i\geq 1$.
\end{proof}

\begin{corollary}\label{direct-image-E-k-L}
Pick two integers $i\geq 1$ and $k\geq 2$. Then:
   \begin{enumerate}[topsep=0pt]
       \item $R^i \rho_{*} \tau^* S^j E_{k-1,L}=0$ and $R^i \rho_* \tau^* \wedge^j E_{k-1,L}=0$.
       \item $R^i \rho_{*} \tau^* S^j M_{k-1,L}=0$ and $R^i \rho_* \tau^* \wedge^j M_{k-1,L}=0$, provided that $L$ separates $k$-schemes.
   \end{enumerate}
\end{corollary}

\begin{proof}
Consider the symmetric or wedge powers of the following short exact sequences     
    \begin{align*}
    &0\lra \operatorname{res}^* L(-F_{k-1})\lra
    \rho^* E_{k,L}\lra
    \tau^* E_{k-1,L}\lra 0,\\
    &0\lra \rho^* M_{k,L}\lra \tau^* M_{k-1,L}\lra \operatorname{res}^*L (-F_{k-1})\lra 0.
    \end{align*}
Then the assertions are immediate from  Lemma \ref{direct-image-res-r-k-1-k}.
\end{proof}

The following result is important for inductive purposes.

\begin{proposition}\label{direct-summand-E-k-L}
The vector bundle $E_{k,L}^{\otimes j}$ is a direct summand of $\rho_{*} (\tau^* E_{k-1,L}^{\otimes j})$ for all $1\leq j\leq k-1$.
In particular, in this range $S^j E_{k,L}$ is a direct summand of $\rho_* \tau^* S^j E_{k-1,L}$.
\end{proposition}

\begin{proof}
Recall from \S\ref{nestedHilb} that the universal family 
$$
\mathcal{W}_{k-1,k} = \{(x, \xi \subseteq \eta) \mid x \in \eta\} \subseteq X \times X^{[k-1,k]}
$$
has two components, namely
$$
\mathcal{V}_{k-1,k} = \{(x, \xi \subseteq \eta) \mid x \in \xi\}~~\text{ and }~~
\Gamma_{\operatorname{res}} = \{(x, \xi \subseteq \eta) \mid x \in \xi \setminus \eta \}. 
$$
Note that $\Gamma_{\operatorname{res}}$ is isomorphic to $X^{[k-1,k]}$ via $\operatorname{pr}_2\colon\mathcal{W}_{k-1,k}\lra X^{[k-1,k]}$. From the equality
\[
\mathcal{V}_{k-1,k}\cap \Gamma_{\operatorname{res}} = (\pr_2^{-1} F_{k-1})|_{\Gamma_{\operatorname{res}}}, 
\] 

\noindent we obtain a scheme-theoretic inclusion
$$
\mathcal{V}_{k-1,k}\cap \Gamma_{\operatorname{res}} \subseteq (\pr_2^{-1} F_{k-1} )|_{\mathcal{V}_{k-1,k}}.
$$
Now, consider
$$
\begin{array}{l}
\mathcal{W}:=\mathcal{W}_{k-1,k}\times_{X^{[k-1,k]}} \mathcal{W}_{k-1,k}\times_{X^{[k-1,k]}} \cdots \times_{X^{[k-1,k]}} \mathcal{W}_{k-1,k},\\
\mathcal{V}:=\mathcal{V}_{k-1,k}\times_{X^{[k-1,k]}} \mathcal{V}_{k-1,k}\times_{X^{[k-1,k]}} \cdots \times_{X^{[k-1,k]}} \mathcal{V}_{k-1,k},\text{ and}\\
\mathcal{Z}:=\mathcal{Z}_k\times_{X^{[k]}} \mathcal{Z}_k\times_{X^{[k]}} \cdots \times_{X^{[k]}} \mathcal{Z}_k,
\end{array}
$$
where the products are taken $j$ times. These schemes fit a commutative diagram
\begin{center}
    \begin{tikzcd}
     \mathcal{V} \ar[r, hookrightarrow, "\iota"]& \mathcal{W} \ar[r, "\operatorname{pr}_{\mathcal{W}}"]\ar[d, "g"]& X^{[k-1,k]}\ar[d, "\rho"]\\
       & \mathcal{Z} \ar[r,"\operatorname{pr}_{\mathcal{Z}}"]& X^{[k]}
    \end{tikzcd},
    \end{center}
where $\iota$ is a closed embedding and the right square is Cartesian. Notice that $\pr_{\mathcal{W}}$, $\pr_{\mathcal{W}} \circ \iota$, and $\pr_{\mathcal{Z}}$ are finite flat and therefore $\mathcal{W}$, $\mathcal{V}$, and $\mathcal{Z}$ are Cohen--Macaulay by the miracle flatness theorem. The schemes $\mathcal{W}$, $\mathcal{V}$, and $\mathcal{Z}$ are generically reduced, and hence, they are reduced. 
We have the identifications
$$
\begin{array}{l}
\mathcal{W}\cong\{(x_1, \ldots, x_j, \xi \subseteq \eta) \mid x_1, \ldots, x_j \in \eta\} \subseteq X^j \times X^{[k-1,k]},\\
\mathcal{V}\cong\{(x_1, \ldots, x_j, \xi \subseteq \eta) \mid x_1, \ldots, x_j \in \xi\} \subseteq X^j \times X^{[k-1,k]},\text{ and}\\
\mathcal{Z}\cong\{(x_1, \ldots, x_j , \eta) \mid x_1, \ldots, x_j \in \eta) \subseteq X^j \times X^{[k]}.
\end{array}
$$
In light of the above, we have
$$
\mathcal{V} \cap \overline{(\mathcal{W} \setminus \mathcal{V})} \subseteq (\pr_{\mathcal{W}}^*F_{k-1})|_{\mathcal{V}}.
$$
From the sequence of inclusions
$$
\iota_* \mathcal{O}_{\mathcal{V}} (-(\pr_{\mathcal{W}}^{-1}F_{k-1})|_{\mathcal{V}}) \lra \mathcal{I}_{\mathcal{V} \cap \overline{(\mathcal{W} \setminus \mathcal{V})}/\mathcal{W} } \lra \mathcal{O}_{\mathcal{W}},
$$
we get an injective map
$$
\iota_* \mathcal{O}_{\mathcal{V}} \lra \mathcal{O}_{\mathcal{W}}(\pr_{\mathcal{W}}^{-1} F_{k-1}).
$$
When $j \leq k-1$, the map $g \circ \iota \colon \mathcal{V} \lra \mathcal{Z}$ is surjective. Then we obtain a sequence of injective maps
\begin{equation}\label{eq:seqinjfromZtoW}
\mathcal{O}_{\mathcal{Z}} \lra g_* \iota_* \mathcal{O}_{\mathcal{V}} \lra g_* \mathcal{O}_{\mathcal{W}}(\pr_{\mathcal{W}}^* F_{k-1}).
\end{equation}
From the above Cartesian diagram, we find that
$$
g_* \mathcal{O}_{\mathcal{W}}(\pr_{\mathcal{W}}^* F_{k-1,k}) \cong \pr_{\mathcal{Z}}^* \rho_* \mathcal{O}_{X^{[k-1,k]}}(F_{k-1}).
$$
As $\rho_* \mathcal{O}_{X^{[k-1,k-]}}(F_{k-1}) \cong \rho_* \omega_{X^{[k-1,k]}/X^{[k]}}$, there is a trace map $\rho_* \mathcal{O}_{X^{[k-1,k]}}(F_{k-1}) \lra \mathcal{O}_{X^{[k]}}$ (see Lemma \ref{lem:O_X^[k]directsummand}). Pulling back this trace map to $\mathcal{Z}$ and composing with (\ref{eq:seqinjfromZtoW}), we get a sequence of maps
$$
\mathcal{O}_{\mathcal{Z}} \lra g_* \iota_* \mathcal{O}_{\mathcal{V}} \lra g_* \mathcal{O}_{\mathcal{W}}(\pr_{\mathcal{W}}^* F_{k-1,k}) \lra \mathcal{O}_{\mathcal{Z}}.
$$
The whole composition is the identity, thus the injective map $\mathcal{O}_{\mathcal{Z}} \lra g_* \iota_* \mathcal{O}_{\mathcal{V}}$ splits. In particular, $\mathcal{O}_{\mathcal{Z}}$ is a direct summand of $g_* j_* \mathcal{O}_{\mathcal{V}}$. The  proposition now follows by twisting this splitting injective map with $\pr_1^* (L^{\boxtimes j})$, where $\pr_1 \colon \mathcal{Z}^j \lra X^j$ is the projection map, and pushing forward to $X^{[k]}$ via $\pr_{\mathcal{Z}}$. In fact, there are natural isomorphisms
$$
\pr_{\mathcal{Z},*} \pr_1^* (L^{\boxtimes j}) \cong E_{k,L}^{\otimes j}~~\text{ and }~~\pr_{\mathcal{W},*} \iota_*  g^* \pr_1^* (L^{\boxtimes j}) \cong \tau^* E_{k-1,L}^{\otimes j}
$$
by the K\"{u}nneth formula. Since the splitting map $E_{k,L}^{\otimes j} \lra \rho_* (\tau^* E_{k-1,L}^{\otimes j})$ is equivariant under the permuting action of the symmetric group $S_j$, the last assertion also holds. 
\end{proof}

\begin{remark}
    The case $j=1$ of Proposition \ref{direct-summand-E-k-L} can be also obtained by considering the trace map relative to the morphism $g\circ \iota\colon \mathcal{V}_{k-1,k}\lra \mathcal{Z}_k$, which expresses $\mathcal{O}_{\mathcal{Z}_k}$ as a direct summand of $g_* \iota_* \mathcal{O}_{\mathcal{V}_{k-1,k}}$ right away. In order to construct this trace map, however, it is important that $g\circ \iota$ is flat above $\operatorname{pr}_2 ^{-1} (X_* ^{[k]})$, so that one may apply \cite[Definition 5.6]{Kollar-Mori}. This approach fails for $j\geq 2$, since $g\circ \iota\colon \mathcal{V}\lra \mathcal{Z}$ is no longer flat on $\operatorname{pr}_2 ^{-1}(X_* ^{[k]})$.
\end{remark}

\begin{remark}\label{remark-not-compatible}
Consider the commutative diagram
\begin{equation}\label{eq:splitting-noncompatible-diagram}
\hbox{
    \begin{tikzcd}
        H^0(X,L)\otimes \mathcal{O}_{X^{[k]}}\ar[r]\ar[d]& E_{k,L}\ar[d]\\
        H^0(X,L)\otimes \rho_* \mathcal{O}_{X^{[k-1,k]}}\ar[r]& \rho_* \tau^* E_{k-1,L}
    \end{tikzcd}.
    }
\end{equation}
Proposition \ref{direct-summand-E-k-L} shows that $E_{k,L}\lra \rho_* \tau^* E_{k-1,L}$ splits. Of course, the map $\mathcal{O}_{X^{[k]}}\lra \rho_* \mathcal{O}_{X^{[k-1,k]}}$ also splits via the trace. These two splitting, however, are not compatible. For example, if $k=2$ and $L$ is a globally generated line bundle that is not very ample, then the bottom horizontal map is surjective but the top horizontal map is not. This happens because the trace splittings of the diagram
\begin{center}
    \begin{tikzcd}
        \mathcal{O}_{\mathcal{Z}_k}\ar[r,equals]\ar[d]& \mathcal{O}_{\mathcal{Z}_k}\ar[d]\\
        g_* \mathcal{O}_{\mathcal{W}_{k-1,k}}\ar[r] &g_* \mathcal{O}_{\mathcal{V}_{k-1,k}}
    \end{tikzcd}
\end{center}
are not compatible either. That is, the restriction of the trace map $g_* \mathcal{O}_{\mathcal{W}_{k-1,k}}\lra \mathcal{O}_{\mathcal{Z}_k}$ to $g_* \mathcal{I}_{\mathcal{V}_{k-1,k}/\mathcal{W}_{k-1,k}}$ is not zero.
As a consequence, one may further argue that the vector bundle $\wedge^j M_{k,L}$ may not be a direct summand of $\rho_* \tau^* \wedge^j M_{k-1,L}$ for $j \geq 1$. For instance, take $X$ to be a smooth projective curve of genus $g \geq 1$ and $L$ a very ample line bundle on $X$, and consider again the case $k=2$ and $j=1$. The right vertical map in (\ref{eq:splitting-noncompatible-diagram}) is an isomorphism. Taking kernels of the two horizontal maps of (\ref{eq:splitting-noncompatible-diagram}), we get a short exact sequence
$$
0 \lra M_{2,L} \lra \rho_* \tau^* M_{1,L} \lra H^0(X, L) \otimes \delta_2^{-1} \lra 0.
$$
Pick any point $x$ in $X$. Regarding $x$ as a divisor of degree 1, we have $H^0(X, M_{1,L}(x))=0$, so $H^0(X^{[1,2]}, \tau^* M_{1,L}(F_1))=0$. Taking pushforward via $\rho$, we obtain $H^0(X^{[2]}, \rho_* \tau^* M_{1,L} \otimes \delta_2)=0$. If the left vertical short exact sequence were to split, then $H^0(X^{[2]}, \rho_* \tau^* M_{1,L} \otimes \delta_2) \lra X^0(X, L)$ would be surjective, but this is a contradiction.
\end{remark}

\begin{proposition}\label{main-direct-image-vanishing-S^jE}
Assume that $\dim (X)= 2$. Then   $R^i h_{*} (S^j E_{k,L} ^\vee \otimes \delta_k ^{-\ell})=0$ for $i \geq 1, \;0 \leq j \leq i$ and $\ell \in \{ 0,1\}$. 
\end{proposition}

\begin{proof}
We proceed by induction on $k$. The case $k=1$ is immediate since $h$ is the identity map. Assume $k \geq 2$. Recall that the dimension of the fibers of $h_k$ is bounded above by $k-1$  (see \cite{Briancon}). By Lemma \ref{lem:standfactforhdi}, it suffices to check
\[
H^i(X^{[k]}, S^j E_{k,L} ^\vee \otimes T_{k,H} \otimes \delta_k^{-\ell})=0~~\text{ for $1 \leq i \leq k-1$ and $0 \leq j \leq i$},
\]
where $H$ is a sufficiently positive line bundle on $X$. Notice that $0 \leq j \leq k-1$. 
By Serre duality, we have
\[
H^i(X^{[k]}, S^j E_{k,L} ^\vee \otimes T_{k,H}\otimes \delta_{k}^{-\ell})\cong
H^{2k-i}(X^{[k]}, S^j E_{k,L}\otimes T_{k,\omega_X\otimes H^{-1}} \otimes \delta_k^{\ell})^\vee.
\]
Since $S^j E_{k,L}$ is a direct summand of $\rho_* \tau^* S^j E_{k-1,L}$ by Proposition \ref{direct-summand-E-k-L}, it is enough to show
\[
H^{2k-i}(X^{[k]}, \rho_{*} \tau^* S^j E_{k-1,L}\otimes T_{k,\omega_X\otimes H^{-1}}\otimes \delta_k^{\ell})=0~~\text{ for $1\leq i\leq k-1$ and $0\leq j\leq i$}.
\]
By Lemma \ref{identities}, we have
\[
\rho^* (T_{k,\omega_X\otimes H^{-1}}\otimes \delta_k^{\ell}) \cong \tau^* (T_{k-1,\omega_X\otimes H^{-1}}\otimes \delta_{k-1}^{\ell}) \otimes \operatorname{res}^* (\omega_X\otimes H^{-1}) (\ell F_{k-1}).
\]
Using Corollary \ref{direct-image-E-k-L}, the above vanishes if
\[
H^{2k-i}(X^{[k-1,k]}, \tau^* (S^j E_{k-1,L}\otimes T_{k-1,\omega_X\otimes H^{-1}}\otimes \delta_{k-1}^{\ell}) \otimes \operatorname{res}^* (\omega_X\otimes H^{-1}) ( \ell F_{k-1}))=0.
\]
Since $R^q \operatorname{bl}_{*} \mathcal{O}_{X^{[k-1,k]}}(\ell F_{k-1})=0$ for all $q\geq 1$ by Lemma \ref{standard-grauert-riemenschneider}, it is enough to show 
\[
H^{2k-i}(X\times X^{[k-1]}, (\omega_X \otimes H^{-1}) \boxtimes
(S^j E_{k-1,L}\otimes T_{k-1,\omega_X\otimes H^{-1}}\otimes \delta_{k-1}^{\ell}))=0.
\]
By K\"unneth formula, this reduces to 
\[
H^2(X,\omega_X \otimes H^{-1})\otimes H^{2k-i-2}(X^{[k-1]}, S^j E_{k-1,L}\otimes T_{k-1,\omega_X\otimes H^{-1}}\otimes \delta_{k-1}^{\ell})=0.
\]
Serre duality gives 
\[
H^{2k-i-2}(X^{[k-1]}, S^j E_{k-1,L}\otimes T_{k-1,\omega_X\otimes H^{-1}}\otimes \delta_{k-1}^{\ell})\cong
H^i(X^{[k-1]}, S^j E_{k-1,L} ^\vee \otimes T_{k-1,H}\otimes \delta_k ^{-\ell})^{\vee},
\]
but the latter cohomology vanishes for $i \geq 1$ and $0 \leq j \leq i$ by inductive hypothesis.
\end{proof}

\begin{remark}
Note that we do not impose any positivity condition on $L$ in Proposition \ref{direct-summand-E-k-L} and Proposition \ref{main-direct-image-vanishing-S^jE}. A similar phenomenon was also observed in \cite[Proposition 6.1]{Agostini}. Furthermore, Proposition \ref{main-direct-image-vanishing-S^jE} for $0 \leq j \leq 2$ and $\ell=0$ was obtained in \cite[Lemma 6.3]{Agostini}.
\end{remark}

\begin{proof}[Proof of Theorem \ref{main-direct-image-vanishing}]
When $n=1$ or $k=1$, the Hilbert--Chow morphism $h$ is the identity map. Thus we only need to consider the cases $n \geq 2$ or $k \geq 2$.
First, assume that $n=2$. The short exact sequence
$$
0 \lra E_{k,L}^{\vee} \lra H^0(X, L)^{\vee} \otimes \mathcal{O}_{X^{[k]}} \lra M_{k,L}^{\vee} \lra 0
$$
gives a long exact sequence
$$
\cdots \lra S^2 E_{k,L}^\vee \otimes V_2 \lra E_{k,L}^\vee \otimes V_1\lra V_0 \lra \wedge^{r_k-j} M_{k,L}^{\vee}\lra 0,
$$
where the $V_m$ are trivial vector bundles for all $m \geq 0$ and $r_k=\operatorname{rank}(M_{k,L})$. Note that $\wedge^{r_k-j} M_{k,L}^{\vee} \cong \wedge^j M_{k,L}\otimes N_{k,L}$. We may then conclude  by a diagram chase and Proposition \ref{main-direct-image-vanishing-S^jE}.

Next, assume that $n \geq 3$. Then $k=2$ or $k=3$. Note that the natural addition morphism $a_k \colon X\times \operatorname{Sym}^{k-1} (X)\lra \operatorname{Sym}^k X$ is finite. Let $h_{k-1,k} \colon X^{[k-1,k]} \lra X \times \operatorname{Sym}^{k-1}(X)$ be the composition
$$
X^{[k-1,k]} \xrightarrow{~\bl~} X \times X^{[k-1]} \xrightarrow{\operatorname{id}_X \times h} X \times \operatorname{Sym}^{k-1}(X).
$$
We have a commutative diagram 
\begin{center}
\begin{tikzcd}
X^{[k-1,k]}\ar[r,"h_{k-1,k}"] \ar[d,"\rho"] & X\times \operatorname{Sym}^{k-1}(X)\ar[d,"a_k"]\\
X^{[k]} \ar[r,"h"] & \operatorname{Sym}^k(X)
\end{tikzcd}.
\end{center}
Recall that $\rho^* \delta_k = \bl^*(\mathcal{O}_X \boxtimes \delta_{k-1})(F_{k-1})$. 
By Lemma \ref{lem:O_X^[k]directsummand}, the vector bundle $\wedge^j M_{k,L} \otimes \delta_k^{-\ell}$ is a direct summand of $\rho_*(\rho^* \wedge^j M_{k,L} \otimes \bl^*(\mathcal{O}_X \boxtimes \delta_{k-1}^{-\ell})(-F_{k-1}))$. 
Projection formula, the Grauert--Riemenschneider vanishing theorem and Lemma \ref{direct-image-res-r-k-1-k} show
$$
R^i \rho_*(\rho^* \wedge^j M_{k,L} \otimes \bl^*(\mathcal{O}_X \boxtimes \delta_{k-1}^{-\ell})(-F_{k-1})) = 0~~\text{ for $i \geq 1$}. 
$$
In view of this vanishing, the above commutative diagram and the fact that $a_k$ is finite, it suffices to prove
\[
R^i h_{k-1,k,*} (\rho^* \wedge^j M_{k,L} \otimes \operatorname{bl}^*(\mathcal{O}_X\boxtimes \delta_{k-1}^{-\ell})(-F_{k-1}))=0~~\text{ for $i\geq 1$}.
\]
We proceed by induction on $k$. As the case of $k=1$ is trivial, we assume $k \geq 2$. 
Form the short exact sequence
$$
0 \lra \rho^* M_{k,L} \lra \bl^*(\mathcal{O}_X \boxtimes M_{k-1,L}) \lra \bl^*(L \boxtimes \mathcal{O}_{X^{[k-1]}})(-F_{k-1}) \lra 0,
$$
we get a long exact sequence
$$
\cdots \lra \bl^*(L^{-2} \boxtimes \wedge^{j+2} M_{k-1,L})(F_{k-1}) \lra \bl^*(L^{-1} \boxtimes \wedge^{j+1} M_{k-1,L}) \lra \rho^* \wedge^j M_{k,L}(-F_{k-1}) \lra 0.
$$
Note that if $0 \leq m \leq n-1$, then $\bl_* \mathcal{O}_{X^{[k-1,k]}}(mF_{k-1}) = \mathcal{O}_{X \times X^{[k-1]}}$ and $R^i \bl_* \mathcal{O}_{X^{[k-1,k]}}(mF_{k-1}) = 0$ for $i\geq 1$. As every fiber of $\bl$ has dimension at most $n-1$, the diagram chase shows 
$$
R^i \bl_* \rho^* \wedge^j M_{k,L}(-F_{k-1})=0~~\text{ for $i\geq 1$.}
$$
Thus we have reduced the problem to
$$
R^i (\operatorname{id}_X \times h)_* (\bl_* \rho^* \wedge^j M_{k,L}(-F_{k-1}) \otimes (\mathcal{O}_X \boxtimes \delta_{k-1}^{-\ell})) = 0~~\text{ for $i \geq 1$.}
$$
We have a long exact sequence
$$
\cdots \lra L^{-2} \boxtimes \wedge^{j+2} M_{k-1,L} \lra L^{-1} \boxtimes \wedge^{j+1} M_{k-1,L} \lra \bl_* \rho^* \wedge^j M_{k,L}(-F_{k-1}) \lra 0.
$$
By inductive hypothesis and K\"unneth formula, we have
\[
R^i (\operatorname{id}\times h)_* (L^{-s}\boxtimes (\wedge^j M_{k-1,L} \otimes \delta_{k-1}^{-\ell})) = 0~~\text{ for $1 \leq i \leq n-1$ and $s \geq 1$}.
\]
As every fiber of $\operatorname{id}_X \times h$ has dimension at most $n-1$, a diagram chase shows the desired vanishing of higher direct images.
\end{proof}

\subsection{Main global vanishing}
Here we prove a global counterpart to 
Theorem \ref{main-direct-image-vanishing}. The first step in this direction is to relate the cohomology of the universal families to the cohomology of the syzygy bundle $M_{k,L}$. This relation is well-known for $k=1$ 
(see \cite[Theorem 3.6, Remark 3.8]{Bangere.Lacini.24} and \cite[Proposition 1.1]{Ein.Lazarsfeld.Yang}).
We briefly set up the following notation. Let $\pr_{i,j+1} \colon X^j \times X^{[k]} \lra X \times X^{[k]}$ be the map given by $(x_1, \ldots, x_j, \xi) \longmapsto (x_i, \xi)$ for $1 \leq i \leq j$, and let $\mathcal{I}_{\mathcal{Z}_k}$ be the ideal sheaf of $\mathcal{Z}_k$ in $X \times X^{[k]}$. 

\begin{lemma}\label{diagonal-syzygy-bundle}
Suppose that $L$ separates $k$-schemes and $H^i(X, L)=0$ for $i \geq 0$. Then we have
$$
R^i \operatorname{pr}_{j+1,*}(\operatorname{pr}_{1,j+1}^* \mathcal{I}_{\mathcal{Z}_k} \otimes \cdots \otimes \operatorname{pr}_{j,j+1}^* \mathcal{I}_{\mathcal{Z}_k} \otimes
(L^{\boxtimes j}\boxtimes \mathcal{O}_{X^{[k]}}))= \begin{cases} M_{k,L}^{\otimes j} & \text{if $i=0$} \\ 0 & \text{if $i\geq 1$}. \end{cases}
$$
\end{lemma}

\begin{proof}
The statement for $j=1$ is immediate by definition.
Note furthermore that $M_{k,L}$ is locally free since $L$ separates $k$-schemes.
For $j\geq 2$, we may therefore conclude by K\"unneth formula 
relative to the morphism
\[
(X\times X^{[k]})\times_{X^{[k]}} (X\times X^{[k]})
\times_{X_{[k]}}\cdots \times_{X^{[k]}} (X\times X^{[k]})\lra X^{[k]}
\]

\noindent that the following equality holds:
\begin{align*}
\operatorname{pr}_{j+1,*}(\operatorname{pr}_{1,j+1}^* \mathcal{I}_{\mathcal{Z}_k} \otimes \cdots &\otimes \operatorname{pr}_{j,j+1}^* \mathcal{I}_{\mathcal{Z}_k} \otimes
(L^{\boxtimes j}\boxtimes \mathcal{O}_{X^{[k]}}))=\\
&\operatorname{pr}_{j+1,*}(\operatorname{pr}_{1,j+1}^* \mathcal{I}_{\mathcal{Z}_k} \otimes \operatorname{pr}_1 ^* L) \otimes \cdots \otimes 
\operatorname{pr}_{j+1,*}(
\operatorname{pr}_{j,j+1}^* \mathcal{I}_{\mathcal{Z}_k} \otimes \operatorname{pr}_j ^* L). \qedhere
\end{align*}
\end{proof}

\noindent We will see below in Lemma \ref{subscheme-transversality}
that there is an isomorphism
$$
\operatorname{pr}_{1,j+1}^* \mathcal{I}_{\mathcal{Z}_k} \otimes \cdots \otimes \operatorname{pr}_{j,j+1}^* \mathcal{I}_{\mathcal{Z}_k} 
\cong \mathcal{I}_{\operatorname{pr}_{1,j+1}^{-1}(\mathcal{Z}_k)}  \cdots \mathcal{I}_{\operatorname{pr}_{j,j+1}^{-1}(\mathcal{Z}_k)} 
\cong \mathcal{I}_{\operatorname{pr}_{1,j+1}^{-1}(\mathcal{Z}_k) \cup \cdots \cup \operatorname{pr}_{j,j+1}^{-1}(\mathcal{Z}_k)},
$$

\noindent but we do not need this fact at this point.

The following is the main vanishing theorem (Theorem \ref{intro-vanishing}), which is the culmination of all the above discussions and a key ingredient of our main results.

\begin{theorem}\label{main-vanishing-serre}
Fix an integer $j_0$, and assume that $L$ is sufficiently positive. Then 
$$
H^i(X^{[k]}, \wedge^j M_{k,L}\otimes A_{k,L})=0~~\text{ for $i \geq 1$ and $0 \leq j \leq j_0$.}
$$
\end{theorem}

\begin{proof}
By Theorem \ref{main-direct-image-vanishing}, we have
    \[
    H^i(X^{[k]}, \wedge^j M_{k,L}\otimes A_{k,L})\cong H^i(\operatorname{Sym}^k(X), h_{*}(\wedge^j M_{k,L}\otimes \delta_{k}^{-2})\otimes S_{k,L}).
    \]
Since $\wedge^j M_{k,L}$ is a direct summand of $M_{k,L}^{\otimes j}$, it is enough to show that
\begin{equation}\label{eq:mvt_proof}
H^i(\operatorname{Sym}^k (X), h_{*}(M_{k,L}^{\otimes j} \otimes \delta_k ^{-2})\otimes S_{k,L})=0~~\text{ for $i\geq 1$ and $0 \leq j \leq j_0$}.
\end{equation}
Consider the commutative diagram
\begin{center}
\begin{tikzcd}
    X^j \times X^{[k]}\ar[r, "\operatorname{id}_{X^j}\times h"]\ar[d,  "\operatorname{pr}_{j+1}"]& X^j \times \operatorname{Sym}^k (X) \ar[d,"\operatorname{pr}_2"]\\
    X^{[k]}\ar[r, "h"]& \operatorname{Sym}^k (X)
\end{tikzcd}.
\end{center}
Set $\mathcal{J}:=\operatorname{pr}_{1,j+1}^* \mathcal{I}_{\mathcal{Z}_k} \otimes \cdots \otimes \operatorname{pr}_{j,j+1}^* \mathcal{I}_{\mathcal{Z}_k} \otimes
(\mathcal{O}_{X^j} \boxtimes \delta_{k}^{-2})$. By Lemma \ref{diagonal-syzygy-bundle}, we have
$$
M_{k,L}^{\otimes j}\otimes \delta_k ^{-2} \cong \pr_{j+1,*} \mathcal{J} \otimes (L^{\boxtimes j} \boxtimes \mathcal{O}_{X^{[k]}}).
$$
Therefore, we obtain
$$
h_{*}(M_{k,L}^{\otimes j}\otimes \delta_k ^{-2}) \otimes S_{k,L}  \cong \pr_{2,*} \big( (\operatorname{id}_{X^{j}} \times h)_*  \mathcal{J} \otimes ( L^{\boxtimes j} \boxtimes S_{k,L} ) \big).
$$
If is $L$ sufficiently positive relative with respect to $j_0$, then by Fujita's vanishing theorem (\cite[Theorem 1.4.35]{pag1}) we have
$$
H^i(X^j \times \operatorname{Sym}^k(X), (\operatorname{id}_{X^{j}} \times h)_*  \mathcal{J} \otimes ( L^{\boxtimes j} \boxtimes S_{k,L} ))=0~~\text{ for $i \geq 1$ and $0 \leq j \leq j_0$.}
$$
Similarly, by Lemma \ref{lem:standfactforhdi}, we also have
$$
R^i \pr_{2,*} \big( (\operatorname{id}_{X^{j}} \times h)_*  \mathcal{J} \otimes ( L^{\boxtimes j} \boxtimes S_{k,L} ) \big) = 0~~\text{ for $i \geq 1$}.
$$
Therefore we obtain (\ref{eq:mvt_proof}), which concludes the proof.
\end{proof}

\begin{remark}\label{rem:meaningofsuffpos}
One may rephrase Theorem \ref{main-vanishing-serre} as follows. Let $L:=A^m \otimes B \otimes P$ be a line bundle on $X$, where $A$ is ample, $B$ is a nef, and $P$ is an arbitrary. For an integer $j_0 \geq 0$, there is an integer $m_0 \geq 1$ depending on $j_0$ but independent of $B$ such that the conclusion of  Theorem \ref{main-vanishing-serre} holds whenever $m \geq m_0$.
\end{remark}

\begin{corollary}\label{corollary-vanishing-serre}
Assume that $L$ is sufficiently positive. Then the multiplication map
    \[
    H^0(X^{[k]}, A_{k,L})\otimes S^j H^0(X^{[k]}, E_{k,L})\lra H^0(X^{[k]}, A_{k,L}\otimes S^j E_{k,L})
    \]
is surjective for all $j \geq 0$. Furthermore we have $H^i(X^{[k]}, A_{k,L}\otimes S^j E_{k,L})=0$ for all $i \geq 1$ and $j \geq 0$. 
\end{corollary}

\begin{proof}
Consider the short exact sequence
$$
0 \longrightarrow M_{k, L} \longrightarrow H^0(X, L) \otimes \OO_{X^{[k]}} \longrightarrow E_{k, L} \longrightarrow 0.
$$
It induces a long exact sequence
\[
\cdots \longrightarrow S^{j-2} H^0(L)  \otimes \wedge^2 M_{k, L} \longrightarrow S^{j-1} H^0(L)  \otimes M_{k, L} \longrightarrow S^j H^0(L) \otimes \OO_{X^{[k]}} \longrightarrow S^j E_{k, L} \longrightarrow 0.
\]
By Theorem \ref{main-vanishing-serre}, we have the vanishing
$$
H^i(X^{[k]}, \wedge^j M_{k,L} \otimes A_{k,L})=0~~\text{ for $i \geq 1$ and $0 \leq j \leq i$}.
$$
The statement then follows from diagram chasing.
\end{proof}

\begin{remark}\label{remark:hard-problem}
We note here a subtle but important issue, which is the reason for our somewhat involved approach to Theorem \ref{main-vanishing-serre}. If $k=1$, then the blow up of $\mathcal{J}$ is smooth. Furthermore, if $L$ is sufficiently positive, then $\mathcal{J}\otimes (L^{\boxtimes j}\otimes T_{k,L})$ is globally generated. Combining these two fact with the Kawamata--Viehweg vanishing theorem, one easily obtains $H^1(X, M_L^{\otimes j} \otimes L)=0$. If $k\geq 2$, however,
both facts may fail (see for example Remark \ref{no-global-generation} regarding global generation). 
In fact, it turns out that $H^1(X^{[k]}, M_{k,L}^{\otimes j} \otimes A_{k,L})$ is not zero in general, even when $L$ is sufficiently positive. We show below that, if $X$ is a surface, then $H^1(X^{[2]}, M_{2,L}^{\otimes j} \otimes A_{k,L}) \neq 0$ for $j \geq 4$. Let $G$ be a general fiber of the Hilbert--Chow morphism $h \colon X^{[2]} \lra \operatorname{Sym}^2(X)$ over $h(E_2)$. Then $G \cong \mathbb{P}^1$. Note that $M_{2,L}|_{G} \cong \mathcal{O}_{\mathbb{P}^1}(-1) \oplus \mathcal{O}_{\mathbb{P}^1}^{\operatorname{rank}M_{2,L}-1}$ and $A_{2,L}|_{G} \cong \mathcal{O}_{\mathbb{P}^1}(2)$. Thus $R^1 h_* (M_{2,L}^{\otimes j} \otimes A_{2,L}) \neq 0$ for $j \geq 4$. This implies that $R^1 (\operatorname{id}_{X^4} \times h)_* (\mathcal{J}  \otimes (L^{\boxtimes 4} \boxtimes T_{2,L})) \neq 0$. Therefore, we obtain
$$
H^1(X^{[2]}, M_{2,L}^{\otimes j} \otimes A_{2,L}) \cong H^1(X^4 \times X^{[2]}, \mathcal{J} \otimes (L^{\boxtimes 4} \boxtimes T_{2,L})) \neq 0.
$$
\end{remark}

In the remaining of this section, we show an effective version of Theorem \ref{main-vanishing-serre} for $k=2$ and $i \geq j$. This won't be needed for the main results of the paper, but it gives an effective result for the projective normality of $2$-secant varieties (see Theorem \ref{effective-projective-normal}). For this purpose, we also need some preparation.

\subsection{Transversality}
We give a brief discussion on transversality.

\begin{lemma}\label{sheaf-transversality}
    Let $X_1$ and $X_2$ be arbitrary projective schemes over a scheme $S$. Let $\mathcal{F}_1$ be 
    a coherent sheaf on $X_1$ and let $\mathcal{F}_2$ be a coherent sheaf on $X_2$. Assume that $X_2$ and $\mathcal{F}_2$ are flat over $S$. Then
    \[
    \STor_i ^{\mathcal{O}_{X_1\times_S X_2}}(\operatorname{pr}_1^* \mathcal{F}_1, \operatorname{pr}_2 ^* \mathcal{F}_2)=0~~\text{ for $i>0$.}
    \]
\end{lemma}
\begin{proof} 
Since the statement is local, we may assume that $X_1$, $X_2$ and $S$ are affine.
Let $E_{\bullet}$ be a projective $\mathcal{O}_{X_1}$-resolution of $\mathcal{F}_1$.
Since $\operatorname{pr}_1 \colon X_1 \times_S X_2\lra X_1$ is flat, $\operatorname{pr}_1 ^* E_{\bullet}$ is a projective $\mathcal{O}_{X_1 \times_S X_2}$-resolution of 
$\operatorname{pr}_1 ^*\mathcal{F}_1$. 
Since $\mathcal{F}_2$ is flat over $S$, we have that $\operatorname{pr}_1 ^* E_\bullet \otimes_{\mathcal{O}_{X_1 \times_S X_2}} \operatorname{pr}_2^* \mathcal{F}_2$ is exact, hence the result follows. 
\end{proof}

\begin{lemma}\label{subscheme-transversality}
    Let $X_1$ and $X_2$ be arbitrary irreducible projective varieties flat over an irreducible projective variety $S$. Let $Z_i\subseteq X_i$ be subschemes for $1\leq i\leq 2$, flat over $S$. Then on $X_1\times_S X_2$ we have
    \[
    \operatorname{pr}_{1}^* \mathcal{I}_{Z_1}\otimes \operatorname{pr}_{2}^* \mathcal{I}_{Z_2}\cong 
    \mathcal{I}_{\operatorname{pr}_{1}^{-1} Z_1}\cdot \mathcal{I}_{\operatorname{pr}_{2}^{-1} Z_2}=
    \mathcal{I}_{\operatorname{pr}_{1}^{-1} Z_1 \cup \operatorname{pr}_{2}^{-1} Z_2}.
    \]
\noindent If furthermore $\operatorname{Bl}_{Z_i} X_i \lra S$ is flat for some $i$, then we have
\[
\operatorname{Bl}_{\operatorname{pr}_1 ^{-1} Z_1 \cup \operatorname{pr}_2 ^{-1} Z_2} X_1\times_S X_2\cong \operatorname{Bl}_{Z_1} X_1 \times_S \operatorname{Bl}_{Z_2} X_2.
\]
\end{lemma}

\begin{proof}
Since $\operatorname{pr}_{i}$ is flat, we have
$\operatorname{pr}_{i}^* \mathcal{I}_{Z_i}\cong \mathcal{I}_{\operatorname{pr}_{1}^{-1} Z_i}$
for $1\leq i\leq 2$. 
Consider the short exact sequence
\[
0\lra \mathcal{I}_{\operatorname{pr}_{1}^{-1} Z_1}
\lra \mathcal{O}_{X_1\times_S X_2}\lra
\mathcal{O}_{\operatorname{pr}_{1}^{-1} Z_1}
\lra 0.
\]
\noindent Taking the tensor product with $\mathcal{I}_{\operatorname{pr}_{2}^{-1} Z_2}$ and using 
Lemma \ref{sheaf-transversality}, we get
\[
\mathcal{I}_{\operatorname{pr}_{1}^{-1} Z_1}
\otimes
\mathcal{I}_{\operatorname{pr}_{2}^{-1} Z_2}
\cong 
\mathcal{I}_{\operatorname{pr}_{1}^{-1} Z_1}
\cdot
\mathcal{I}_{\operatorname{pr}_{2}^{-1} Z_2}.
\]

\noindent Taking the tensor product with 
$\mathcal{O}_{\operatorname{pr}_{2}^{-1} Z_2}$ and using 
Lemma \ref{sheaf-transversality}, we get
\[
\mathcal{I}_{\operatorname{pr}_{1}^{-1} Z_1}
\cdot
\mathcal{I}_{\operatorname{pr}_{2}^{-1} Z_2}
=
\mathcal{I}_{\operatorname{pr}_{1}^{-1} Z_1}
\cap
\mathcal{I}_{\operatorname{pr}_{2}^{-1} Z_2}.
\]

\noindent The second part follows from 
\cite[\href{https://stacks.math.columbia.edu/tag/080A}{Tag 080A}]{stacks-project} and
\cite[\href{https://stacks.math.columbia.edu/tag/0805}{Tag 0805}]{stacks-project}.
\end{proof}

\begin{lemma}\label{fiber-product-lci}
    Let $X_1$, $X_2$ and $S$ be smooth varieties. Let $f_1 \colon X_1\lra S$ and $f_2 \colon X_2\lra S$ be two flat surjective morphisms. Then $X_1\times_S X_2$ is a local complete intersection variety. 
\end{lemma}
\begin{proof}
 Consider the following Cartesian diagram
 \begin{center}
\begin{tikzcd}
    X_1\times_S X_2 \ar[r]\ar[d]& X_1\times X_2\ar[d]\\
    \Delta_S \ar[r]& S\times S
\end{tikzcd}.
\end{center}

\noindent Since $\Delta_S$ is smooth and $f_1\times f_2$ is flat, we have that $X_1\times_S X_2$ is a local complete intersection variety. 
\end{proof}

\begin{remark}
A simple inductive argument shows that the conclusions of Lemma \ref{subscheme-transversality} and Lemma \ref{fiber-product-lci}
hold for any number of varieties. 
\end{remark}

\subsection{Blow-ups of $X \times X^{[1,2]}$}
We consider the blow-ups $\operatorname{Bl}_{\mathcal{V}_{1,2}} X\times X^{[1,2]}$ and $\operatorname{Bl}_{\mathcal{W}_{1,2}} X\times X^{[1,2]}$ of $X \times X^{[1,2]}$ along $\mathcal{V}_{1,2}$ and $\mathcal{W}_{1,2}$, respectively. Note that ${\mathcal{V}_{1,2}} \cong X^{[1,2]}$ is smooth but $\mathcal{W}_{1,2} = {\mathcal{V}_{1,2}} \cup \Gamma_{\res}$ is reducible  (see \S \ref{nestedHilb}). In particular, $\operatorname{Bl}_{\mathcal{V}_{1,2}} X\times X^{[1,2]}$ is smooth. On the other hand, $\operatorname{Bl}_{\mathcal{W}_{1,2}} X\times X^{[1,2]}$ is the triple nested Hilbert scheme $X^{[1,2,3]} \subseteq X \times X^{[2]} \times X^{[3]}$ consisting of triples $(x,\xi,\eta)$ such that $x \in \xi \subseteq \eta$. 

\begin{lemma}\label{123}
We have that 
\[
X^{[1,2,3]} \cong \operatorname{Bl}_{\mathcal{W}_{1,2}} X\times X^{[1,2]}\cong X^{[1,2]}\times_{X^{[2]}} X^{[2,3]}
\]
is normal with canonical local complete intersection singularities.
\end{lemma}

\begin{proof}
Since $\rho \colon X^{[1,2]}\lra X^{[2]}$ is flat, the commutative diagram
\begin{center}
\begin{tikzcd}
    \operatorname{Bl}_{\mathcal{W}_{1,2}} X\times X^{[1,2]}\ar[r]\ar[d]& X^{[2,3]}\ar[d, "\bl"]\\
    X\times X^{[1,2]}\ar[r, "\operatorname{id}_X \times \rho"]& X\times X^{[2]}
\end{tikzcd}
\end{center}
is Cartesian by \cite[\href{https://stacks.math.columbia.edu/tag/0805}{Tag 0805}]{stacks-project}. Thus $\operatorname{Bl}_{\mathcal{W}_{1,2}} X\times X^{[1,2]}\cong X^{[1,2]}\times_{X^{[2]}} X^{[2,3]}$. Then $X^{[1,2,3]}$ has local complete intersection singularities by Lemma \ref{fiber-product-lci}. Now, consider the Cartesian diagram
\begin{center}
\begin{tikzcd}
    X^{[1,2,3]} \ar[r]\ar[d]& X^{[2,3]}\ar[d, "\tau"]\\
     X^{[1,2]}\ar[r, "\rho"]&  X^{[2]}
\end{tikzcd}.
\end{center}
Since $\tau \colon X^{[2,3]}\lra X^{[2]}$ is flat, we see that the map $X^{[1,2,3]} \lra X^{[1,2]}$ is also flat. Note that every fiber of $X^{[1,2,3]} \lra X^{[1,2]}$ has rational singularities and $X^{[1,2]}$ is smooth. Thus $X^{[1,2,3]}$ has rational singularities by \cite[Th\'eor\`eme 5]{Elkik.78}. Therefore $X^{[1,2,3]}$ has canonical singularities by \cite[Corollary 5.24]{Kollar-Mori}. 
\end{proof}

\subsection{Effective global vanishing}
In this subsection, we let $L:=\omega_X \otimes A^m \otimes B$, where $A$ is a very ample line bundle and $B$ is a nef line bundle on $X$, and $C$ be a nef line bundle on $X^{[1,2]}$.
It is well-known that $\omega_X \otimes A^{\ell} \otimes B$ is base point free as soon as $\ell \geq n+1$  (for example, see \cite[Example 1.8.23]{pag1}). By \cite[Theorem 1.1]{HTT}, the line bundle $\omega_X \otimes A^{\ell} \otimes B$ separates $k$-schemes when $\ell \geq n+k$. 
Our aim is to prove the following effective version of the main vanishing theorem. 

\begin{theorem}\label{main-effective-vanishing}
If $m\geq 2n+j$, then
\[
H^i(X^{[2]}, \wedge^j M_{2,L}\otimes A_{2,L})=0~~\text{ for all $i\geq j \geq 1$.}
\]
\end{theorem}

\begin{lemma}\label{pullback-tau-j=1case}
If $m \geq 2n+1$, then
\[
H^i(X^{[1,2]}, \tau^* M_L \otimes \rho^* N_{2,\omega_X \otimes A^{2n+1}} \otimes C)=0~~\text{ for all $i \geq 1$.} 
\]    
\end{lemma}

\begin{proof}
Consider the blow-up $b \colon \operatorname{Bl}_{\mathcal{V}_{1,2}} X \times X^{[1,2]} \lra  X \times X^{[1,2]}$ with the exceptional divisor $E$, and recall that $\operatorname{Bl}_{\mathcal{V}_{1,2}} X \times X^{[1,2]}$ is smooth. Notice that the projection map $p \colon \operatorname{Bl}_{\mathcal{V}_{1,2}} X \times X^{[1,2]} \lra X^{[1,2]}$ is flat and 
$$
R^i p_* b^*(L \boxtimes \OO_{X^{[1,2]}})(-E) \cong \begin{cases} \tau^* M_L & \text{if $i=0$} \\ 0 & \text{if $i\geq 1$}. \end{cases}
$$
We need to show that
$$
H^i(\operatorname{Bl}_{\mathcal{V}_{1,2}} X \times X^{[1,2]}, b^*((L \boxtimes \rho^* N_{2, \omega_X \otimes A^{2n+1}}) \otimes C)(-E))=0~~\text{ for $i \geq 1$.}
$$

\noindent Note that $b^*(A\boxtimes \rho^* N_{2,A})(-E)$ is globally generated by Section \ref{global-generation-universal-family}.
We may then compute
\begin{multline*}
b^*(L \boxtimes \rho^* N_{2, \omega_X \otimes A^{2n+1}} \otimes C))(-E) \otimes \omega_{\operatorname{Bl}_{\mathcal{V}_{1,2}} X \times X^{[1,2]}}^{-1}\\
\cong b^* ( (A^{m-n} \otimes B) \boxtimes (\tau^* A \otimes \res^* A^{n+1} \otimes C)) \otimes b^*(A^n \boxtimes N_{2,A} ^n)(-nE).
\end{multline*}

\noindent The desired vanishing follows now from the Kawamata--Viehweg vanishing theorem.
\end{proof}

\begin{lemma}\label{pullback-tau}
If $m \geq 2n+1$, then
\[
H^i(X^{[1,2]}, \tau^* \wedge^j M_L \otimes \rho^* N_{2,\omega_X \otimes A^{2n+1}} \otimes C)=0~~\text{ for all $i \geq 1$ and $i \geq j$.} 
\]
\end{lemma}

\begin{proof}
We proceed by induction on $j$. If $j=0$, then the lemma follows from  the Kawamata--Viehweg vanishing theorem since
$$
\rho^* N_{2, \omega_X \otimes A^{2n+1}} \otimes C \otimes \omega_{X^{[1,2]}}^{-1} \cong \rho^* T_{2, A^{n+1}} \otimes C \otimes (\tau^* A^n \otimes \res^* A^n)(-nF_1)
$$
is nef and big. 
The case $j=1$ is treated in Lemma \ref{pullback-tau-j=1case}. Assume $j \geq 2$. 
Consider the short exact sequence
$$
0 \lra \tau^* \wedge^j M_L  \lra \tau^* (\wedge^j H^0(X, L) \otimes \OO_X) \lra \tau^* (\wedge^{j-1} M_L \otimes L) \lra 0.
$$
Tensoring by $\rho^* N_{2,\omega_X \otimes A^{2n+1}} \otimes C$ and taking cohomology, we reduce the problem to
$$
\begin{array}{l}
H^{i-1}(X^{[1,2]},\tau^* \wedge^{j-1} M_L \otimes \rho^* N_{2, \omega_X \otimes A^{2n+1}} \otimes (\tau^* L \otimes C))=0, \text{ and}\\
H^i(X^{[1,2]}, \rho^* N_{2, \omega_{X} \otimes A^{2n+1}} \otimes C)=0
\end{array}
$$
for $i \geq j \geq 2$. These follow from the inductive hypothesis.
\end{proof}

\begin{lemma}\label{pullback-r}
If $m\geq 2n+1$,
then
\[
H^i(X^{[1,2]}, \rho^* (M_{2,L}  \otimes N_{2,\omega_X \otimes A^{2n+1}}) \otimes C)=0~~\text{ for all $i\geq 1$.}
\]
\end{lemma}

\begin{proof}
Recall from  Lemma \ref{123} that $\operatorname{Bl}_{\mathcal{W}_{1,2}} X\times X^{[1,2]} \cong X^{[1,2]}\times_{X^{[2]}} X^{[2,3]}$. There are two projections $\pr_1 \colon X^{[1,2]}\times_{X^{[2]}} X^{[2,3]} \to X^{[1,2]}$ and $\pr_2 \colon X^{[1,2]}\times_{X^{[2]}} X^{[2,3]} \to X^{[2,3]}$. By \cite[Chapter III, Proposition 9.3]{Hartshorne.77}, we have
\[
R^i \operatorname{pr}_{1,*} \operatorname{pr}_2 ^* (\operatorname{res}^* L (-F_2)) \cong \begin{cases} \rho^* M_{2,L} & \text{if $i=0$} \\ 0 & \text{if $i\geq 1$}. \end{cases}
\]
Then it is enough to show that
\[
H^i(X^{[1,2]}\times_{X^{[2]}} X^{[2,3]}, \operatorname{pr}_2 ^*( \operatorname{res}^* L (-F_2)) \otimes \operatorname{pr}_1^* (\rho^* N_{2,\omega_X\otimes A^{2n+1}}\otimes C))= 0~~\text{ for $i\geq 1$}.
\]
As $\omega_{X^{[1,2]}\times_{X^{[2]}} X^{[2,3]}}\cong \operatorname{pr}_2 ^* \omega_{X^{[2,3]}}(\operatorname{pr}_1 ^* F_1)$, we have that
\begin{multline*}
\operatorname{pr}_2 ^*( \operatorname{res}^* L (-F_2)) \otimes \operatorname{pr}_1^* (\rho^* N_{2,\omega_X\otimes A^{2n+1}}\otimes C) \otimes \omega_{X^{[1,2]}\times_{X^{[2]}} X^{[2,3]}}^{-1}\\
\cong \operatorname{pr}_1^* C  \otimes \operatorname{pr}_2^* (\rho^*(N_{3,A^2}^n \otimes T_{3,A})  \otimes \operatorname{res}^*(A^{m-2n-1} \otimes B))
\end{multline*}
is nef and big (note that $N_{3,A^2}$ is globally generated since $A^2$ separates $3$-schemes by \cite[Theorem 1.1]{HTT}). Since $X^{[1,2]}\times_{X^{[2]}} X^{[2,3]}$ has canonical singularities by Lemma \ref{123}, we may conclude by the Kawamata--Viehweg vanishing theorem.
\end{proof}

\begin{proposition}\label{effective-vanishing}
If $m\geq 2n+j$, then 
\[
H^i(X^{[1,2]}, \rho^* (\wedge^j M_{2,L}\otimes N_{2,\omega_X \otimes A^{2n+j}}) \otimes C)=0~~\text{ for all $i\geq j \geq 1$.}
\]
\end{proposition}

\begin{proof}
We proceed by induction on $j$. The case $j=1$ follows from Lemma \ref{pullback-r}. Assume $j \geq 2$. Consider the short exact sequence
\[
0\lra \rho^* \wedge^j M_{2,L}\lra \tau^* \wedge^j M_{L}\lra \rho^* \wedge^{j-1} M_{2,L}\otimes \operatorname{res}^* L (-F_1)\lra 0. 
\]
Twisting by $\rho^* N_{2,\omega_X \otimes A^{2n+j}} \otimes C$ and taking cohomology, we reduce the problem to
$$
\begin{array}{l}
H^{i-1}(X^{[1,2]}, \rho^* (\wedge^{j-1} M_{2,L} \otimes N_{2, \omega_X \otimes A^{2n+j-1}}) \otimes C \otimes \res^* L \otimes (\tau^* A \otimes \res^* A)(-F_1))=0, \text{ and}\\
H^i(X^{[1,2]}, \tau^* \wedge^j M_L \otimes \rho^* N_{2,\omega_X \otimes A^{2n+1}} \otimes C \otimes \rho^* T_{2, A^{j-1}})=0
\end{array}
$$
for $i \geq j \geq 2$. Note that $(\tau^* A \otimes \res^* A)(-F_1)$ is globally generated. Therefore, we are done from the inductive hypothesis and Lemma \ref{pullback-tau}.
\end{proof}

\begin{proof}[Proof of Theorem \ref{main-effective-vanishing}]
By Lemma \ref{lem:O_X^[k]directsummand}, the vector bundle $\wedge^j M_{2,L}\otimes A_{2,L}$ is a direct summand of the vector bundle $\rho_* \rho^* (\wedge^j M_{2,L} \otimes N_{2,L})$. The theorem is an immediate consequence of Lemma \ref{direct-image-res-r-k-1-k} and Proposition \ref{effective-vanishing}, applied with $C=\rho^* T_{2, A^{m-2n-j} \otimes B}$.
\end{proof}

\section{Singularities of secant varieties}\label{main-results-section}

This section is devoted to the proof of Theorem \ref{theoremA}. We assume that $L$ is a sufficiently positive line bundle on $X$ giving an embedding $X \subseteq \mathbb{P} H^0(X, L) = \mathbb{P}^r$, unless otherwise stated. We also assume that $n=\dim (X) \leq 2$ or $k \leq 3$. Under these assumptions, $X^{[k]}$ and $X^{[k-1,k]}$ are smooth, so the $k$-secant variety $\sigma_k$ (respectively, the relative secant variety $\sigma_{m,k}$) coincides with the $k$-cactus scheme $\kappa_k$ (respectively, the relative cactus scheme $\kappa_{m,k}$).

\subsection{Normality and projective normality}
We prove the normality and projective normality of secant varieties by induction on $k$.

\begin{lemma}\label{lem:alpha_*O_Z}
If $\sigma_{k-1}$ is normal, then $\alpha_{k,*}\mathcal{O}_{Z_{k}} \cong \mathcal{O}_{\sigma_{k-1}}$.
\end{lemma}

\begin{proof}
The morphism $\alpha_k$ fits in the following commutative diagram:
\begin{center}
    \begin{tikzcd}
        B^{k-1,k}\ar[r, "\alpha_{k-1,k}"]\ar[d, swap, "\tilde{\tau}"]& Z_{k}\ar[r,symbol=\subseteq]\ar[d]& B^k\ar[d, "\alpha_k"]\\
        B^{k-1}\ar[r, swap, "\alpha_{k-1}"]&
    \sigma_{k-1}\ar[r,symbol=\subseteq]&
        \sigma_k
    \end{tikzcd}.
\end{center}
Consider the composite map $\mathcal{O}_{\sigma_{k-1}}\hookrightarrow \alpha_{k,*} \mathcal{O}_{Z_k} \hookrightarrow \alpha_{k,*}  \alpha_{k-1,k,*} \mathcal{O}_{B^{k-1,k}}$. 
We have
\begin{equation}\label{eq:alpha_*O_B^{k-1,k}}
\alpha_{k,*}  \alpha_{k-1,k,*} \mathcal{O}_{B^{k-1,k}} \cong \alpha_{k-1,*}\widetilde{\tau}_* \mathcal{O}_{B^{k-1,k}} \cong \alpha_{k-1,*} \mathcal{O}_{B^{k-1}} \cong \mathcal{O}_{\sigma_{k-1}},
\end{equation}
so the composite map is an isomorphism.
\end{proof}

\begin{lemma}\label{lem:S^lH^0(O_B(1))->H^0(O_B(l))}
If $\sigma_{k-1} \subseteq \mathbb{P}^r$ is projectively normal, then the map $S^\ell H^0(B^k, \mathcal{O}_{B^k}(1)) \lra H^0(B^k, \mathcal{O}_{B^k}(\ell))$ is surjective for every $\ell \geq 0$.
\end{lemma}

\begin{proof}
This map is an isomorphism for $1 \leq \ell \leq k$ by Danila's theorem (see Theorem \ref{Danila}). It is therefore enough to show that the map 
\[
H^0(B^k, \mathcal{O}_{B^k}(k))\otimes 
H^0(B^k, \mathcal{O}_{B^k}(\ell))
\lra
H^0(B^k, \mathcal{O}_{B^k}(k+\ell))
\]
is surjective for all $\ell \geq 1$. Consider the short exact sequence
$$
0\lra \mathcal{O}_{B^k}(-Z_{k})\lra \mathcal{O}_{B^k}\lra \mathcal{O}_{Z_{k}}\lra 0.
$$
By Proposition \ref{A&Z}, we have
$$
\pi_{k,*}(\mathcal{O}_{B^k}((-Z_{k}+kH_k)+\ell H_k))\cong A_{k,L}\otimes S^\ell E_{k,L}.
$$
Lemma \ref{lem:alpha_*O_Z} implies that $H^0(Z_k, \mathcal{O}_{Z_k}(k+\ell)) \cong H^0(\sigma_{k-1}, \mathcal{O}_{\sigma_{k-1}}(k+\ell))$ for all $\ell \geq 0$.
Twisting the short exact sequence by $\mathcal{O}_{B^k}(k+\ell)$ and $H^0(B^k, \mathcal{O}_{B^k}(\ell))\otimes \mathcal{O}_{B^k}(k)$, and taking the global sections, we get a commutative diagram
\begin{center}
\begin{tikzcd}
0 \ar[r] & H^0(A_{k,L})\otimes H^0(S^\ell E_{k,L}) \ar[r] \ar[d] & H^0(kH_k)\otimes H^0(\ell H_k) \ar[r] \ar[d] & H^0(\OO_{\sigma_{k-1}}(k))\otimes H^0(\ell H_k) \ar[d] \\
0 \ar[r] & H^0(A_{k,L}\otimes S^\ell E_{k,L}) \ar[r] & H^0((k+\ell)H_k) \ar[r] & H^0(\OO_{\sigma_{k-1}}(k+\ell)) 
\end{tikzcd}.
\end{center}
By Corollary \ref{corollary-vanishing-serre}, the left vertical map is surjective. Since $\sigma_{k-1}$ is projectively normal, the two right horizontal maps are surjective as well (note that, alternatively, this can also be deduced by Corollary \ref{corollary-vanishing-serre}).
Since $\sigma_{k-1} \subseteq \mathbb{P}^r$ is projectively normal, the right-most vertical map is surjective. Thus the middle vertical map is surjective, as desired.
\end{proof}

\begin{theorem}\label{thm:normalsing}
The following hold:
\begin{enumerate}[topsep=0pt]
    \item $\sigma_k$ is normal.
    \item The restriction map $H^0(\mathbb{P}^r, \mathcal{O}_{\mathbb{P}^r}(\ell))\lra H^0(\sigma_k,\mathcal{O}_{\sigma_k}(\ell))$ is surjective for every $\ell\geq 0$.
\end{enumerate}
In particular, $\sigma_k \subseteq \mathbb{P}^r$ is projectively normal.
\end{theorem}

\begin{proof}
We proceed by induction on $k$. When $k=1$, part (1) is obvious and (2) follows from \cite[Theorem 1]{Ein.Lazarsfeld}. We assume $k \geq 2$. Letting $\mathcal{Q}:=\alpha_{k,*} \mathcal{O}_{B^k}/\mathcal{O}_{\sigma_k}$, we have a short exact sequence
\[
0\lra \mathcal{O}_{\sigma_k}\lra \alpha_{k,*}\mathcal{O}_{B^k}\lra \mathcal{Q}\lra 0.
\]
By inductive hypothesis and Lemma \ref{lem:S^lH^0(O_B(1))->H^0(O_B(l))}, the map $S^{\ell} H^0(\sigma_k, \mathcal{O}_{\sigma_k}(1)) \lra H^0(\sigma_k,\alpha_{k,*}\mathcal{O}_{B^k}(\ell))$ is surjective for each $\ell \geq 0$. But this surjection factors through $H^0(\sigma_k, \mathcal{O}_{\sigma_k}(\ell))$, so (2) holds. Then Serre's vanishing theorem yields $H^0(\sigma_k, \mathcal{Q}(\ell))=0$ for $\ell \gg 0$. This implies $\mathcal{Q}=0$, so (1) holds.
\end{proof}

The first statement of the following theorem is the Du Bois-type condition, which is key for the remaining assertions of Theorem \ref{theoremA} and Theorem \ref{thm:N_{k+2,p}}. 

\begin{theorem}\label{Du Bois-type condition}
The following hold:
\begin{enumerate}[topsep=0pt]
    \item $R^i\alpha_{k,*} \mathcal{O}_{B^k}(-Z_k) \cong \begin{cases} \mathcal{I}_{\sigma_{k-1}/\sigma_k} & \text{if $i=0$} \\
0 & \text{if $i \geq 1$}.\end{cases}$
    \item $H^i(\sigma_k, \OO_{\sigma_k}(\ell))=0$ for $i\geq 1$ and $\ell\geq 1$. In particular, we have $\operatorname{reg} (\mathcal{O}_{\sigma_k}) \leq d_k+1 = nk+k$.
\end{enumerate}
\end{theorem}

\begin{proof}
Recall the isomorphism $\mathcal{O}_{B^k}(\ell H_k - Z_k) \cong \mathcal{O}_{B^k}(\ell-k) \otimes \pi_k^* A_{k,L}$. By projection formula and Corollary \ref{corollary-vanishing-serre}, we have
\begin{equation}\label{eq:H^i=0projnormalproof}
H^i(B^k, \mathcal{O}_{B^k}(\ell H_k - Z_k) )=0~~\text{ for $i \geq 1$ and $\ell \geq 1$}.
\end{equation}
This in turn implies
$R^i \alpha_{k,*} \mathcal{O}_{B^k}(-Z_k) = 0$ for $i \geq 1$, by Lemma \ref{lem:standfactforhdi}.
Applying $\alpha_{k,*}$ to the short exact sequence
$$
0\lra \mathcal{O}_{B^k}(-Z_{k})\lra \mathcal{O}_{B^k}\lra \mathcal{O}_{Z_{k}}\lra 0,
$$
we get $\alpha_{k,*} \mathcal{O}_{B^k}(-Z_k) = \mathcal{I}_{\sigma_{k-1}/\sigma_k}$ by Lemma \ref{lem:alpha_*O_Z} and Theorem \ref{thm:normalsing}. Therefore (1) holds. Next, we prove (2) by induction on $k$. The case $k=1$ is obvious, hence we may assume $k\geq 2$. Consider the short exact sequence
$$
0 \longrightarrow \cI_{\sigma_{k-1}/\sigma_k} \longrightarrow \OO_{\sigma_k} \longrightarrow \OO_{\sigma_{k-1}} \longrightarrow 0.
$$
By (1) and (\ref{eq:H^i=0projnormalproof}), we have
$$
H^i(\sigma_k, \cI_{\sigma_{k-1}/\sigma_k}(\ell)) \cong H^i(B^k, \OO_{B^k}(\ell H_k-Z_{k}))=0~~\text{ for $i\geq 1$ and $\ell\geq 1$}.
$$
As $H^i(\sigma_{k-1}, \mathcal{O}_{\sigma_{k-1}}(\ell))=0$ for $i \geq 1$ and $\ell \geq 1$ by inductive hypothesis, we get (2).
\end{proof}

\begin{corollary}\label{serre-on-secant}
$H^i(B^k, \omega_{B^k}(\ell H_k + Z_k))\cong H^{d_k-i}(\sigma_k , \mathcal{I}_{\sigma_{k-1}/\sigma_k}(-\ell))^{\vee}$ for all $i$ and $\ell$. 
\end{corollary}

\begin{proof}
It is an immediate consequence of Theorem \ref{Du Bois-type condition} and Serre duality.
\end{proof}

\begin{remark}\label{rem:effectivefornormalsing}
The above two theorems hold if $L$ separates $2k$-schemes and
$$
H^i(X^{[m]}, \wedge^j M_{m,L} \otimes A_{m,L})=0~~\text{ for $i \geq 1$, $0 \leq j \leq i$ and $1 \leq m \leq k$.}
$$ 
Furthermore, it is easy to deduce $H^i(X, L^\ell)=0$ for $i \geq 1$ and $\ell \geq 1$ from this vanishing, applied with $m=1$.
For the remaining of this section, it is sufficient to assume that $L$ separates $2k$-schemes, $H^i(X, L^\ell)=0$ for $i \geq 1$ and $\ell \geq 1$ and the conclusions of the above two theorems hold.
\end{remark}

Since we have an effective version of the main vanishing theorem for $k=2$ (that is, Theorem \ref{main-effective-vanishing}), Remark \ref{rem:effectivefornormalsing} implies Theorem \ref{intro-effective}.

\begin{theorem}\label{effective-projective-normal}
Set $L:=\omega_X \otimes A^m \otimes B$, where $A$ is very ample and $B$ is nef. If $m \geq 4n$, then $\sigma_2 \subseteq \mathbb{P}^r$ is projectively normal and $H^i(\sigma_2, \OO_{\sigma_2}(\ell))=0$ for $i \geq 1$ and $\ell \geq 1$.
\end{theorem}

We remind the reader that $\sigma_2$ has normal singularities and the Du Bois-type condition (that is, Theorem \ref{Du Bois-type condition} (1)) holds for $k=2$ as soon as $m \geq 2n+2$ by Ullery \cite{Ullery.16} and Chou--Song \cite{Chou.Song.18}.

\subsection{Du Bois singularities} 
The Du Bois-type condition implies the following:

\begin{theorem}\label{thm:dubois}
$\sigma_k$ has Du Bois singularities.
\end{theorem}
\begin{proof}
We proceed by induction on $k$ with the case $k=1$ being trivial. We assume that $k \geq 2$ and $\sigma_{k-1}$ has Du Bois singularities.
As in \cite[Proof of Theorem 6.27]{Kollar}, we have a commutative diagram
$$
\xymatrix{
\cI_{\sigma_{k-1}/\sigma_k} \ar[r]^-{\simeq_\text{qis}} \ar[d]_-{\varphi} & \mathcal{R} \alpha_{k,*} \OO_{B^k}(-Z_{k}) \ar[d]^-{\psi} \\
\underline{\Omega}^0_{\sigma_{k-1}, \sigma_k} \ar[r] & \mathcal{R} \alpha_{k,*} \underline{\Omega}^0_{Z_{k}, B^k}
}
$$
in the derived category $\mathbf{D}^b(\sigma_k)$, where the upper horizontal map is a quasi-isomorphism by the Du Bois-type condition (Theorem \ref{Du Bois-type condition} (1)). If $\psi$ has a left inverse, then so does $\varphi$. By \cite[Corollary 6.24]{Kollar}, this implies that $(\sigma_k, \sigma_{k-1})$ is a Du Bois pair, so by \cite[Proposition 6.15]{Kollar}, $\sigma_k$ has Du Bois singularities since $\sigma_{k-1}$ has Du Bois singularities. Thus it is enough to show that $\psi$ has a left inverse. The map $\psi$ fits in the following commutative diagram
\begin{equation}\label{eq:commdiagproofofDuBois}
\begin{gathered}
\xymatrix{
\mathcal{R} \alpha_{k,*} \OO_{B^k}(-Z_{k}) \ar[r] \ar[d]^-{\psi} & \mathcal{R} \alpha_{k,*} \OO_{B^k} \ar[r] \ar[d]^-{\simeq_\text{qis}} & \mathcal{R} \alpha_{k,*} \OO_{Z_{k}} \ar[r]^-{+1} \ar[d]^-{\lambda} & \\
\mathcal{R} \alpha_{k,*} \underline{\Omega}_{Z_{k}, B^k}^0 \ar[r] & \mathcal{R} \alpha_{k,*} \underline{\Omega}_{B^k}^0 \ar[r] & \mathcal{R} \alpha_{k,*} \underline{\Omega}_{Z_{k}}^0 \ar[r]^-{+1}&
}
\end{gathered}
\end{equation}
in the derived category $\mathbf{D}^b(\sigma_k)$. 
Recall that $\alpha_{k-1,k} \colon B^{k-1,k} \lra Z_{k}$ is a resolution of singularities and $\alpha_{k,*} \alpha_{k-1,k,*} \OO_{B^{k-1,k}} \cong \OO_{\sigma_{k-1}}$
by (\ref{eq:alpha_*O_B^{k-1,k}}). The map $\lambda$ fits in the following commutative diagram
$$
\xymatrix{
\mathcal{R} \alpha_{k,*} \OO_{Z_{k}} \ar[r]^-{\lambda}  \ar[d] &  \mathcal{R} \alpha_{k,*}  \underline{\Omega}_{Z_{k}}^0 \ar[d] \\
  \mathcal{R} \alpha_{k,*}  \mathcal{R} \alpha_{k-1,k,*} \OO_{B^{k-1,k}} \ar[r]^-{\simeq_\text{qis}}  & \mathcal{R} \alpha_{k,*}  \mathcal{R} \alpha_{k-1,k,*} \underline{\Omega}_{B^{k-1,k}}^0  
}
$$
in the derived category $\mathbf{D}^b(\sigma_k)$. Recall from Lemma \ref{lem:alpha_*O_Z} and Theorem \ref{thm:normalsing} that $\alpha_{k,*} \mathcal{O}_{Z_k} \cong \mathcal{O}_{\sigma_{k-1}}$. Taking cohomology, we get a commutative diagram
\begin{center}
\begin{tikzcd}
\OO_{\sigma_{k-1}} \ar[r] \ar[d, equal] & h^0(\mathcal{R}\alpha_{k,*} \underline{\Omega}_{Z_{k}}^0 )\ar[d] \\
\OO_{\sigma_{k-1}} \ar[r, equal] & \OO_{\sigma_{k-1}}
\end{tikzcd}.
\end{center}
Thus the upper horizontal map $\OO_{\sigma_{k-1}} \lra h^0(\mathcal{R}\alpha_{k,*} \underline{\Omega}_{Z_{k}}^0 )$ is injective. Then $\lambda \colon \mathcal{R} \alpha_{k,*} \OO_{Z_k} \lra \mathcal{R} \alpha_{k,*} \underline{\Omega}^0_{Z_k}$ can be realized as a morphism of complexes 
$$
(0 \to A^0 \to A^1 \to \cdots ) \xrightarrow{~\lambda~} (0 \to B^0 \to B^1 \to \cdots)
$$
in such a way that $\lambda^0 \colon A^0 \lra B^0$ is injective. 
Now, the exact triangle
$$
\mathcal{R} \alpha_{k,*} \OO_{B^k}(-Z_{k}) \longrightarrow \mathcal{R} \alpha_{k,*} \OO_{B^k} \longrightarrow \mathcal{R} \alpha_{k,*} \OO_{Z_{k}} \xrightarrow{+1}
$$
in the upper side of (\ref{eq:commdiagproofofDuBois}) can be realized as a complex of complexes
$$
(0 \to \cI_{\sigma_{k-1}/\sigma_{k}} \to 0 \to \cdots) \longrightarrow (0 \to \cI_{\sigma_{k-1}/\sigma_{k}} \oplus A^0 \to A^1 \to \cdots) \longrightarrow (0 \to A^0 \to A^1 \to \cdots),
$$
where the middle complex is the mapping cone of the map $\mathcal{R} \alpha_{k,*} \OO_{Z_{k}}[-1] \lra \mathcal{R} \alpha_{k,*} \OO_{B^k}(-Z_{k})$ and the differential map $\cI_{\sigma_{k-1}/\sigma_{k}} \oplus A^0 \lra A^1$ is given by $(u,v)\longmapsto d_A(v)$,
such that the degree $0$ part of the complex of complexes
$$
0 \longrightarrow \cI_{\sigma_{k-1}/\sigma_{k}} \longrightarrow \cI_{\sigma_{k-1}/\sigma_{k}} \oplus A^0 \longrightarrow A^0 \longrightarrow 0
$$
is a splitting short exact sequence. 
Then the exact triangle
$$
\mathcal{R} \alpha_{k,*} \underline{\Omega}_{Z_{k}, B^k}^0 \longrightarrow \mathcal{R} \alpha_{k,*} \underline{\Omega}_{B^k}^0 \longrightarrow \mathcal{R} \alpha_{k,*} \underline{\Omega}_{Z_{k}}^0 \xrightarrow{+1}
$$
in the lower side of (\ref{eq:commdiagproofofDuBois}) can be realized as a complex of complexes
$$
(0 \to \cI_{\sigma_{k-1}/\sigma_{k}} \oplus A^0 \to B^0 \oplus A^1 \to \cdots) \longrightarrow (0 \to \cI_{\sigma_{k-1}/\sigma_{k}} \oplus A^0 \to A^1 \to \cdots) 
\longrightarrow (0 \to B^0 \to B^1 \to \cdots),
$$
where the first complex is $-1$ shifting of the mapping cone of the map $\mathcal{R} \alpha_{k,*} \underline{\Omega}_{B^k}^0 \longrightarrow \mathcal{R} \alpha_{k,*} \underline{\Omega}_{Z_{k}}^0$ and the differential map $\cI_{\sigma_{k-1}/\sigma_{k}} \oplus A^0 \lra B^0 \oplus A^1$ is given by $(u,v) \longmapsto (\lambda^0(v), d_A(v))$. Since $\psi$ is now the natural inclusion map of complexes
$$
(0 \to \cI_{\sigma_{k-1}/\sigma_{k}} \to 0 \to \cdots) \xrightarrow{~\psi~} (0 \to \cI_{\sigma_{k-1}/\sigma_{k}} \oplus A^0 \to B^0 \oplus A^1 \to \cdots),
$$
it follows that $\psi$ has a left inverse in the derived category $\mathbf{D}^b(\sigma_k)$.
\end{proof}

\begin{remark}
If $Z_k$ has Du Bois singularities, then the Du Bois-type condition (Theorem \ref{Du Bois-type condition} (1)) implies that $\sigma_k$ has Du Bois singularities by \cite[Theorem 6.27]{Kollar}. This is the case when $k=2$. However, when $X$ is a surface and $k \geq 9$, using Scala's result \cite[Theorem 3.13]{Scala.19} that the isospectral Hilbert scheme $X_k$ does not have log canonical singularities, one may see that $Z_k$ does not have Du Bois singularities. 
\end{remark}

\subsection{Cohen--Macaulay singularities}
The aim of this subsection is to prove the following theorem, which follows from Proposition \ref{not-CM} and Theorem \ref{thm:CM}. 

\begin{theorem}\label{thm:CM<=>H^i(O_X)=0}
Assume that $k \geq 2$. Then 
$\sigma_k$ is Cohen--Macaulay if and only if $H^i(X, \mathcal{O}_X)=0$ for $1 \leq i \leq n-1$.
\end{theorem}

It is well-known that $\sigma_k$ is Cohen--Macaulay if and only if $H^i(\sigma_k, \mathcal{O}_{\sigma_k}(-\ell))=0$ for $1 \leq i \leq d_k-1$ and $\ell \gg 0$ (see \cite[Corollary 5.72]{Kollar-Mori}). In order to take an inductive approach on $k$, Corollary \ref{serre-on-secant} and Lemma \ref{dualizing-sheaf-secant-variety} below suggest to study $R^i \alpha_{k,*} \omega_{B^k}(Z_{k})$. 
For Lemma \ref{dualizing-sheaf-secant-variety}, we closely follow the proof of \cite[Claim 5.9]{Chou.Song.18}, where the case $k=2$ is treated. 

\begin{lemma}\label{dualizing-sheaf-secant-variety}
We have $\alpha_{k,*}\omega_{B^k}(Z_{k}) \cong \omega_{\sigma_k}$.
\end{lemma}

\begin{proof}
Let $j \colon \sigma_k \setminus \sigma_{k-1} \hookrightarrow \sigma_k$ be the inclusion map. We obtain a homomorphism between our sheaves of interest via the following composition
 \[
 \alpha_{k,*}\omega_{B^k}(Z_{k}) \longrightarrow j_*j^*\alpha_{k,*}\omega_{B^k}(Z_{k}) \longrightarrow j_*j^*\omega_{\sigma_k} \cong \omega_{\sigma_k}.
 \]
It suffices to show that this homomorphism gives isomorphisms of the stalk of every closed point --- i.e., we need to show that
\[
\alpha_{k,*}\omega_{B^k}(Z_{k})\otimes  \mathcal{O}_{\sigma_k, x}\longrightarrow \omega_{\sigma_k}\otimes \mathcal{O}_{\sigma_k, x}
\]
is an isomorphism for any closed point $x \in \sigma_k$, where $\mathcal{O}_{\sigma_k, x}$ is the local ring at $x$. At the level of complexes, we observe the following quasi-isomorphisms:
\begin{align*}
    \mathcal{R}\alpha_{k,*}\omega_{B^k}(Z_{k}) \otimes \mathcal{O}_{\sigma_k, x} & \simeq_\text{qis} \mathcal{R}\SHom_{\sigma_k}(R\alpha_{k,*}\mathcal{O}_{B^k}(-Z_{k}), \omega_{\sigma_k}^{\bullet}) \otimes \mathcal{O}_{\sigma_k, x} \\
    & \quad\text{(by coherent duality \cite[Chapter III, Theorem 11.1]{Hartshorne.66})}\\
    & \simeq_\text{qis} \mathcal{R}\SHom_{\sigma_k}(\alpha_{k,*}\mathcal{O}_{B^k}(-Z_{k}), \omega_{\sigma_k}^{\bullet}) \otimes \mathcal{O}_{\sigma_k, x}\\
    & \quad\text{(by Part (1) of Theorem \ref{Du Bois-type condition})}\\
    & \simeq_\text{qis}\operatorname{RHom}(\alpha_{k,*}\mathcal{O}_{B^k}(-Z_{k})_x, \omega_{\sigma_k, x}^{\bullet})\\
    & \quad\text{(by \cite[Lemma 3.1]{Kovacs.2011})}.
\end{align*}
(Note here that the dualizing complex of $\sigma_k$ can be obtained as $\omega_{\sigma_k}^{\bullet} = \mathcal{R}\SHom_{\mathbb{P}^r}(\mathcal{O}_{\sigma_k}, \omega_{\mathbb{P}^r}[r])$. However, we will not need this.) Taking $(-d_k)$-th cohomology on both sides, this gives
 \begin{equation*} 
\alpha_{k,*}\omega_{B^k}(Z_{k})_x \cong \operatorname{Ext}^{-d_k}(\alpha_{k,*}\mathcal{O}_{B^k}(-Z_{k})_x, \omega_{\sigma_{k},x}^{\bullet}).
 \end{equation*}

Now, let $I_x$ be the injective hull of the residue field $\mathbb{C}(x)$ of $\mathcal{O}_{\sigma_{k},x}$ and denote by $(-)^{\widehat{}}$ the operation of completing an $\OO_{\sigma_k,x}$-module at the maximal ideal $\mathfrak{m}_x \subseteq \OO_{\sigma_{k,x}}$. By \cite[Chapter V, Corollary 6.3]{Hartshorne.66}, for any finitely generated $\OO_{\sigma_k,x}$-module $M$, we have
\begin{align}\label{eq:omega_sigma}
    \operatorname{Ext}^i(M,\omega_{\sigma_{k},x}^{\bullet})^{\widehat{}} & \cong \operatorname{Hom}(H_x^{-i}(M),I_x).
\end{align}
Note that $H^{d_k-1}_x(\OO_{\sigma_{k-1,x}}) = H^{d_k}_x(\OO_{\sigma_{k-1,x}}) = 0$ because $d_k-1 > d_{k-1}$.
Since $\mathcal{I}_{\sigma_{k-1}/\sigma_k} \cong \alpha_{k,*}\OO_{B^k}(-Z_k)$ by Theorem \ref{Du Bois-type condition} (1), localizing the ideal sequence of $\sigma_{k-1}$ in $\sigma_k$ at $x$ and taking local cohomology gives
\[
H^{d_k}_x(\alpha_{k,*}\mathcal{O}_{B^k}(-Z_{k})_x) \cong H^{d_k}_x(\mathcal{O}_{\sigma_k,x}).
\]
Applying $\operatorname{Hom}(-,I_x)$ to both sides of this isomorphism, (\ref{eq:omega_sigma}) above then gives
\[
\operatorname{Ext}^{-d_k}(\alpha_{k,*}\OO_{B^k}(-Z_k)_x,\omega_{\sigma_k,x}^{\bullet})_x^{\widehat{}} \cong \operatorname{Ext}^{-d_k}(\OO_{\sigma_k,x},\omega_{\sigma_k,x}^{\bullet})_x^{\widehat{}}.
\]
But $\operatorname{Ext}^{-d_k}(\OO_{\sigma_k,x},\omega_{\sigma_k,x}^{\bullet})_x^{\widehat{}}$ is simply $(\omega_{\sigma_k,x})_x^{\widehat{}}$, so we can conclude that $\operatorname{Ext}^{-d_k}(\alpha_{k,*}\mathcal{O}_{B^k}(-Z_{k})_x, \omega_{\sigma_k,x}^{\bullet})^{\widehat{}}$ and $(\omega_{\sigma_k,x})_x^{\widehat{}}$ are isomorphic. Since the completion functor is exact on finitely generated modules, this means the $\OO_{\sigma_{k,x}}$-modules themselves are isomorphic
$$
\operatorname{Ext}^{-d_k}(\alpha_{k,*}\OO_{B^k}(-Z_k)_x,\omega_{\sigma_k,x}^{\bullet}) \cong \omega_{\sigma_k,x}.
$$
Hence $\alpha_{k,*}\omega_{B^k}(Z_{k})_x \rightarrow \omega_{\sigma_k,x}$ is an isomorphism for each closed point $x$ of $\sigma_k$.
\end{proof}

Next, we focus on $R^i \alpha_{k,*} \omega_{B^k}(Z_{k})$ for $i\geq 1$. We start with a few simple observations.

\begin{lemma}\label{higher-dualizing-Z}
$R^i \alpha_{k,*} \omega_{B^k}(Z_{k})\cong R^i \alpha_{k,*} \omega_{Z_{k}}$ for $i\geq 1$.
\end{lemma}
\begin{proof}
Immediate by applying the Grauert--Riemenschneider vanishing theorem to the short exact sequence
    \[
    0\lra \omega_{B^k}\lra \omega_{B^k}(Z_{k})\lra \omega_{Z_{k}}\lra 0. \qedhere
    \]
\end{proof}

\begin{lemma}\label{higher-tau-CM}
    $R^i \tau_{*}\left((\operatorname{res} ^* \omega_X)((n-1)F_{k-1})\right)\cong H^i(X,\omega_X)\otimes \mathcal{O}_{X^{[k-1]}}$ for $i \geq 0$.
\end{lemma}
\begin{proof}
By Lemma \ref{standard-grauert-riemenschneider}, we have
\[
R^i \operatorname{bl}_{*} \mathcal{O}_{X^{[k-1,k]}}((n-1)F_{k-1})\cong
\begin{cases}
    \mathcal{O}_{X\times X^{[k-1]}} & \text{ for $i=0$}\\
    0 & \text{ for $i\geq 1$.}
\end{cases}
\] 
Recalling that $\operatorname{bl} = \operatorname{res}\times\tau$, the statement then follows by applying the Leray spectral sequence to the composition $\tau = \operatorname{bl}\circ \operatorname{pr}_2$, where $\operatorname{pr}_2$ denotes the projection $X\times X^{[k-1]} \lra X^{[k-1]}$.
\end{proof}

\begin{lemma}\label{appearance-Hi-CM}
$R^i \tilde{\tau}_{*} (\alpha_{k-1,k}^* \omega_{Z_{k}})\cong \omega_{B^{k-1}} (Z_{k-1})\otimes H^i(X,\omega_X)$ for $i \geq 0$.
\end{lemma}
\begin{proof}
Recall that
\[
\omega_{X^{[k-1,k]}}\cong\operatorname{res}^* \omega_X \otimes \tau^* \omega_{X^{[k-1]}}((n-1)F_{k-1}).
\]
By taking base-change, we obtain
\[
\omega_{B^{k-1,k}}\cong \tilde{\tau}^* \omega_{B^{k-1}}
\otimes \pi_{k-1,k}^* (\operatorname{res}^* \omega_X ((n-1)F_{k-1})).
\]
Recall also from the proof of Proposition \ref{A&Z} that
\[
\omega_{B^{k-1,k}}\cong (\alpha_{k-1,k}^* \omega_{Z_k})(-\tilde{\tau}^* Z_{k-1}).
\]
Combining the two, we get
\[
\alpha_{k-1,k}^* \omega_{Z_k}\cong \tilde{\tau}^* \omega_{B^{k-1}}(Z_{k-1})\otimes \pi_{k-1,k}^* (\operatorname{res}^* \omega_X ((n-1)F_{k-1})).
\]
Now using Lemma \ref{higher-tau-CM} and \cite[Chapter III, Proposition 9.3]{Hartshorne.77}, we compute
\begin{align*}
    R^i \tilde{\tau}_{*} (\alpha_{k-1,k}^* \omega_{Z_{k}})&\cong R^i \tilde{\tau}_{*} 
    \left(\tilde{\tau}^* \omega_{B^{k-1}}(Z_{k-1})\otimes \pi_{k-1,k}^* (\operatorname{res}^* \omega_X ((n-1)F_{k-1}))\right)\\
    &\cong \omega_{B^{k-1}} (Z_{k-1})\otimes R^i \tilde{\tau}_{*} \left(\pi_{k-1,k}^* (\operatorname{res}^* \omega_X ((n-1)F_{k-1}))\right)\\
    &\cong \omega_{B^{k-1}} (Z_{k-1})\otimes \pi_{k-1}^* R^i \tau_{*} (\operatorname{res}^* \omega_X ((n-1)F_{k-1}))\\
    &\cong \omega_{B^{k-1}} (Z_{k-1})\otimes H^i(X,\omega_X). \qedhere
\end{align*}
\end{proof}

 We spell out here the following easy observation.

\begin{lemma}\label{remark-support}
For any sheaves $\mathcal{F}$ and $\mathcal{Q}$ on $B^k$ and $B^{k-1,k}$ respectively, we have that $R^i \alpha_{k,*} \mathcal{F}$ and $R^i \alpha_{k-1,k,*} \mathcal{Q}$ are supported on $\sigma_{k-1}$ and $\sigma_{k-2,k}$ respectively for any $i\geq 1$.
\end{lemma}

\begin{proof}
Immediate from Theorem \ref{terracini} and Theorem \ref{relative-terracini}.
\end{proof}

\begin{lemma}\label{direct-image-alpha-k-1-k}
$R^i \alpha_{k-1,k,*} \mathcal{O}_{B^{k-1,k}} = 0$ for $i\geq 1$. 
\end{lemma}

\begin{proof}
By Lemma \ref{lem:standfactforhdi}, we only need to show that $H^i(B^{k-1,k}, \alpha_{k-1,k}^* (\mathcal{O}_{B^k}(m_1) \otimes \pi_k^* A_{k,H}^{m_2})|_{Z_k}) = 0$ for $i \geq 1$ and $m_2 \gg m_1 \gg 0$, where $H$ is a sufficiently positive line bundle on $X$. Note that
$$
H^i(B^{k-1,k}, \alpha_{k-1,k}^* (\mathcal{O}_{B^k}(m_1) \otimes \pi_k^* A_{k,H}^{m_2})|_{Z_k}) \cong H^i(X^{[k-1,k]}, \tau^* S^{m_1} E_{k-1,L} \otimes \rho^* A_{k,H}^{m_2})
$$
for $i \geq 1$ and $m_1$, $m_2 \geq 1$.
By Corollary \ref{direct-image-E-k-L} (1) and Serre vanishing, we have
$$
H^i(X^{[k-1,k]}, \tau^* S^{m_1} E_{k-1,L} \otimes \rho^* A_{k,H}^{m_2})
\cong H^i(X^{[k]}, \rho_* \tau^* S^{m_1} E_{k-1,L} \otimes A_{k,H}^{m_2})=0
$$
for $i \geq 1$ and $m_2 \gg m_1 \gg 0$.
\end{proof}

\begin{lemma}\label{support-alpha-alpha-alpha}
For $i\geq 0$, the sheaf $R^i \alpha_{k,*} (\alpha_{k-1,k,*} \alpha_{k-1,k}^* \omega_{Z_{k}})$ is supported on $\sigma_{k-2}$ if and only if $H^i(X,\omega_X)=0$.
\end{lemma}
\begin{proof}
By Lemma \ref{direct-image-alpha-k-1-k} and the equality $\alpha_k \circ \alpha_{k-1,k}=\alpha_{k-1}\circ \tilde{\tau}$, we have
\[
R^i \alpha_{k,*} (\alpha_{k-1,k,*} \alpha_{k-1,k}^* \omega_{Z_{k}}) \cong R^i (\alpha_k \circ \alpha_{k-1,k})_* (\alpha_{k-1,k}^* \omega_{Z_{k}})
\cong R^i (\alpha_{k-1}\circ \tilde{\tau})_* (\alpha_{k-1,k}^* \omega_{Z_{k}}).
\]
There is a spectral sequence
\begin{equation}\label{eq:spectralseqCMproof}
E_2 ^{p,q}=R^p \alpha_{k-1,*} R^q \tilde{\tau}_{*} \alpha_{k-1,k}^* \omega_{Z_{k-1}} \implies E^{p+q} =R^{p+q}(\alpha_{k-1}\circ \tilde{\tau})_* \alpha_{k-1,k}^* \omega_{Z_{k-1}} .
\end{equation}
Combining this and Lemma \ref{remark-support}, we see that $R^i (\alpha_{k-1}\circ \tilde{\tau})_* \alpha_{k-1,k}^* \omega_{Z_{k-1}}$ is supported on $\sigma_{k-2}$ if and only if $\alpha_{k-1,*} R^i \tilde{\tau}_{*} (\alpha_{k-1,k}^* \omega_{Z_{k}})$ is supported on $\sigma_{k-2}$. We conclude by Lemma \ref{appearance-Hi-CM}.
\end{proof}

\begin{lemma}\label{CM-support}
For $i\geq 1$, the sheaf $R^i \alpha_{k,*} \omega_{B^k}(Z_{k})$ is supported on $\sigma_{k-2}$ if and only if $H^i(X,\omega_X)=0$.
\end{lemma}
\begin{proof}
    By Lemma \ref{higher-dualizing-Z}, it suffices to show the statement for $R^i \alpha_{k,*} \omega_{Z_{k}}$. Consider the short exact sequence
\[
0\lra \omega_{Z_{k}}\lra \alpha_{k-1,k,*} \alpha_{k-1,k}^* \omega_{Z_{k}}\lra \mathcal{F}\lra 0.
\]
The sheaf $\mathcal{F}$ is supported on $\sigma_{k-2,k}$ by Proposition \ref{relative-terracini}. We have an induced long exact sequence
\begin{equation}\label{eq:longexactalpha_*omega_Z}
\cdots \lra R^{i-1}\alpha_{k,*}\mathcal{F}\lra
R^i \alpha_{k,*} \omega_{Z_{k}}\lra R^i \alpha_{k,*} (\alpha_{k-1,k,*}\alpha_{k-1,k}^* \omega_{Z_{k}})\lra
R^i \alpha_{k,*} \mathcal{F}\lra \cdots
\end{equation}
We may conclude by Lemma \ref{support-alpha-alpha-alpha}.
\end{proof}

\begin{proposition}\label{not-CM}
If $k \geq 2$ and $H^i(X, \omega_X)\neq 0$ for some $1\leq i\leq n-1$, then $\sigma_k $ is not Cohen--Macaulay.
\end{proposition}
\begin{proof}
By Lemma \ref{CM-support},  we have $R^i \alpha_{k,*} \omega_{B^k}(Z_{k})\neq 0$. Considering the Leray spectral sequence for $\alpha_k$, we obtain 
\[
H^{d_k-i}(\sigma_k, \mathcal{I}_{\sigma_{k-1}/\sigma_k}(-\ell))\neq 0~~\text{ for $\ell \gg 0$}
\]
by Corollary \ref{serre-on-secant}.
Since $d_k-i>d_{k-1}+1$, we have
\[
H^{d_k-i}(\sigma_k, \mathcal{O}_{\sigma_k}(-\ell))\cong H^{d_k-i}(\sigma_k, \mathcal{I}_{\sigma_{k-1}/\sigma_k}(-\ell))\neq 0.
\]
Therefore, $\sigma_k$ is not Cohen--Macaulay by \cite[Corollary 5.72]{Kollar-Mori}.
\end{proof}

Finally, we show that if $H^i(X, \mathcal{O}_X)=0$ for $1 \leq i \leq n-1$, then $\sigma_k$ is instead Cohen--Macaulay. The key ingredient is Koll\'ar's injectivity theorem \cite[Theorem 9.12]{Kollar.Shafarevich}. We may apply this theorem to $\sigma_k$ since $\sigma_k$ has Du Bois singularities by Theorem \ref{thm:dubois}. We start with the following.

\begin{lemma}\label{injectivity-implies-vanishing}
Pick $k\geq 2$ and $H\in |\ell H_k|$. For an integer $i\geq 1$, suppose that the map
        \[
        H^i(B^k, \omega_{B^k}(sH_k+Z_k))\overset{\cdot H}{\lra} H^i(B^k, \omega_{B^k}((s+\ell) H_k+Z_k))
        \]
        \noindent is injective for all $s$ and $\ell$ sufficiently divisible. Then the only subsheaf of $R^i \alpha_{k,*} \omega_{B^k}(Z_k)$ supported in $\alpha_{k,*} (H)$ is the zero sheaf.
\end{lemma}
\begin{proof}
By Serre vanishing, we have
    \[
    H^j\left(\sigma_k, (R^{i-j}\alpha_{k,*} \omega_{B^k}(Z_k))(s)\right)=0~~\text{ for all $1\leq j\leq i$ and $s$ sufficiently large}.
    \]
The map
\[
H^0(\sigma_k, R^i\alpha_{k,*} \omega_{B^k}(sH_k+Z_k))\overset{\cdot H}{\lra}
H^0(\sigma_k, R^i\alpha_{k,*} \omega_{B^k}((s+\ell)H_k+Z_k))
\]
is injective by the Leray spectral sequence and by hypothesis. The conclusion easily follows.
\end{proof}

\begin{lemma}\label{injectivity-preliminary}
Pick $k\geq 2$ and $H\in |\ell H_k|$ such that $Z_k$ is not contained in the support of $H$. Assume that $\sigma_{k-1}$ is Cohen--Macaulay. If $s$ and $\ell$ are sufficiently divisible, then the map
        \[
        H^i(B^k, \omega_{B^k}(sH_k+Z_k))\overset{\cdot H}{\lra} H^i(B^k, \omega_{B^k}((s+\ell) H_k+Z_k))
        \]
is injective for all $i\not\in\{0, n+1\}$.
\end{lemma}
\begin{proof}
By abuse of notation, we also use $H$ for $\alpha_k(H)$. By Corollary \ref{serre-on-secant}, the assertion is equivalent to the surjectivity of
\[
H^{d_k - i}(\sigma_k, \mathcal{I}_{\sigma_{k-1}/\sigma_k}(-s-\ell))\overset{\cdot H}{\lra} H^{d_k - i}(\sigma_k, \mathcal{I}_{\sigma_{k-1}/\sigma_k}(-s)).
\]
Consider the commutative diagram
\begin{center}
\begin{tikzcd}[column sep=small]
    H^{d_k-i-1}(\mathcal{O}_{\sigma_{k-1}}(-s-\ell))\ar[r]\ar[d, "\cdot H"]& H^{d_k-i}(\mathcal{I}_{\sigma_{k-1}/\sigma_k}(-s-\ell))\ar[r]\ar[d, "\cdot H"]& H^{d_k-i}(\mathcal{O}_{\sigma_k}(-s-\ell))\ar[r]\ar[d, "\cdot H"]& 
    H^{d_k-i}(\mathcal{O}_{\sigma_{k-1}}(-s-\ell))\ar[d, "\cdot H"]\\
    H^{d_k-i-1}(\mathcal{O}_{\sigma_{k-1}}(-s))\ar[r]& H^{d_k-i}(\mathcal{I}_{\sigma_{k-1}/\sigma_k}(-s))\ar[r]& H^{d_k-i}(\mathcal{O}_{\sigma_k}(-s))\ar[r]& 
    H^{d_k-i}(\mathcal{O}_{\sigma_{k-1}}(-s))
\end{tikzcd}.
\end{center}
Since $\sigma_k$ has Du Bois singularities (Theorem \ref{thm:dubois}), it follows by \cite[Theorem 9.12]{Kollar.Shafarevich} that the map
\[
H^{d_k-i}(\sigma_k, \mathcal{O}_{\sigma_k}(-s-\ell))\overset{\cdot H}{\lra} H^{d_k-i}(\sigma_k, \mathcal{O}_{\sigma_k}(-s))
\]
is surjective for each $i\geq 1$. Since $\sigma_{k-1}$ is Cohen--Macaulay, it follows from Serre's vanishing theorem that $H^j(\sigma_{k-1}, \mathcal{O}_{\sigma_{k-1}}(-\ell'))=0$ for $j\neq d_{k-1}=d_k - n - 1$ and $\ell'\gg 0$. Moreover, the map 
\[
H^{d_k-n-1}(\sigma_{k-1}, \mathcal{O}_{\sigma_{k-1}}(-s-\ell))\overset{\cdot H}{\lra} H^{d_k-n-1}(\sigma_{k-1}, \mathcal{O}_{\sigma_{k-1}}(-s)),
\]
is surjective. Then the statement easily follows.
\end{proof}

\begin{lemma}\label{CM-preliminary}
Assume that $H^i(X,\mathcal{O}_X)=0$ for all $1\leq i\leq n-1$. Then
        \[
        R^i \alpha_{k,*} \omega_{B^k}(Z_{k})\cong\begin{cases}
            \omega_{\sigma_k} & \text{if } i=0\\
            0 & \text{if $i\not\in\{0,n,n+1\}$}.
        \end{cases}
        \]
\end{lemma}
\begin{proof}
Lemma \ref{dualizing-sheaf-secant-variety} shows the case $i=0$.
    By Lemma \ref{CM-support} we have that
    $R^i \alpha_{k,*} \omega_{B^k}(Z_{k})$ is supported in $\sigma_{k-2}$ if $i\not\in\{0,n\}$.
Since we may pick $H\in|\ell H_k|$ that vanishes on the support of
$R^i \alpha_{k,*} \omega_{B^k}(Z_{k})$ but not on $Z_k$, we are done by Lemma \ref{injectivity-implies-vanishing} and Lemma \ref{injectivity-preliminary}. 
\end{proof}

\begin{lemma}\label{preliminary-CM-2}
Assume that $\sigma_{k-1}$ is Cohen--Macaulay with $k \geq 2$ and $H^i(X,\mathcal{O}_X)=0$ for $1\leq i\leq n-1$. If $i\not\in\{0,d_k-n,d_k-n-1\}$, then 
$H^i(\sigma_k, \mathcal{O}_{\sigma_k}(-\ell))=0$ for any $\ell$ sufficiently large.
\end{lemma}
\begin{proof}
Consider the Leray spectral sequence for $\alpha_k$ and use Corollary \ref{serre-on-secant} and Lemma \ref{CM-preliminary}.
\end{proof}

\begin{theorem}\label{thm:CM}
Assume $H^i(X,\mathcal{O}_X)=0$ for all $1\leq i\leq n-1$. Then the following hold:
    \begin{enumerate}[topsep=0pt]
        \item If $\ell, s$ are sufficiently divisible integers and $D \in |\ell H_k|$ is a member such that $Z_k$ is not contained in the support of $D$, then the map
        \[
        H^i(B^k, \omega_{B^k}(sH_k+Z_k))\overset{\cdot D}{\lra} H^i(B^k, \omega_{B^k}((s+\ell) H_k+Z_k))
        \]
        \noindent is injective for all $i\geq 1$.
        \item 
        $     
        R^i \alpha_{k,*} \omega_{B^k}(Z_{k})\cong\begin{cases}
            \omega_{\sigma_k} & \text{if } i=0\\
            \omega_{\sigma_{k-1}} & \text{if } i=n\\
            0 & \text{otherwise}.
        \end{cases}
        $
        \item $H^i(\sigma_k, \mathcal{O}_{\sigma_k}(-\ell))=0$ for $1 \leq i \leq d_k-1$ and $\ell \gg 0$. In particular, $\sigma_k$ is Cohen--Macaulay.
    \end{enumerate}
\end{theorem}
\begin{proof}
We proceed by induction on $k$. The case $k=1$ is obvious using the convention $\omega_{\sigma_0}=0$. Assume that $k\geq 2$. 
We only need to prove (1) for $i=n+1$ by Lemma \ref{injectivity-preliminary}, (2) for  $i\in\{n,n+1\}$ by Lemma \ref{CM-preliminary}, and (3) for $i\in\{d_k-n-1,d_k-n\}$ by Lemma \ref{preliminary-CM-2}. Note that the second assertion of (3) follows from the first one by  \cite[Corollary 5.72]{Kollar-Mori}.

There is a map $R^n \alpha_{k,*} \omega_{Z_k}\lra R^n \alpha_{k,*} (\alpha_{k-1,k,*} \alpha_{k-1,k}^* \omega_{Z_k})$ coming from the exact sequence (\ref{eq:longexactalpha_*omega_Z}). Applying the inductive hypothesis and Lemma \ref{appearance-Hi-CM} to the spectral sequence (\ref{eq:spectralseqCMproof}), we get a surjective map $R^n \alpha_{k,*} (\alpha_{k-1,k,*} \alpha_{k-1,k}^* \omega_{Z_k}) \lra \omega_{\sigma_{k-1}}$. Composing these two maps, we have a map 
$$
\phi \colon R^n \alpha_{k,*} \omega_{Z_k} \lra \omega_{\sigma_{k-1}}.
$$
We want to show that $\phi$ is an isomorphism.
Notice that both $\operatorname{ker}(\phi)$ and $\operatorname{coker}(\phi)$ are supported on $\sigma_{k-2}$. Therefore we may pick $H\in|\ell H_k|$ vanishing along the supports of $\operatorname{ker}(\phi)$ and $\operatorname{coker}(\phi)$ but not on $\sigma_{k-1}$. 
By Lemma \ref{injectivity-implies-vanishing} and the claim of (1) for $i=n$, we have $\ker(\phi)=0$. 
Therefore the map
\[
H^0(\sigma_k, (R^n \alpha_{k,*} \omega_{B^k}(Z_k)) (\ell'))\lra H^0(\sigma_{k-1}, \omega_{\sigma_{k-1}}(\ell'))
\]
is injective for every $\ell' \geq 0$. 
By Theorem \ref{Du Bois-type condition} (1) and the inductive hypothesis for (3), this map is dual to the map
\begin{equation} \label{dual of map on H^0} 
H^{d_k-n-1}(\mathcal{O}_{\sigma_{k-1}}(-\ell')) \lra H^{d_k-n}(\mathcal{I}_{\sigma_{k-1}/\sigma_k}(-\ell')), 
\end{equation}
which is then surjective for every $\ell' \geq 0$. Thus we get an exact sequence
$$
0 \to H^{d_k-n-1}(\mathcal{I}_{\sigma_{k-1}/\sigma_k}(-\ell')) \to H^{d_k-n-1}(\mathcal{O}_{\sigma_k}(-\ell')) \lra H^{d_k-n-1}(\mathcal{O}_{\sigma_{k-1}}(-\ell')) \to H^{d_k-n}(\mathcal{I}_{\sigma_{k-1}/\sigma_k}(-\ell')) \to 0,
$$
and $H^{d_k-n}(\sigma_k, \mathcal{O}_{\sigma_k}(-\ell'))=0$ for $\ell' \gg 0$. In particular, (3) holds for $i=d_k-n$. Since $H^0(\operatorname{coker}(\phi)(s))^{\vee}$ is the kernel of the morphism (\ref{dual of map on H^0}) for $\ell' \gg 0$, we have a commutative diagram
\begin{center}
\begin{tikzcd}
H^{d_k-n-1}(\OO_{\sigma_k}(-s-\ell)) \ar[r,twoheadrightarrow] \ar[d,"\cdot H"] & H^0(\operatorname{coker}(\phi)(s+\ell))^{\vee}\ar[d,"\cdot H"]\\
H^{d_k-n-1}(\OO_{\sigma_k}(-s)) \ar[r,twoheadrightarrow] & H^0(\operatorname{coker}(\phi)(s))^{\vee}
\end{tikzcd},
\end{center}
where the two horizonal maps are surjective. Since the left vertical map is surjective by \cite[Theorem 9.12]{Kollar.Shafarevich}, so is the right vertical map. However, by the choice of $H$, the right vertical map is the zero map. Therefore, $\operatorname{coker}(\phi)=0$. This implies (1) for $i=n+1$ and (2) for $i=n$. Then Lemma \ref{support-alpha-alpha-alpha} and Lemma \ref{injectivity-implies-vanishing}, we get (2) for $i=n+1$. Now, by Corollary \ref{serre-on-secant} and (2) coupled with the Leray spectral sequence, we finally obtain
$$
H^{d_k-n-1}(\sigma_k, \mathcal{O}_{\sigma_k}(-\ell')) \cong H^{d_k-n-1}(\mathcal{I}_{\sigma_{k-1}/\sigma_k}(-\ell')) \cong H^{n+1}(\omega_{B^k}(\ell' H_k + Z_k))=0.
$$
Therefore, (3) holds for $i=d_k-n-1$. 
\end{proof}

\subsection{Arithmetically Cohen--Macaulay}
Our goal is to show that $\sigma_k \subseteq \mathbb{P}^r$ is arithmetically Cohen--Macaulay when $\sigma_k$ is Cohen--Macaulay (equivalently, $H^i(X, \mathcal{O}_X)=0$ for $1 \leq i \leq n-1$ by Theorem \ref{thm:CM<=>H^i(O_X)=0}). The starting point is some cohomology computations on $X^{[k]}$.

\begin{lemma}\label{cohomology-structure-sheaf}
    Assume $H^i(X,\mathcal{O}_X)=0$ for $1\leq i\leq n-1$. If $n$ is even, then
    \[
    H^i(X^{[k]}, \mathcal{O}_{X^{[k]}})\cong \begin{cases}
        S^i H^n(X,\mathcal{O}_X) & \text {if $n \mid i$ and $i \leq kn$}\\
        0 & \text{otherwise}
    \end{cases}
    \]
    \[
    H^i(X^{[k]}, \delta_k ^{-1})\cong \begin{cases}
        \wedge^k H^n(X,\mathcal{O}_X) & \text{if $i=kn$}\\
        \wedge^{k-1} H^n(X,\mathcal{O}_X) & \text{if $i=kn-n$}\\
        0 & \text{otherwise.}
    \end{cases}
    \]
    
\noindent If $n$ is odd, then
    \[
    H^i(X^{[k]}, \mathcal{O}_{X^{[k]}})\cong \begin{cases}
        \wedge^i H^n(X,\mathcal{O}_X) & \text{if $n \mid i$ and $i \leq kn$}\\
        0 & \text{otherwise}
    \end{cases}
    \]
    \[
    H^i(X^{[k]},\delta_k ^{-1})\cong \begin{cases}
        S^k H^n(X,\mathcal{O}_X) & \text{if $i=kn$}\\
        S^{k-1} H^n(X,\mathcal{O}_X) & \text{if $i=kn-n$}\\
        0 & \text{otherwise.}
    \end{cases}
    \]
\end{lemma}
\begin{proof}
The argument closely follows \cite[Lemma 3.7]{Ein.Niu.Park.20}.
Recall from \S \ref{subsubsec-line-bundles} that the sheaf $\mathcal{O}_X ^{\boxtimes k}$ on $X^k$ equipped with two different permutation actions of the symmetric group $S_k$ via symmetric and alternating ways. We have 
$h_* \mathcal{O}_{X^{[k]}} \cong q_* ^{S_k} \mathcal{O}_X ^{\boxtimes k, \operatorname{sym}}$ and $h_* \delta_k ^{-1} \cong q_* ^{S_k}\mathcal{O}_X ^{\boxtimes k, \operatorname{alt}}$, where $q_* ^{S_k}$ denotes the equivariant pushforward.
By Lemma \ref{canonical-bundle} and Lemma \ref{standard-grauert-riemenschneider}, we get $R^i h_*\mathcal{O}_{X^{[k]}}=R^i h_*\delta_k ^{-1}=0$ for $i>0$. Therefore, we have
\[
H^i(X^{[k]}, \mathcal{O}_{X^{[k]}})\cong H^i(\operatorname{Sym}^k (X), q_* ^{S_k} \mathcal{O}_X ^{\boxtimes k, \operatorname{sym}})\cong H^i(X^k, \mathcal{O}_X ^{\boxtimes k, \operatorname{sym}})^{S_k}
\]
\[
H^i(X^{[k]}, \delta_k ^{-1})\cong H^i(\operatorname{Sym}^k (X), q_{*} ^{S_k} \mathcal{O}_X ^{\boxtimes k, \operatorname{alt}})\cong H^i(X^k, \mathcal{O}_X ^{\boxtimes k, \operatorname{alt}})^{S_k}.
\]
By K\"unneth formula, $H^i(X^k,\mathcal{O}_X ^{\boxtimes k})=0$ if $n$ does not divide $i$. Suppose therefore that $n$ divides $i$ and set $j=i/n$. Again by K\"unneth formula, we have $H^i(X^k,\mathcal{O}_X ^{\boxtimes k})$ is a direct sum of spaces isomorphic to $H^0(X,\mathcal{O}_X) ^{\otimes (k-j)}\otimes H^n (X,\mathcal{O}_X)^j\cong H^n (X,\mathcal{O}_X)^j$. It is easy to see that
\[
H^i(X^k,\mathcal{O}_X ^{\boxtimes k})^{S_k} \cong \left( H^n (X,\mathcal{O}_X) ^{\otimes j}\right)^{S_j}.
\]
Since the inclusion $H^n (X,\mathcal{O}_X) ^{\otimes j}\lra H^i(X^k,\mathcal{O}_X ^{\boxtimes k})$ is given by the cup product, we have that the action induced by $\mathcal{O}_X ^{\boxtimes k, \operatorname{sym}}$ on $H^n (X,\mathcal{O}_X) ^{\otimes j}$ is symmetric if $n$ is even, and alternating if $n$ is odd. The opposite holds for the action induced by $\mathcal{O}_X ^{\boxtimes k, \operatorname{alt}}$. The lemma now follows immediately. (For further discussion of the details concerning the equivariance of the cup product inclusion above, see \cite[\S 4.2]{Sheridan.2020}).
\end{proof}

When $n=2$, this lemma also follows from \cite[Corollary 4.1]{Krug.18} and \cite[Theorem 5.2.1]{Scala.09}.

\begin{corollary}\label{cohomology-ideal-secant}
Assume $H^i(X,\mathcal{O}_X)=0$ for $1\leq i\leq n-1$. Then the following hold:
\begin{enumerate}[topsep=0pt]
\item $H^0(\sigma_k, \omega_{\sigma_k}) \cong \begin{cases} \wedge^k H^0(X, \omega_X) & \text{if $n$ is even} \\ S^k H^0(X, \omega_X) & \text{if $n$ is odd.}
\end{cases}$
\item $H^i(\sigma_k, \mathcal{I}_{\sigma_{k-1}/\sigma_k})\cong\begin{cases}
    H^0(\sigma_{k},\omega_{\sigma_{k}})^{\vee} & \text{if $i=d_k$}\\
        H^0(\sigma_{k-1},\omega_{\sigma_{k-1}})^{\vee} & \text{if $i=d_k-n$}\\
        0 & \text{otherwise}.
    \end{cases}$
\end{enumerate}
\end{corollary}
\begin{proof}
Part (1) is immediate by Lemma \ref{dualizing-sheaf-secant-variety}, Lemma \ref{cohomology-structure-sheaf}, and the isomorphism $\omega_{B^k}(Z_k)\cong \pi_k ^*(\omega_{X^{[k]}}\otimes \delta_k)$.
For (2), note that 
$$
H^i(\mathcal{I}_{\sigma_{k-1}/\sigma_k})\cong H^i(\mathcal{O}_{B^k}(-Z_k))\cong H^{d_k-i}(\omega_{B^k}(Z_k))^{\vee} \cong H^{d_k-i}(\omega_{X^{[k]}}\otimes \delta_k)^{\vee}
\cong H^{i+kn-d_k}(\delta_k ^{-1}).
$$
Then (2) follows from Lemma \ref{cohomology-structure-sheaf} and (1).
\end{proof}

\begin{theorem}\label{thm:arithmetically-CM}
Assume $H^i(X,\mathcal{O}_X)=0$ for $1\leq i\leq n-1$. Then $\sigma_k \subseteq \mathbb{P}^r$ is arithmetically Cohen--Macaulay.  
\end{theorem}
\begin{proof}
It is well-known that $\sigma_k \subseteq \mathbb{P}^r$ is arithmetically Cohen--Macaulay if and only if $\sigma_k \subseteq \mathbb{P}^r$ is projectively normal and $H^i(\sigma_k, \OO_{\sigma_k}(\ell))=0$ for $1 \leq i \leq d_k-1$ and $\ell \in \mathbb{Z}$. By Theorem \ref{thm:normalsing}, we only need to check that $H^i(\sigma_k, \OO_{\sigma_k}(\ell))=0$ for $1 \leq i \leq d_k-1$ and $\ell \leq 0$. 
Recall that $\sigma_k$ has Cohen--Macaulay and Du Bois singularities by Theorem \ref{thm:CM} and Theorem \ref{thm:dubois}. By \cite[Theorem 10.42]{Kollar}, we have $H^i(\sigma_k, \mathcal{O}_{\sigma_k}(\ell))=0$ for $1\leq i\leq d_k-1$ and $\ell<0$. Therefore, it suffices to show $H^i(\sigma_k, \mathcal{O}_{\sigma_k})=0$ for $1\leq i\leq d_k-1$. We do this by induction on $k$, with the case $k=1$ being obvious. Consider the long exact sequence
\[
\cdots \lra H^i(\sigma_k, \mathcal{I}_{\sigma_{k-1}/\sigma_k})\lra H^i(\sigma_k, \mathcal{O}_{\sigma_k})\lra H^i(\sigma_{k-1},\mathcal{O}_{\sigma_{k-1}})\lra H^{i+1}(\sigma_k, \mathcal{I}_{\sigma_{k-1}/\sigma_k})\lra \cdots.
\]
By inductive hypothesis and Corollary \ref{cohomology-ideal-secant}, we get $H^i(\sigma_k,\mathcal{O}_{\sigma_k})=0$ for all $1\leq i\leq d_k-1$ except possibly when $i\in\{d_k-n-1,d_k-n\}$. For the exceptional cases, consider the exact sequence
\[
0\lra H^{d_k-n-1}(\mathcal{O}_{\sigma_k})\lra 
H^{d_k-n-1}(\mathcal{O}_{\sigma_{k-1}})\overset{f}{\lra} H^{d_k-n}(\mathcal{I}_{\sigma_{k-1}/\sigma_k})\lra
H^{d_k-n}(\mathcal{O}_{\sigma_k})\lra 0.
\]
A careful check shows that $f$ coincides with the composed isomorphism 
\[
H^{d_k-n-1}(\sigma_{k-1},\mathcal{O}_{\sigma_{k-1}}) \cong H^{0}(\sigma_{k-1}, \omega_{\sigma_{k-1}})^{\vee} \cong H^{d_k-n}(\sigma_k, \mathcal{I}_{\sigma_{k-1}/\sigma_k})
\]
obtained by Serre duality and Corollary \ref{cohomology-ideal-secant}. Thus $H^i(\sigma_k,\mathcal{O}_{\sigma_k})=0$ for the remaining cases.
\end{proof}

\begin{remark}
If $\sigma_k \subseteq \mathbb{P}^r$ is arithmetically Cohen--Macaulay, then $\sigma_k$ is Cohen--Macaulay so that $H^i(X, \mathcal{O}_X)=0$ for $1 \leq i \leq n-1$. Thus the converse of Theorem \ref{thm:arithmetically-CM} also holds. As $L$ is sufficiently positive, $X \subseteq \mathbb{P}^r$ is arithmetically Cohen--Macaulay if and only if $H^i(X, \mathcal{O}_X)=0$ for $1 \leq i \leq n-1$. Hence we may conclude that $X \subseteq \mathbb{P}^r$ is arithmetically Cohen--Macaulay if and only if so is $\sigma_k \subseteq \mathbb{P}^r$.
\end{remark}

\subsection{Rational singularities} 
We investigate when $\sigma_k$ has rational, log terminal, or log canonical singularities.

\begin{theorem}\label{thm:rational-sing}
$\sigma_k$ has rational singularities if and only if $H^i(X, \mathcal{O}_X)=0$ for $1 \leq i \leq n$.
\end{theorem}

\begin{proof}
Recall that $\sigma_k$ has rational singularities if and only if $\sigma_k$ is Cohen--Macaulay and $\alpha_{k,*} \omega_{B^k} = \omega_{\sigma_k}$. 
By Theorem \ref{thm:CM<=>H^i(O_X)=0}, we only need to show that $\alpha_{k,*} \omega_{B^k} = \omega_{\sigma_k}$ if and only if $H^0(X,\omega_X)=0$  under the assumption that $H^i(X, \mathcal{O}_X)=0$ for $1 \leq i \leq n-1$.
Consider the short exact sequence
\[
0\lra \omega_{B^k}\lra \omega_{B^k}(Z_k)\lra \omega_{Z_k}\lra 0.
\]
By taking push-forward under $\alpha_k$ and using Lemma \ref{dualizing-sheaf-secant-variety} and the Grauert--Riemenschneider vanishing theorem, we obtain a short exact sequence
\[
0\lra \alpha_{k,*} \omega_{B^k}\lra \omega_{\sigma_k}\lra \alpha_{k,*}\omega_{Z_k}\lra 0.
\]
Therefore, it suffices to show that $\alpha_{k,*} \omega_{Z_k}=0$ if and only if $H^0(X,\omega_X)=0$. Consider the short exact sequence
\[
0\lra \omega_{Z_k}\lra \alpha_{k-1,k,*}\alpha_{k-1,k}^* \omega_{Z_k}\lra \mathcal{F}\lra 0,
\]
where $\mathcal{F}$ is a sheaf supported on $\sigma_{k-2,k}$. By Lemma \ref{higher-dualizing-Z} and Theorem \ref{thm:CM}, we have $R^1 \alpha_{k,*} \omega_{Z_k} = 0$. Hence we obtain a short exact sequence
\[
0\lra \alpha_{k,*} \omega_{Z_k}\lra \alpha_{k,*} \alpha_{k-1,k,*}\alpha_{k-1,k}^* \omega_{Z_k}\lra \alpha_{k,*} \mathcal{F}\lra 0.
\]
Note that $\alpha_{k,*} \mathcal{F}$ is supported on $\sigma_{k-2}$. By Lemma \ref{appearance-Hi-CM} and the proof of Lemma \ref{support-alpha-alpha-alpha}, we have
\[
\alpha_{k,*}\alpha_{k-1,k,*}\alpha_{k-1,k}^* \omega_{Z_k}\cong \omega_{\sigma_{k-1}}\otimes H^0(X,\omega_X).
\]
Therefore, $\alpha_{k,*} \omega_{Z_k}=0$ if and only if $H^0(X,\omega_X)=0$, as desired. 
\end{proof}

\begin{remark}
Recall that a variety $X$ is said to have \emph{weakly rational singularities} if there is a resolution of singularities $f\colon Y\lra X$ such that $f_* \omega_Y$ is reflexive.
Chou and Song \cite{Chou.Song.18} showed that the secant variety $\sigma_2$ has weakly rational singularities if and only if $H^n(X, \OO_X) = 0$. Using ideas from above, we briefly show that the same statement holds for $\sigma_k$ when $\dim(X) \leq 2$ or $k \leq 3$.  By Lemma \ref{dualizing-sheaf-secant-variety} and \ref{appearance-Hi-CM}, it we have an injection $\alpha_{k,*} \omega_{Z_k} \hookrightarrow \alpha_{k-1,*} \tilde{\tau}_*\alpha_{k-1,k}^*\omega_{Z_k} \cong H^0(X, \omega_X) \otimes \omega_{\sigma_{k-1}}$. It follows that $\alpha_{k,*} \omega_{Z_k} = 0$ if $H^n(X, \OO_X) = 0$. From \ref{dualizing-sheaf-secant-variety} and the Grauert--Riemenschneider vanishing theorem, we deduce that $\alpha_{k,*}\omega_{B^k} \cong \omega_{\sigma_k} := h^{-d_k}(\omega_{\sigma_k}^{\bullet})$. By \cite{Olano.Raychaudhury.Song}, this implies that $\sigma_k$ has weakly rational singularities.  Conversely, assume that $\alpha_{k,*}\omega_{Z_k} = 0$ holds. Let $W$ be the intersection of $Z_k \cap \pi_k^{-1}(X^{[k]} \setminus \partial X^{[k]})$ with $\alpha_k^{-1}(\sigma_{k-1} \setminus \sigma_{k-2})$. The restriction $\alpha_k|_W$ is flat because each fiber is isomorphic to the blow-up $X'$ of $X$ at $k-1$ distinct points. By Grauert's theorem, we have $H^0(X^{'}, \omega_{X^{'}}) = 0$. This implies that $H^n(X, \OO_X) = 0$.
\end{remark}

Recall that if $(Y, \Gamma)$ is a Kawamata log terminal pair, then $Y$ has rational singularities (\cite[Theorem 5.22]{Kollar-Mori}) and if  $(Y, \Gamma)$ is a log canonical pair, then $Y$ has Du Bois singularities (\cite[Corollary 6.32]{Kollar}). It is natural to ask when $\sigma_k$ has log terminal or log canonical singularities (see for example \cite[Question 1.6]{Chou.Song.18}). Here we give a necessary condition.

\begin{proposition}\label{not-klt-lc}
Suppose that there is a $\mathbb{Q}$-divisor $\Gamma$ on $\sigma_k$ such that $(\sigma_k, \Gamma)$ is Kawamata log terminal (respectively, log canonical). Then $-K_{\operatorname{Bl}_\eta X}$ is big (respectively, effective) for a general point $[\eta]\in X^{[k]}$. 
\end{proposition}

\begin{proof}
    We may write
    \[
    K_{B^k} + \alpha_{k,*}^{-1} \Gamma +Z_k = \alpha_{k}^* (K_{\sigma_k}+\Gamma)+a Z_k,
    \]
where $a$ is the log discrepancy of $Z_k$ with respect to $(\sigma_k, \Gamma)$. In particular, $a>0$ if $(\sigma_k, \Gamma)$ is Kawamata log terminal, and $a\geq 0$ if $(\sigma_k, \Gamma)$ is log canonical. Pick $x\in \pi_{k-1}^{-1}(X^{[k-1]}_{\textnormal{lci}}) \subseteq U^{k-1}(L)$, and set $[\eta] = \pi_{k-1}(x)$. By Proposition \ref{fiber-product} and the fact that $\eta$ is local complete intersection, we have $\alpha_k ^{-1}(x) \cong \operatorname{Bl}_\eta X$. Denote by $\operatorname{bl}_\eta \colon \operatorname{Bl}_\eta X \lra X$ the blow-up of $X$ along $\eta$ with exceptional divisor $E_\eta$. By Lemma \ref{fiber-conormal-decomposition} and adjunction, we have
\[
K_{\operatorname{Bl}_\eta X}+(\alpha_{k,*}^{-1} \Gamma)|_{\alpha_k ^{-1} (x)} = -a(\operatorname{bl}_{\eta}^* L - 2E_\eta).
\]
As $L$ is sufficiently positive, $\operatorname{bl}_{\eta}^* L - 2E_\eta$ is big. 
\end{proof}

\begin{corollary}\label{cor:noklt/lc}
Assume that $X$ is a surface.
If $k \geq 8$ (respectively, $k \geq 9$), then there is no $\mathbb{Q}$-divisor $\Gamma$ such that $(\sigma_k, \Gamma)$ is Kawamata log terminal (respectively, log canonical).
\end{corollary}

\begin{proof}
If $k \geq 8$ (respectively, $k \geq 9$) and $[\eta] \in X^{[k]}$ is general, then $-K_{\operatorname{Bl}_\eta X}$ is not big (respectively, not effective). Thus the corollary follows from Proposition \ref{not-klt-lc}.
\end{proof}

\subsection{Non-normal examples}\label{subsection-not-normal} 
Taking the opposite perspective of the above treatment, we exhibit non-normal secant varieties when the embedding line bundle is not positive enough. We no longer assume that $\dim (X) \leq 2$ or $k \leq 3$. Furthermore, $L$ is only a very ample line bundle on $X$, unless otherwise specified. Recall from Proposition \ref{terracini} that  $\sigma_k \setminus \sigma_{k-1}$ is smooth provided $X^{[k]}$ is smooth and $L$ separates $2k$-schemes. Here we want to show that this is sharp. For any given integers $n \geq 1$ and $k \geq 1$, we explicitly construct a smooth projective variety $X$ of dimension $n$ with a line bundle $L$ such that $L$ separates $(2k-1)$-schemes and $\sigma_k\setminus \sigma_{k-1}$ is not normal even when $X^{[k]}$ is smooth --- see Corollary \ref{cor:non-normal-examples}. The starting point is the following simple observation. 

\begin{lemma}\label{lemma-not-normal}
Let $L$ be a line bundle on $X$ that separates $(2k-1)$-schemes. Assume that there are finitely many reduced $2k$-subschemes $\eta_i$ for $1\leq i\leq \ell$ such that
    \begin{enumerate}[topsep=0pt]
    \item $H^0(X,L)\lra H^0(\eta_i,L|_{\eta_i})$ is not surjective for $1\leq i
        \leq \ell$, and
        \item $H^0(X,L)\lra H^0(\eta,L|_\eta)$ is surjective for all $2k$-subschemes $\eta\not\in\{\eta_i\}_{1\leq i\leq \ell}$.
    \end{enumerate}

\noindent Then $\sigma_k \setminus \sigma_{k-1}$ is not normal.
\end{lemma}
\begin{proof}
    Let $\{\xi_{i,j}\}$ be the set of all $k$-schemes contained in $\eta_i$. 
    By assumption $(2)$, we have that $\alpha_k$ is injective on $U^k\setminus (\cup_{i,j} \pi_k ^{-1}(\xi_{i,j}))$. 
    After possibly relabelling the indexes, we may write $\eta_1 = \xi_{1,1}\cup \xi_{1,2}$. By the fact that $L$ separates $(2k-1)$-schemes and by assumption $(1)$, we have that the linear spaces $\alpha_k(\pi_k ^{-1} (\xi_{1,1}))$ and $\alpha_k(\pi_k ^{-1} (\xi_{1,2}))$ meet at exactly one point, say $p$. Therefore $\alpha_k ^{-1}(p)$ is not connected, and hence, $\sigma_k $ is not normal at $p$. 
\end{proof}

In what follows, we will apply Lemma \ref{lemma-not-normal} to $X=\operatorname{Sym}^{\ell} (C)$ for a smooth projective curve $C$ and $L=N_{\ell,M}$ for some line bundle $M$ on $C$. This example is inspired by \cite[Example 5.14]{Ein.Niu.Park.20}.

\begin{lemma}\label{M-to-N-separation}
    Let $C$ be a smooth projective curve, and $M$ be a line bundle on $C$. Suppose that $M$ separates $(\ell+j)$-schemes and $H^1(C,M)=0$ for $\ell \geq 1$ and $j \geq 0$. Then the line bundle $N_{\ell,M}$ on $\operatorname{Sym}^{\ell}(C)$ separates $(j+1)$-schemes. 
\end{lemma}
\begin{proof}
We proceed by induction on $\ell$ and $j$. If $\ell=1$, then $\operatorname{Sym}^{\ell} (C) = C$ and $N_{\ell, M}=M$ so that there is nothing to prove. If $j=0$, then $N_{\ell,M}$ is globally generated since $E_{\ell, M}$ is globally generated. Now, we assume that $\ell \geq 2$ and $j \geq 1$. Let $\eta$ be a $(j+1)$-scheme in $\operatorname{Sym}^\ell (C)$. Pick a point $p\in C$ contained in the support of a $0$-cycle corresponding to a closed point of $\eta$, and consider the divisor $C_{p,\ell}=q_\ell(\{p\}\times C^{\ell-1}) \subseteq \operatorname{Sym}^{\ell} (C)$. Note that $C_{p, \ell} \cong \operatorname{Sym}^{\ell-1}(C)$ and $\mathcal{O}_{\operatorname{Sym}^{\ell} (C)}(C_{p, \ell}) = S_{\ell, \mathcal{O}_C(p)}$. By \cite[Remark 3.3 (1), (7)]{Ein.Niu.Park.20}, we have $N_{\ell, M}(-C_{p,\ell})\cong N_{\ell, M(-p)}$ and $N_{\ell,M}|_{C_{p,\ell}}\cong N_{\ell - 1, M(-p)}$. Let $\xi:=\eta \cap C_{p,\ell}$, and $\lambda:=\eta\setminus \xi$ be the residue scheme defined by $\mathcal{I}_{\lambda} = (\mathcal{I}_\xi : \mathcal{I}_\eta)$. By definition, there is a commutative diagram
\begin{center}
\begin{tikzcd}
0 \ar[r]&  \mathcal{O}_{\operatorname{Sym}^\ell (C)} (-C_{p,\ell})\ar[r]\ar[d]&
\mathcal{O}_{\operatorname{Sym}^\ell (C)}\ar[r]\ar[d]& \mathcal{O}_{C_{p,\ell}}\ar[r]\ar[d]& 0\\
0 \ar[r]&  \mathcal{O}_\lambda (-C_{p,\ell})\ar[r]&
\mathcal{O}_{\eta}\ar[r]& \mathcal{O}_\xi\ar[r]& 0
\end{tikzcd}.
\end{center}
Twisting by $N_{\ell, M}$, we get a commutative diagram
\begin{equation}\label{inductive-diagram}
\begin{tikzcd}
0 \ar[r]&  H^0(N_{\ell, M(-p)})\ar[r]\ar[d]&
H^0(N_{\ell, M})\ar[r]\ar[d]& H^0(N_{\ell-1, M(-p)})\ar[r]\ar[d]& H^1(N_{\ell, M(-p)})\\
0 \ar[r]& H^0(N_{\ell, M(-p)}\vert_\lambda)\ar[r]&
H^0(N_{\ell, M}\vert_\eta)\ar[r]& H^0(N_{\ell-1, M(-p)}\vert_\xi)\ar[r]& 0
\end{tikzcd}.
\end{equation}
Note that $H^1(C,M(-p))=0$ and hence $H^1(\operatorname{Sym}^\ell(C), N_{\ell, M(-p)})=0$ by \cite[Lemma 3.7]{Ein.Niu.Park.20}. We may now conclude by inductive hypothesis and the snake lemma. 
\end{proof}

\begin{proposition}\label{prop:M-to-N-separation-nonseparation}
    Let $C$ be a smooth projective curve of genus $g$, and set $M:=\omega_C(D)$, where $D$ is a general effective divisor with $\operatorname{deg}(D)=\ell+j\leq g-1$ for some $\ell\geq 1$ and $j\geq 0$. Then the line bundle $N_{\ell,M}$ on $\operatorname{Sym}^{\ell}(C)$ separates $j$-schemes, and the set of $(j+1)$-schemes $\eta\subseteq \operatorname{Sym}^\ell (C)$ for which $H^0(\operatorname{Sym}^\ell (C), N_{\ell, M}) \lra H^0(\eta, N_{\ell, M}|_\eta)$ is not surjective is finite and non-empty. 
\end{proposition}
\begin{proof}
Note that $\operatorname{deg}(M)=2g+\ell+j-2$ and hence $M$ separates $(\ell+j-1)$-schemes and $H^1(C,M)=0$. By Lemma \ref{M-to-N-separation}, $N_{\ell, M}$ separates $j$-schemes. This proves the first part of the proposition. For the second part, let $\eta$ be a $(j+1)$-scheme in $\operatorname{Sym}^{\ell}(C)$ such that $H^0(\operatorname{Sym}^\ell (C), N_{\ell, M}) \lra H^0(\eta, N_{\ell, M}|_\eta)$ is not surjective. 
We claim that $\eta$ is a reduced scheme whose closed points correspond to the $0$-cycle s of the form $E_{\eta}+x_i$ for $1 \leq i \leq j+1$ such that $E_{\eta}$ is an effective divisor of degree $\ell-1$ and $E_{\eta}+x_1 + \cdots + x_{j+1} = D$. 

To prove the claim, we proceed by induction on $\ell$. If $\ell=1$ or $j =0$, then the claim follows from Serre duality and the fact that $h^0(C,D)=1$. We assume $\ell\geq 2$ and $j \geq 1$. As in the proof of Lemma \ref{M-to-N-separation}, pick a point $p \in C$ contained in the support of a $0$-cycle corresponding to a closed point of $\eta$, and let $\xi:=\eta \cap C_{p,\ell}$ and $\lambda:=\eta \setminus \xi$. Note that $H^1(C, M(-p))=0$ and $H^1(\operatorname{Sym}^\ell(C), N_{\ell, M(-p)})=0$. 
If $p \not\in \operatorname{Supp}(D)$, then $M(-p)$ still separates $(\ell+j-1)$-schemes. By Lemma \ref{M-to-N-separation}, $N_{\ell, M(-p)}$ separates $j$-schemes and $N_{\ell-1, M(-p)}$ separates $(j+1)$-schemes. Considering the commutative diagram (\ref{inductive-diagram}), we see that $H^0(\operatorname{Sym}^\ell (C), N_{\ell, M}) \lra H^0(\eta, N_{\ell, M}|_\eta)$ is surjective. Thus we must have $p\in \operatorname{Supp}(D)$. 

If $\eta$ is contained in $C_{p,\ell}$ (that is, $\lambda = \emptyset$), then the claim follows by induction on $\ell$. Suppose that $\eta$ is not contained in $C_{p,\ell}$ (that is, $\lambda \neq \emptyset$). As $M(-p)$ separates $(\ell+j-2)$-schemes, Lemma \ref{M-to-N-separation} shows $N_{\ell-1, M(-p)}$ separates $j$-schemes so that $H^0(N_{\ell-1, M(-p)}) \lra H^0(N_{\ell-1, M(-p)}\vert_{\xi})$ is surjective. In view of the commutative diagram (\ref{inductive-diagram}), it remains to investigate surjectivity of the restriction map
$H^0(N_{\ell, M(-p)}) \lra H^0(N_{\ell,M(-p)}\vert_{\lambda})$.
If $\operatorname{length}(\lambda)\leq j-1$, this map is surjective by Lemma \ref{M-to-N-separation}. Assume that $\operatorname{length}(\lambda)=j$. By induction on $j$, this restriction map is also surjective unless $\lambda$ is a reduced scheme satisfying the conditions in the claim. Replacing $p$ by a closed point in $E_{\eta}$, we see that the restriction map $H^0(\operatorname{Sym}^\ell (C), N_{\ell, M}) \lra H^0(\eta, N_{\ell, M}|_\eta)$ is surjective or $\eta$ is contained in $C_{p,\ell}$. We have shown the claim.

Now, by the claim, there are only finitely many possibilities for $\eta$. Conversely, by induction on $\ell$, one may also check that if $\eta$ is a reduced scheme as in the claim, then $H^0(\operatorname{Sym}^\ell (C), N_{\ell, M}) \lra H^0(\eta, N_{\ell, M}|_\eta)$ is not surjective. 
\end{proof}

\begin{corollary}\label{cor:non-normal-examples}
For given integers $n \geq 1$ and $k \geq 1$, let $C$ be a smooth projective curve of genus $g \geq n+2k-1$, and $M:=\omega_C(D)$ for a general effective divisor $D$ of degree $n+2k-2$ so that $\deg D = n+2k-2 \leq g-1$. If $X:=\operatorname{Sym}^n (C)$ and $L:=N_{n,M}$, then $L$ separates $(2k-1)$-schemes but $\sigma_k \setminus \sigma_{k-1}$ is not normal.
\end{corollary}

\begin{proof}
Immediate from Lemma \ref{lemma-not-normal} and Proposition \ref{prop:M-to-N-separation-nonseparation}.
\end{proof}

\begin{remark}
Even if $L$ is not positive enough, $\sigma_k(X, L)$ may still be normal. For instance, Landsberg--Weyman \cite[Theorem 1.1]{Landsberg.Weyman.07} proved that every $k$-secant variety of the Segre embedding of $\mathbb{P}^1 \times \mathbb{P}^{b-1} \times \mathbb{P}^{c-1}$ is normal for any $b, c \geq 2$.
\end{remark}

\section{Syzygies of secant varieties}\label{section-syzygies}
In this section, we prove Theorem \ref{thm:N_{k+2,p}} on syzygies of secant varieties, Corollary \ref{cor:idealofsecants} describing their ideal in projective space, and Theorem \ref{thm:weight-one-syzygies} regarding weight-one syzygies.  

\subsection{Property $N_{k+1,p}$ for secant varieties}
While our approach to Theorem \ref{theoremA} is quite different from the approach of \cite{Ein.Niu.Park.20}, the proof of Theorem \ref{thm:N_{k+2,p}}, granting the main vanishing (Theorem \ref{main-vanishing-serre}), is almost identical to the proof in \cite{Ein.Niu.Park.20}. We include all the details for the reader's convenience. Fix an integer $p \geq 0$. Let $L$ be a sufficiently positive line bundle on $X$ giving an embedding $X \subseteq \mathbb{P} H^0(X, L) = \mathbb{P}^{r}$. We assume that $n = \dim (X) \leq 2$ or $k \leq 3$ so that $\sigma_k$ is normal and $\sigma_k \subseteq \mathbb{P}^r$ is projectively normal by Theorem \ref{thm:normalsing}. We shall show that $\sigma_k \subseteq \mathbb{P}^r$ satisfies property $N_{k+1,p}$. 

Set $M_{\sigma_k}:=M_{\mathcal{O}_{\sigma_k}(1)}$. We may identify $H^0(B^k, H_k)\cong H^0(X^{[k]}, E_{k,L})\cong H^0(X, L)\cong H^0(\sigma_k, \mathcal{O}_{\sigma_k}(1))$. Then $M_{H_k}:=\alpha_k^* M_{\sigma_k}$ is the kernel of the evaluation map $H^0(B^k, H_k) \lra \mathcal{O}_{B^k}(H_k)$. By the projection formula and the normality of $\sigma_k$, we have $\alpha_{k,*} \wedge^j M_{H_k} = \wedge^j M_{\sigma_k}$. 
Note, however, that for $i\geq 1$ we may have $R^i \alpha_{k,*} \wedge^j M_{H_k} \neq 0$, since $\sigma_k$ may not have rational singularities (see Theorem \ref{thm:rational-sing}). Projection formula and the Du Bois-type condition (Theorem \ref{Du Bois-type condition} (1)) yield $\alpha_{k,*} \wedge^j M_{H_k} \otimes \mathcal{O}_{B^k}(kH_k - Z_k) \cong \wedge^j M_{\sigma_k} \otimes \mathcal{I}_{\sigma_{k-1}/\sigma_k}(k)$ and $R^i \alpha_{k,*} \wedge^j M_{H_k} \otimes \mathcal{O}_{B^k}(kH_k - Z_k) =0$ for $i>0$. Recall from Proposition \ref{A&Z} that $\mathcal{O}_{B^k}(kH_k - Z_k) = \pi_k^* A_{k,L}$. We have a commutative diagram 
\begin{center}
\begin{tikzcd}
& & &0\ar[d] & \\
& 0\ar[d] & & K\ar[d] &\\
0\ar[r]& \pi_k ^* M_{k,L}\ar[r]\ar[d]& H^0(X,L)\otimes \mathcal{O}_{B^k}\ar[r]\ar[d,equals]& \pi_k ^* E_{k,L}\ar[r]\ar[d]& 0\\
0\ar[r]& M_{H_k}\ar[r]\ar[d]& H^0(X,L)\otimes \mathcal{O}_{B^k}\ar[r]& \mathcal{O}_{B^k}(1)\ar[r]\ar[d]& 0\\
& K\ar[d] & & 0 & \\
& 0 & & &\\
\end{tikzcd}
\end{center}
where the right column is the relative Euler sequence, the top row is sequence (\ref{secSES}), and the rest is completed by exactness. 

\begin{lemma}\label{lem:vanishing-higher-direct-wedge-K}
We have the following:
\begin{enumerate}[topsep=0pt]
 \item $R^i \pi_{k,*} \wedge^j K = 0$ for $i \geq 0$ and $j \geq 1$. 
 \item $\pi_{k,*} \wedge^j M_{H_k} = \wedge^j M_{k,L}$ for $j \geq 0$ and $R^i \pi_{k,*} \wedge^j M_{H,k}=0$ for $i \geq 1$ and $j \geq 0$.
\end{enumerate} 
\end{lemma}

\begin{proof}
Notice that if $F$ is a fiber of $\pi_k$, then $F \cong \mathbb{P}^{k-1}$ and $K|_{F} \cong \Omega^1_{\mathbb{P}^{k-1}}(1)$. Thus (1) immediately follows from Bott's formula. Part (2) is trivial when $j=0$, so we may assume $j \geq 1$. Consider a filtration
\[
\wedge^j M_{H_k} = F^0 \supseteq F^1\supseteq \cdots \supseteq F^j \supseteq F^{j+1}=0
\]
with $F^t / F^{t+1}\cong \pi_k ^* (\wedge ^t M_{k,L}) \otimes \wedge^{j-t}K$. Then (2) follows from the projection formula and (1). 
\end{proof}

\begin{lemma}\label{vanishing-M-inductive}
Pick $p\geq 0$. We have
$H^i(\sigma_k, \wedge^{j} M_{\sigma_k}\otimes \mathcal{I}_{\sigma_{k-1}/\sigma_k}(k))=0$ for $i \geq 1$ and $i \leq j \leq p$.
\end{lemma}

\begin{proof}
Using the projection formula and the Du Bois-type condition (Theorem \ref{Du Bois-type condition} (1)), we obtain
$$
H^i(\sigma_k, \wedge^{j} M_{\sigma_k}\otimes \mathcal{I}_{\sigma_{k-1}/\sigma_k}(k)) \cong H^i(B^k, \wedge^{j} M_{H_k}(k H_k - Z_{k})).
$$
Since $\mathcal{O}_{B^k}(kH_k - Z_k) \cong \pi_k^* A_{k,L}$, projection formula and Lemma \ref{lem:vanishing-higher-direct-wedge-K} show 
$$
H^i(B^k, \wedge^{i} M_{H_k}(k H_k - Z_{k})) \cong H^i(X^{[k]}, \wedge^{i} M_{k,L} \otimes A_{k,L}). 
$$
Then the assertion follows from the main vanishing theorem (Theorem \ref{main-vanishing-serre}).
\end{proof}

The following is Theorem \ref{thm:N_{k+2,p}}. 

\begin{theorem}\label{thm:syz}
$\sigma_k \subseteq \mathbb{P}^{r}$ satisfies property $N_{k+1, p}$. 
\end{theorem}

\begin{proof}
We need to prove that $K_{i,j}(\sigma_k; \OO_{\sigma_k}(1))=0$ for $0 \leq i \leq p$ and $j \geq k+1$. We proceed by induction on $k$. The case $k=1$ follows from \cite[Theorem 1]{Ein.Lazarsfeld}, so we may assume $k\geq 2$. 
Recall from Theorem \ref{thm:normalsing} that $H^i(\sigma_k, \OO_{\sigma_k}(m))=0$ for $i\geq 1$ and $m\geq 1$. Then \cite[Proposition 2.1]{Park2} gives
$$
K_{i,j}(\sigma_k; \OO_{\sigma_k}(1))\cong H^{j-k}(\sigma_k, \wedge^{i+j-k} M_{\sigma_k} (k))~~\text{ for $i \geq 0$ and $j \geq k+1$}.
$$
Now, consider the short exact sequence
$$
0 \longrightarrow \cI_{\sigma_{k-1}/\sigma_k} \longrightarrow \OO_{\sigma_k} \longrightarrow \OO_{\sigma_{k-1}} \longrightarrow 0.
$$
By induction, we have
$$
H^{j-k}(\sigma_{k-1}, \wedge^{i+j-k} M_{\sigma_{k-1}}(k)) \cong K_{i,j}(\sigma_{k-1}; \OO_{\sigma_{k-1}}(1))=0~~\text{ for $0 \leq i \leq p$ and $j \geq k+1$}.
$$
Thus it is enough to show that
$$
H^{j-k}(\sigma_k, \wedge^{i+j-k} M_{\sigma_k} \otimes \mathcal{I}_{\sigma_{k-1}/\sigma_k}(k))=0 ~~\text{ for $0 \leq i \leq p$ and $j \geq k+1$}.
$$
But this immediately follows from Lemma \ref{vanishing-M-inductive}.
\end{proof}

\begin{remark}\label{rem:effectiveforN_{k+1,p}}
Theorem \ref{thm:syz} holds if $L$ separates $2k$-schemes and
$$
H^i(X^{[m]}, \wedge^{j} M_{m,L} \otimes A_{m,L})~~\text{ for $i \geq 1$ and $0 \leq j \leq i+p$},
$$    
where $1 \leq m \leq k$ (see Remark \ref{rem:effectivefornormalsing}).
\end{remark}

Recall that Danila's theorem (Theorem \ref{Danila}) implies $H^0(\sigma_k, \OO_{\sigma_k}(\ell))=H^0(\mathbb{P}^r, \OO_{\mathbb{P}^r}(\ell))$ for $1 \leq \ell \leq k$. In the situation of Theorem \ref{thm:syz}, for $0 \leq i \leq p$ and $j \geq 0$, we have that 
$$
K_{i,j}(\sigma_k; \OO_{\sigma_k}(1)) \neq 0 ~\Longleftrightarrow~ \text{either $i=j=0$ or $i \geq 1$ and $j=k$}.
$$
In other words, the minimal free resolution of the section ring $R(\sigma_k, \OO_{\sigma_k}(1))$ is as simple as possible for the first $p$ steps.

\subsection{Equations defining secant varieties}
Theorem \ref{thm:syz} tells us that if $L$ is sufficiently positive, then $\sigma_k \subseteq \mathbb{P}^r$ satisfies property $N_{k+1,1}$. Together with Danila's theorem (Theorem \ref{Danila}), we have the following.

\begin{proposition}\label{prop:idealofsecants}
The ideal $I_{\sigma_k / \mathbb{P}^r}$ is generated by homogeneous polynomials of degree $k+1$, and 
$$
H^0(\mathbb{P}^r, \mathcal{I}_{\sigma_k/\mathbb{P}^r}(k+1)) \cong H^0(X^{[k+1]}, A_{k+1, L}).
$$
\end{proposition}

\begin{proof}
From property $N_{k+1,1}$, we see that $I_{\sigma_k / \mathbb{P}^r}$ is generated by homogeneous polynomials of degree $\leq k+1$. 
However, $H^0(\mathbb{P}^r, \OO_{\mathbb{P}^r}(m)) \cong H^0(\sigma_k, \OO_{\sigma_k}(m))$ for $1 \leq m \leq k$ by Theorem \ref{Danila}, so  $I_{\sigma_k / \mathbb{P}^r}$ does not contain homogeneous polynomials of degree strictly less than $k+1$. This proves the first assertion. Now, consider the short exact sequence
$$
0 \lra \mathcal{O}_{B^{k+1}}((k+1)H_{k+1}-Z_{k+1}) \lra \mathcal{O}_{B^{k+1}}(k+1) \lra \mathcal{O}_{Z_{k+1}}(k+1) \lra 0.
$$
Recall from Proposition \ref{A&Z} that $\mathcal{O}_{B^{k+1}}((k+1)H_{k+1}-Z_{k+1}) \cong \pi_{k+1}^* A_{k+1, L}$. By Theorem \ref{Danila} and Lemma \ref{lem:alpha_*O_Z}, we have
$$
H^0(\mathcal{O}_{B^{k+1}}(k+1)) \cong H^0(\mathcal{O}_{\mathbb{P}^r}(k+1))~~\text{ and }~~H^0(\mathcal{O}_{Z_{k+1}}(k+1)) \cong H^0(\mathcal{O}_{\sigma_k}(k+1)).
$$
Thus we obtain the second assertion.
\end{proof}

The following is Corollary \ref{cor:idealofsecants} and gives a nice description of the generators of $I_{\sigma_k/\mathbb{P}^r}$.

\begin{corollary}\label{thm:ideals}
The ideal $I_{\sigma_k / \mathbb{P}^r}$ is generated by the $(k+1)^{\text{th}}$ graded piece of $I_{X/\mathbb{P}^r}^{(k)}$, the $k^{\text{th}}$ symbolic power of the ideal $I_{X/\mathbb{P}^r}$.
\end{corollary}
\begin{proof}
By Proposition \ref{prop:idealofsecants}, the ideal $I_{\sigma_k / \mathbb{P}^r}$ is generated in degree $k+1$. On the other hand, $X \subseteq \mathbb{P}^r$ satisfies property $N_{2,1}$ by \cite[Theorem 1]{Ein.Lazarsfeld}, so $I_{X/\mathbb{P}^r}$ is generated by quadrics. Then the result follows from \cite[Theorem 1.2, Corollary 2.10] {Sidman.Sullivant}.
\end{proof}

\subsection{A vanishing theorem on weight-one syzygies of algebraic surfaces}\label{subsection-weight-one-syzygies}
From a somewhat different perspective, we consider weight-one syzygies of algebraic surfaces. Let $S$ be a smooth projective surface, $B$ be a line bundle on $S$, and let $L$ be a sufficiently positive line bundle. Recall from \cite[Theorem A]{Agostini} that if $K_{p,1}(S, B; L)=0$, then $B$ is $p$-very ample. Here we prove that the converse also holds, so we obtain Theorem \ref{thm:weight-one-syzygies}.

\begin{theorem}\label{thm:wt1syz}
If $B$ is $p$-very ample, then $K_{p,1}(S, B; L)=0$.
\end{theorem}

\begin{proof}
By \cite[Proposition 5.4]{Agostini}, it is enough to check that $R^1 h_{*} (M_{k,B} \otimes \delta_{p+1}^{-1}) = 0$. But this follows from Theorem \ref{main-direct-image-vanishing} with $\ell=1$ and $j=1$. For the reader's convenience, we explain \cite[Proposition 5.4]{Agostini} more in detail. 
As $K_{p,1}(S, B; L)$ is the cokernel of the map 
$$
H^0(S, B) \otimes H^0(S^{[p+1]}, N_{p+1,L}) \lra H^0(S^{[p+1]}, E_{p+1,B} \otimes N_{p+1,L})
$$
by \cite[Lemma 5.1]{Agostini} (see also \cite[Corollary 5.5 and Remark 5.6]{Aprodu.Nagel.2010}), it suffices to show that $H^1(S^{[p+1]}, M_{p+1,B} \otimes N_{p+1,L})=0$. Then Theorem \ref{main-direct-image-vanishing} for $\ell=1$ and $j=1$ implies that
$$
H^1(S^{[p+1]}, M_{p+1,B} \otimes N_{p+1,L}) \cong H^1(\operatorname{Sym}^{p+1}(S), h_{*} (M_{p+1,B} \otimes \delta_{p+1}^{-1}) \otimes S_{p+1, L}),
$$
but this is equal to $0$ by Fujita's vanishing theorem.
\end{proof}

\begin{remark}\label{rem:weight-one-vanishing}
Note that $H^1(X^{[p+1]}, M_{p+1,B} \otimes N_{p+1,L})=0$ implies $K_{p,1}(X, B; L)=0$ even when $\dim(X) \geq 3$ and $0 \leq p \leq 2$ by \cite[Corollary 5.5 and Remark 5.6]{Aprodu.Nagel.2010}. Therefore, the above theorem holds also for $\dim(X) \geq 3$ and $0 \leq p \leq 2$ thanks to Theorem \ref{main-direct-image-vanishing} and Fujita's vanishing theorem.
\end{remark}

\begin{remark}\label{rmk:detailed-discussion-weight-one}
As mentioned in the introduction, the proof of the gonality conjecture by Ein and Lazarsfeld \cite{Ein.Lazarsfeld.Gonality} suggests that our Theorem \ref{thm:wt1syz} may be regarded as a generalization of the gonality conjecture (for curves) to surfaces. Moreover, the theorem can be regarded as fitting into a more general framework concerning the asymptotic non-vanishing of Koszul cohomology --- we take a moment here to discuss this connection in more detail. With $X$ a smooth projective variety of arbitrary dimension $n$, let $B$ denote a line bundle on $X$ and set $L_d = A^d\otimes P$ for $A$ ample and $P$ arbitrary. Theorem \cite[Theorem 1.1]{Park2}, which originated in Ein and Lazarsfeld's influential paper \cite{Ein.Lazarsfeld.Asymptotic} and was completed by the third author \cite{Park1, Park2}, states the following: for $1 \leq q \leq n$, there exist functions $c_q(d)$ and $c_q'(d)$ with
$$
c_q(d) = \Theta(d^{q-1})~~\text{ and }~~c_q'(d) = \begin{cases} \Theta(d^{n-q}) & \text{if $H^{q-1}(X, B) =0$ or $q=1$} \\  q-1 & \text{if $H^{q-1}(X, B) \neq 0$ and $q \geq 2$} \end{cases}
$$
such that if $d$ is sufficiently large, then
$$
K_{p,q}(X, B; L_d) \neq 0~~\Longleftrightarrow~~c_q(d) \leq p \leq r_d-c_q'(d)
$$
(here the $\Theta$ symbol is used to refer to the so-called ``big-Theta" notation\footnote{Recall that for real-valued $g$ we write $g(d) = \Theta(f(d))$ to mean that there are constants $k_1$ and $k_2$ for which $k_1f(d) \leq g(d) \leq k_2f(d)$ for all $d$.}). A similar statement holds when $q=0$ or $n+1$ (see \cite[Proposition 5.1 and Corollary 5.2]{Ein.Lazarsfeld.Asymptotic}).\\
\\
For curves, Ein--Lazarsfeld's generalization and proof of the gonality conjecture shows that $c_1(d)$ is determined exactly by the $p$-very ampleness of $B$, hence in particular is constant as a function of $d$. It is natural to ask to what extent $c_1(d)$ behaves similarly in higher dimensional cases. Yang \cite{Yang} proved that $c_1(d)$ is a constant function of $d$, answering \cite[Problem 7.2]{Ein.Lazarsfeld.Asymptotic}. Along similar lines, Ein--Lazarsfeld--Yang \cite[Theorem A]{Ein.Lazarsfeld.Yang} proved that if $B$ is $p$-jet very ample, then $K_{p,1}(X, B; L_d)=0$ (hence providing a lower bound on the constant $c_1(d)$ in terms of the $p$-jet very ampleness of $B$). Conversely, Agostini \cite[Theorem A]{Agostini} proved that if $K_{p,1}(X, B; L_d)=0$, then $B$ is $p$-very ample (i.e. $B$ is $p$-very ample whenever $p < c_1(d)$). Taken together, these statements are saying that a lower bound on the constant $c_1(d)$ tells us about $p$-very ampleness of $B$, while $p$-jet-very ampleness of $B$ gives us lower bounds on $c_1(d)$. However, $p$-jet-very ampleness is still a priori stronger than $p$-very ampleness even for the surface case (\cite[Remark 2.3]{Ein.Lazarsfeld.Yang}), so it is still unclear how $c_1(d)$ is determined. Note that $K_{p,1}(X, B; L)$ can be calculated from the smoothable component of $X^{[p+1]}$ even when $X^{[p+1]}$ has several components,(see \cite[\S2]{Voisin.2002} or \cite[\S5.1]{Aprodu.Nagel.2010}). In order to check if $B$ is $p$-very ample, however, one needs to check for all schemes of length $(p+1)$, not just the ones that are smoothable (see \cite[Footnote 9]{Ein.Lazarsfeld.survey}). Taking this point of view, Ein--Lazarsfeld--Yang suggested that $K_{p,1}(X, B; L_d) = 0$ could be equivalent to $p$-jet very ampleness (\cite[Remark 2.2]{Ein.Lazarsfeld.Yang}) of $B$. Theorem \ref{thm:weight-one-syzygies}, however, shows that $K_{p,1}(X, B; L_d) = 0$ is indeed equivalent to $p$-very ampleness of $B$ if $\operatorname{dim}(X)=2$, and completely resolves \cite[Problem 4.12]{Ein.Lazarsfeld.survey} for surfaces.
\end{remark}

\section{Questions}\label{section-questions}

We conclude the paper with some natural questions. We hope these questions will inspire and motivate future research. Unless otherwise stated, here $X$ is a smooth projective variety of arbitrary dimension $n$, $L$ is a sufficiently positive line bundle on $X$ giving an embedding $X \subseteq \mathbb{P} H^0(X, L) = \mathbb{P}^r$ unless otherwise stated, and $k$ is an arbitrary positive integer.

First, we would like to know if the $k$-secant variety $\sigma_k$ is always normal.

\begin{question}\label{question-secant-normal}
Is $\sigma_k$ normal for all $n$ and $k$?
\end{question}

\noindent We expect that Theorem \ref{thm:normalsing} (1) could be extended to the case when $X^{[k]}$ has normal singularities. However, we believe Theorem \ref{terracini} suggests that the answer to Question \ref{question-secant-normal} is negative in general. In particular, we ask the following.

\begin{question}\label{question-hilbert-normal}
Is $(\mathbb{A}^n)_{\operatorname{sm,Gor}}^{[k]}$ normal for all $n$ and $k$?
\end{question}

\noindent It seems very unlikely that Question \ref{question-hilbert-normal} has a positive answer. To the best of the author's knowledge, it is not even known whether $(\mathbb{A}^n)_{\operatorname{sm}}^{[k]}$ is normal for all $n$ and $k$. Notice that a counterexample to Question \ref{question-hilbert-normal} would provide a counterexample to Question \ref{question-secant-normal}. On the other hand, the converse implication is mysterious. It would be very interesting to see to what extent to which Question \ref{question-hilbert-normal} determines Question \ref{question-secant-normal}. 

\begin{question}\label{question-singularities-lower-secant}
Is it true that $\sigma_k$ is normal if and only if $X_{\operatorname{sm,Gor}}^{[k]}$ is normal?
\end{question}

As previously pointed out, the cactus scheme may be strictly larger than the secant variety in higher dimensions. It is natural to try understanding exactly when
they are the same and when not, and to investigate the geometry of the other components of the cactus variety (see for example \cite{GMR23}). A starting point would be the following.

\begin{question}\label{question-secant-vs-cactus}
For which $n$ and $k$, do we have $(\mathbb{A}^n)_{\operatorname{Gor}}^{[k]}\subseteq (\mathbb{A}^n)_{\operatorname{sm}}^{[k]}$? 
\end{question}

\noindent As mentioned in Section \ref{section-hilbert-scheme}, we know that $k\leq 13$ or $\dim (X) \leq 5$ and $k\leq 14$ suffices for $(\mathbb{A}^n)_{\operatorname{Gor}}^{[k]}\subseteq (\mathbb{A}^n)_{\operatorname{sm}}^{[k]}$ by \cite{Casnati.Jelisiejew.Notari.15}. We are asking for a complete answer.

Next, we would like to determine the precise singular locus of $\sigma_k$. In view of Corollary \ref{corollary-singularities-secant} and Remark \ref{rem:singular-locus}, we ask the following.

\begin{question}
What points in $(\kappa_{k-1} \cap \sigma_k) \setminus \sigma_{k-1}$ are singular points of $\sigma_k$? Is $\sigma_k$ normal along $(\kappa_{k-1} \cap \sigma_k) \setminus \sigma_{k-1}$?
\end{question}

Regarding singularities of secant varieties, Chou and Song \cite[Question 1.6]{Chou.Song.18} ask when $\sigma_2$ has log terminal or log canonical singularities. Chou--Song's original question is still open. Here we ask the following generalized question.

\begin{question}\label{question-klt-lc}
When does there exist a $\mathbb{Q}$-divisor $\Gamma$ on $\sigma_k$ such that
    \begin{enumerate}[topsep=0pt]
        \item $(\sigma_k,\Gamma)$ is Kawamata log terminal, or
        \item $(\sigma_k,\Gamma)$ is log canonical?
    \end{enumerate}
\end{question}

\noindent In the case $\dim(X)=1$, \cite[Theorem 1.1]{Ein.Niu.Park.20} gives a very satisfactory answer: (1) holds if and only if $X$ is isomorphic to $\mathbb{P}^1$, in which case $\sigma_k$ is a Fano variety, and (2) holds if and only if (1) holds or $X$ is an elliptic curve, in which case $\sigma_k$ is a Calabi--Yau variety. When $\dim(X)\geq 2$, we showed a necessary condition in Proposition \ref{not-klt-lc}. This implies that 
$(\sigma_k, \Gamma)$ is never Kawamata log terminal (respectively, log canonical) when $\dim(X)=2$ and $k \geq 8$ (respectively, $k \geq 9$), see Corollary \ref{cor:noklt/lc}.

Recently, Olano--Raychaudhury--Song \cite{Olano.Raychaudhury.Song} further studied singularities of $\sigma_2$ from a Hodge-theoretic perspective. In particular, they give a necessary and sufficient condition for $\sigma_2$ to have higher Du Bois singularities. One may hope to generalize their work to higher secant varieties.

\begin{question}
When does $\sigma_k$ have higher Du Bois singularities?
\end{question}

In a different direction, one might want a better understanding of how much positivity of the embedding line bundle is needed for our main results --- Theorem \ref{theoremA} and Theorem \ref{thm:N_{k+2,p}}. 

\begin{question}\label{question-effective-results}
Assume that $\dim (X) \leq 2$ or $k \leq 3$. Let $L:=\omega_X \otimes A^m \otimes B$, where $A$ is an ample line bundle and $B$ is a nef line bundle on $X$, and fix an integer $m_0 \geq 0$. Is there an effective bound on $m_0$ independent of $B$ for which Theorem \ref{theoremA} holds when $m \geq m_0$? Similarly, is there an effective bound on $m_p$ depending on $p$ but independent of $B$ for which Theorem \ref{thm:N_{k+2,p}} hold when $m \geq m_p$?
\end{question}

\begin{question}\label{question-precise}
If Question \ref{question-effective-results} admits a positive answer, what are the optimal values of $m_0$ and $m_p$?
\end{question}

Beside the assumption that $L$ separates $2k$-schemes, the only place we use the the sufficient positivity of $L$ is where we apply the main vanishing theorem \ref{main-vanishing-serre} (see Remark \ref{rem:effectivefornormalsing} and Remark \ref{rem:effectiveforN_{k+1,p}}). In view of Remark \ref{rem:meaningofsuffpos}, it is natural to ask the following.

\begin{question}\label{question-effective}
Assume that $\dim (X) \leq 2$ or $k \leq 3$. Let $L:=\omega_X \otimes A^m \otimes B$, where $A$ is an ample line bundle and $B$ is a nef line bundle on $X$, and fix an integer $j_0 \geq 0$. What is an effective choice of $m_0$ depending on $j_0$ but independent of $B$ such that if $m \geq m_0$, then
$$
H^i(X^{[k]}, \wedge^j M_{k,L}\otimes A_{k,L})=0~~\text{ for $i \geq 1$ and $0 \leq j \leq j_0$?}
$$
\end{question}

\noindent For $k=1$, the question is closely related to Mukai's question (\cite[Problem 4.1]{Ein.Lazarsfeld}), which is known to be notoriously difficult even for the surface case. It is quite reasonable to first consider the case that $A$ is very ample or ample and base point free as in \cite{Ein.Lazarsfeld} or \cite{Bangere.Lacini.24}. This seems to be already very challenging. We only have an answer for $k=2$ and $j \leq i$ when $A$ is very ample (see Theorem \ref{main-effective-vanishing}).

In the setting of Question \ref{question-effective}, the $2$-secant variety $\sigma_2$ has normal Du Bois singularities when $m \geq 2n+2$ by \cite{Ullery.16} and \cite{Chou.Song.18}, and $\sigma_2 \subseteq \mathbb{P}^r$ is projectively normal when $m \geq 4n$ by Theorem \ref{effective-projective-normal}. Drawing evidence from these results, a first step towards Question \ref{question-precise} would be answering the following.

\begin{question}
Let $L$ be a line bundle of the form $L:=\omega_X\otimes A^m\otimes B$ with $A$ very ample and $B$ nef. Is $\sigma_2 \subseteq \mathbb{P}^r$ projectively normal when $m\geq 2n+2$?
\end{question}

On the other hand, we have seen that $I_{\sigma_k/\mathbb{P}^r}$ is generated by the $(k+1)^{\text{th}}$ graded piece of $I_{X/\mathbb{P}^r}^{(k)}$ (see Corollary \ref{thm:ideals}).  Smoothness of $X$ is needed by \cite{BGL}.
We would like to know whether the smoothness of $X^{[k]}$ is essential.

\begin{question}
Is $I_{\sigma_k/\mathbb{P}^r}$ generated by the $(k+1)^{\text{th}}$ graded piece of $I_{X/\mathbb{P}^r}^{(k)}$ for all $n$ and $k$?
\end{question}

By work of Ein--Lazarsfeld \cite{Ein.Lazarsfeld.Asymptotic} and the third author \cite{Park1, Park2}, we know that there is a surprisingly uniform asymptotic behavior of vanishing and nonvanishing of $K_{p,q}(X; L)$ as the positivity of $L$ grows. When $X=C$ is a curve, Choe--Kwak--Park \cite{Choe.Kwak.Park.23} proved that the asymptotic behavior of vanishing and nonvanishing of $K_{p,q}(\sigma_k; \OO_{\sigma_k}(1))$ is quite similar to that of $K_{p,q}(C; L)$. We are wondering whether the same is true in higher dimensions. 

\begin{question}
What can be said about the asymptotic behavior of $K_{p,q}(\sigma_k; \OO_{\sigma_k}(1))$ as the positivity of $L$ grows?
\end{question}

One certainly wants an effective result for Theorem \ref{thm:weight-one-syzygies}. In the curve case, a sharp effective bound for the degree of line bundle to satisfy the gonality conjecture is established in \cite{Niu.Park}. In general, Agostini \cite[Theorem A]{Agostini} proved an effective result for nonvanishing of $K_{p,1}(X, B; L)$. We are wondering how much positivity of the line bundle $L$ is needed for vanishing of $K_{p,1}(X, B; L)$. 

\begin{question}
Let $S$ be a smooth projective surface, let $B$ be a line bundle on $S$, set $L:=\omega_S \otimes A^m \otimes P$, where $A$ is an ample line bundle and $P$ is a nef line bundle on $S$, and fix integers $k \geq 1$ and $p \geq 0$. Is there an effective bound on $m_0$ independent of $P$ for which Theorem \ref{thm:wt1syz} holds when $m \geq m_0$? If so, what is the optimal value of $m_0$?
\end{question}

\noindent It is also reasonable to first consider the case that $A$ is very ample or ample and base point free. One may need to put further positivity assumptions on $P$ as in \cite[Theorem A]{Agostini}. 

Finally, for the higher dimensional generalization of Theorem \ref{thm:weight-one-syzygies}, we ask the following.

\begin{question}
Assume that $\dim(X) \geq 3$ and $B$ is a line bundle on $X$. For a fixed integer $p \geq 0$, is it true that $B$ is $p$-very ample if and only if $K_{p,1}(X, B; L)=0$?
\end{question}

\noindent This is true when $0 \leq p \leq 2$ by Remark \ref{rem:weight-one-vanishing}, and the converse implication 
$(\Longleftarrow)$ for any $p$ is known by \cite[Theorem A]{Agostini}. However, as was remarked in \cite[Footnote 9]{Ein.Lazarsfeld.survey} and \cite[Remark 2.2]{Ein.Lazarsfeld.Yang}, the answer may not be positive. We are contemplating whether the answer is positive when $X^{[p+1]}$ is irreducible.

\appendix

\section{The cotangent sheaf of iterated Quot schemes}\label{Section-cotangent}

In this appendix we use different notation and conventions from the main body of the paper. Here instead we let $p \colon X\lra S$ be a morphism of Noetherian schemes with $p_*\OO_X \cong \OO_S$, $\mathcal{H}$
a coherent sheaf on $X$, and $\mathcal{L}$ a line bundle on $X$ which is relatively ample
over $S$. For a fixed numerical polynomial $P\in\mathbb{Q}[t]$, recall that the Quot scheme $\text{Quot}(\mathcal{H},P)$ represents the functor
\begin{center}
\begin{tikzcd}[row sep = tiny]
h \colon (\text{\emph{Sch}}/S)^{\text{op}} \ar[r] & \text{\emph{Sets}}\\
T \ar[r,mapsto] & \left\{\text{\stackanchor{isomorphism classes of quotients $\mathcal{H}_T\twoheadrightarrow F$}{flat over $T$, with $\chi((F\otimes \mathcal{L}^{\otimes n})_t) = P(n)$ for $t\in T$}}\right\}\\
(T'\overset{g}{\rightarrow} T) \ar[r,mapsto] & ([\mathcal{H}_T\twoheadrightarrow F] \mapsto [\mathcal{H}_{T'}\twoheadrightarrow F_{T'}]),
\end{tikzcd}
\end{center}
where the subscript denotes the base change. Assume that $p$ is smooth of relative dimension $d$. The main result of \cite{Lehn.98} identifies the cotangent sheaf of $\text{Quot}(\mathcal{H},P)$ as
$$
\Omega^1_{\text{Quot}(\mathcal{H},P)/S} \cong \SExt^d_{\pr_2}(\mathcal{F},\mathcal{K}\otimes \pr_1 ^*\omega_{X/S}),
$$
where the maps $\pr_1$ and $\operatorname{pr}_2$ out of $X\times_S \text{Quot}(\mathcal{H},P)$ are the natural projections, the right-hand side has definition $\SExt^d_{\pr_2}(\mathcal{F},-) := R^d(\pr_{2,*}\SHom(\mathcal{F},-))$, and \[0 \lra \mathcal{K} \lra \pr_2^*\mathcal{H} \lra \mathcal{F} \lra 0\] is the tautological exact sequence on $X\times_S \text{Quot}(\mathcal{H},P)$.

In what follows, our purpose is twofold. First, we aim to generalize the above description of the cotangent sheaf to the iterated Quot scheme $\text{Quot}(\mathcal{H},P_1,P_2)$, where $P_i \in \mathbb{Q}[t]$ are numerical polynomials with $P_1(n) \geq P_2(n)$ for $n \gg 0$. This scheme represents the functor
\begin{center}
    \begin{tikzcd}[row sep = tiny]
        h' \colon (\text{\emph{Sch}}/S)^{\text{op}} \ar[r] & \text{\emph{Sets}}\\
        T \ar[r,mapsto] & \left\{\begin{array}{c}\text{isomorphism classes of subsequent quotients $\mathcal{H}_T\twoheadrightarrow F \twoheadrightarrow G$}\\ \text{with $F$ and $G$ both flat over $T$,}\\ \text{$\chi((F\otimes \mathcal{L}^{\otimes n})_t)=P_1(n)$ and $\chi((G\otimes \mathcal{L}^{\otimes n})_t) = P_2(n)$ for $t\in T$}\end{array} \right\}\\
        (T'\overset{g}{\rightarrow} T) \ar[r,mapsto] & ([\mathcal{H}_T\twoheadrightarrow F \twoheadrightarrow G] \mapsto [\mathcal{H}_{T'}\twoheadrightarrow F_{T'}\twoheadrightarrow G_{T'}]).
    \end{tikzcd}
\end{center}
We then have the following:

\begin{theorem}\label{cotan-nested}
The cotangent sheaf of the scheme $\operatorname{Quot}(\mathcal{H},P_1,P_2)$ is isomorphic to
$$
\Omega_{\operatorname{Quot}(\mathcal{H},P_1,P_2)/S}^1 \cong \textnormal{coker}\left(\SExt^d_{\pr_2}(\mathcal{G},\mathcal{K}) \lra \SExt^d_{\pr_2}(\mathcal{F},\mathcal{K})\oplus \SExt^d_{\pr_2}(\mathcal{G},\mathcal{J})\right)\otimes \pr_1^*\omega_{X/S},
$$
where the maps $\pr_1$ and $\operatorname{pr}_2$ out of $X\times_S \operatorname{Quot}(\mathcal{H},P_1,P_2)$ are the natural projections and
\begin{center}
\begin{tikzcd}
    0 \ar[r] & \mathcal{K} \ar[r]\ar[d,hookrightarrow] & \pr_2^*\mathcal{H} \ar[r]\ar[d,equals] & \mathcal{F} \ar[r]\ar[d,twoheadrightarrow] & 0\\
    0 \ar[r] & \mathcal{J} \ar[r] & \pr_2^*\mathcal{H} \ar[r] & \mathcal{G} \ar[r] & 0
\end{tikzcd}
\end{center}
is the tautological commutative diagram on $X\times_S \operatorname{Quot}(\mathcal{H},P_1,P_2)$.
\end{theorem}

Of course, the description of the cotangent sheaf of $\operatorname{Quot}(\mathcal{H},P)$ also yields a description of the tangent sheaf by Grothendieck duality:
\[
T_{\operatorname{Quot}(\mathcal{H}, P)/S}\cong
\operatorname{pr}_{2,*} \SHom(\mathcal{K}, \mathcal{F}).
\]
The natural pairing
\[
\mathcal{K}\otimes \operatorname{pr}_2 ^* 
\operatorname{pr}_{2,*} \SHom(\mathcal{K}, \mathcal{F})\lra
\mathcal{F}.
\]
induces a map
\[
\mathcal{K}\lra
\SHom\left(\operatorname{pr}_2 ^* T_{\operatorname{Quot}(\mathcal{H},P)/S} , \mathcal{F}\right).
\]
Our second aim is to describe this map geometrically as the differential of the universal sequence. 

Theorem \ref{cotan-nested} follows without difficulty from \cite{Lehn.98}, and is well-known to the experts. We reproduce Lehn's argument below for the reader's convenience, making the appropriate changes where needed.
The geometric description of the pairing is crucial for the results in the main body of the paper, and we obtain it along the way.

\subsection{Cotangent sequence of a vanishing locus}\label{cotan-van}

Let $S$ be a Noetherian scheme and let $T$ be an $S$-scheme. Let $\alpha \colon N' \lra N$ be a homomorphism of coherent $\OO_T$-modules, with $N$ locally free.
It gives rise to an adjoint $\alpha' \colon \SHom(N,N') \lra \mathcal{O}_T$ by setting $\alpha'(f):=\operatorname{trace}(\alpha\circ f)$.
The ideal $\operatorname{Im}(\alpha')\subseteq \mathcal{O}_T$ defines the largest closed subscheme $Z$ of $T$ for which $\alpha|_Z = 0$. We write
$Z = \text{Zeroes}(\alpha)$. The cotangent sequence 
of $Z\subseteq T$ over $S$ induces then an exact sequence
$$
\SHom(N\vert_Z,N'\vert_Z) \lra \Omega_{T/S}^1\vert_Z \lra \Omega_{Z/S}^1 \lra 0.
$$

\subsection{The cotangent sheaf of the flag scheme}\label{flag-cotangent-section}
In this subsection, we study $\text{Quot}(\mathcal{H},P_1,P_2)$ in the case
when $p \colon X\lra S$ is the identity.
In particular, $P_1 \equiv r_1$ and $P_2 \equiv r_2$ for some non-negative integers $r_1\geq r_2$. Under these hypotheses, we call $\operatorname{Quot}(\mathcal{H},P_1,P_2)$ the flag scheme $\text{$\mathbf{F}$} := \text{Fl}(\mathcal{H};r_1,r_2)$. We denote the structure morphism by $f : \mathbf{F} \lra S$. In the special case when $r_1=r_2=r$, we recover the Grassmannian scheme $G(\mathcal{H},r)$ over $S$.

To begin with, let
\begin{center}
    \begin{tikzcd}
        0 \ar[r] & \mathcal{K} \ar[r]\ar[d,hookrightarrow] & f^*\mathcal{H} \ar[r] \ar[d,equals] & \mathcal{F} \ar[r] \ar[d,twoheadrightarrow] & 0\\
        0 \ar[r] & \mathcal{J} \ar[r] & f^*\mathcal{H} \ar[r] & \mathcal{G} \ar[r] & 0
    \end{tikzcd}
\end{center}
be the tautological commutative diagram on $\mathbf{F}$.

\begin{proposition}\label{flag-cotangent}
The cotangent sheaf of $\mathbf{F}$ is isomorphic to
$$
\Omega_{\mathbf{F}/S}^1 \cong \text{coker}\left(\SHom(\mathcal{G},\mathcal{K}) \lra \SHom(\mathcal{F},\mathcal{K})\oplus \SHom(\mathcal{G},\mathcal{J})\right).
$$
\end{proposition}
\begin{proof}
There is a natural morphism
$$
\mathbf{F} \lra G(\mathcal{H},r_1)\times_S G(\mathcal{H},r_2)
$$
determined at the level of functors by sending
$$
[\mathcal{H}_T \twoheadrightarrow F \twoheadrightarrow G] \mapsto ([\mathcal{H}_T \twoheadrightarrow F],[\mathcal{H}_T\twoheadrightarrow G]).
$$
This is a closed immersion and its image is precisely the locus $\{([\mathcal{H}\twoheadrightarrow F],[\mathcal{H}\twoheadrightarrow G])\}$ where 
\begin{align*}
\operatorname{ker}(\mathcal{H} \lra F) \subseteq \text{ker}(\mathcal{H} \lra G).
\end{align*}
At the level of universal sheaves, the image of $\mathbf{F}$ is the vanishing locus of the homomorphism
$$
\pr_1^*\widetilde{\mathcal{K}} \lra \pr_2^*\widetilde{\mathcal{G}}
$$
for
\begin{align*}
0 \lra \widetilde{\mathcal{K}} \lra \mathcal{H}_{G(\mathcal{H},r_1)} \lra \widetilde{\mathcal{F}} \lra 0\\
0 \lra \widetilde{\mathcal{J}} \lra \mathcal{H}_{G(\mathcal{H},r_2)} \lra \widetilde{\mathcal{G}} \lra 0
\end{align*}
the universal exact sequences on $G(\mathcal{H},r_1)$ and $G(\mathcal{H},r_2)$ respectively. This means we have a surjection $\SHom(\pr_2^*\widetilde{\mathcal{G}},\pr_1^*\widetilde{\mathcal{K}}) \lra \mathcal{I}_{\mathbf{F}/G(\mathcal{H},r_1)\times_S G(\mathcal{H},r_2)}$ which, after restricting to $\mathbf{F}$ and composing with the cotangent sequence of $\mathbf{F} \subseteq G(\mathcal{H},r_1)\times_S G(\mathcal{H},r_2)$, yields an exact sequence \[\SHom(\mathcal{G},\mathcal{K}) \lra (\Omega_{G(\mathcal{H},r_1)\times_S G(\mathcal{H},r_2)/S}^1)\vert_{\mathbf{F}} \lra \Omega_{\mathbf{F}/S}^1 \lra 0 \tag{*}. \]
By \cite[Theorem 2.1]{Lehn.98} we have

\begin{lemma}\label{cotangent-grassmannian}
In the case of the Grassmannian scheme, the above sequence gives an isomorphism
\[
\SHom(\mathcal{F},\mathcal{K})\lra \Omega_{G(\mathcal{H},r)/S}^1.
\]
\end{lemma}

\noindent Combining Lemma \ref{cotangent-grassmannian}
with sequence $(*)$ above, we get
$$
\Omega_{\mathbf{F}/S}^1 \cong \operatorname{coker}\left(\SHom(\mathcal{G},\mathcal{K}) \lra \SHom(\mathcal{F},\mathcal{K})\oplus \SHom(\mathcal{G},\mathcal{J})\right). \qedhere
$$
\end{proof}

\begin{remark}
Note that there are at least three natural sequences
\[
0\lra \SHom(\mathcal{F}, \mathcal{K})
\lra
\SHom(\mathcal{F}, \mathcal{K})\oplus
\SHom(\mathcal{F}, \mathcal{K})\lra
\SHom(\mathcal{F}, \mathcal{K})\lra 0
\]
on $G(\mathcal{H}, r)$.
The first two come from the natural inclusions and projections. The third is induced
by a diagonal embedding as in Proposition \ref{flag-cotangent}.
If $r_1 = r_2 = r$, then the cotangent sequence
\[
0\lra \SHom(\mathcal{F}, \mathcal{K})
\lra \Omega_{G(\mathcal{H}, r)/S} ^1 \times_S 
\Omega_{G(\mathcal{H}, r)/S} ^1\lra
\Omega_{G(\mathcal{H}, r)/S} ^1\lra 0
\]
identifies with an exact sequence of the third type under the isomorphism 
\[
\Omega_{G(\mathcal{H}, r)/S} ^1\cong \SHom(\mathcal{F}, \mathcal{K}).
\]
\end{remark}

\subsection{The differential map of projective bundles}\label{conormal-sheaf-of-sequence}
Assume that $\mathcal{H}$ is locally free on $S$.
Let $g \colon Y\lra S$ be a flat projective morphism and let 
\[
0\lra K\lra g^*\mathcal{H}\lra F\lra 0
\]
be a short exact sequence of locally free sheaves on $Y$. Set $r:=\operatorname{rank}(F)$. By the universal property
of the Grassmannian scheme, there is a morphism
\[
\phi \colon Y\lra G(\mathcal{H}, r)
\]
such that the sequence on $Y$ is pulled back from the universal sequence on $G(\mathcal{H}, r)$.
Let $\mathcal{I}$ be the ideal of $\mathbb{P}(F)$
in $\mathbb{P}(g^*\mathcal{H})=Y\times_S \mathbb{P}(\mathcal{H})$ and let
$\pi \colon \mathbb{P}(F)\lra Y$ and $\lambda: \mathbb{P}(F)\lra \mathbb{P}(\mathcal{H})$ be the natural maps. 
Consider the diagram
\begin{center}
\begin{tikzcd}
& & 0\ar[d] & 0 \ar[d]& \\
& & \pi^* \Omega_Y ^1 (1)\ar[r, equals]\ar[d]&
\pi^* \Omega_Y ^1 (1)\ar[d]& \\
0\ar[r]& \mathcal{I}/\mathcal{I}^2 (1)\ar[r]\ar[d, "\cong"] &\pi^* \Omega_Y ^1 (1)\oplus \lambda^* \Omega_{\mathbb{P}(\mathcal{H})} ^1(1)\ar[r]\ar[d]& \Omega_{\mathbb{P}(F)}^1 (1) \ar[r]\ar[d]& 0\\
0\ar[r]& \pi^*K \ar[r]&\ar[r] \pi^* g^* \mathcal{H}\ar[r]\ar[d]&
\pi^* F\ar[r]\ar[d] & 0\\
& & \mathcal{O}_{\mathbb{P}(F)}(1)\ar[d]\ar[r, equals]&
\mathcal{O}_{\mathbb{P}(F)}(1)\ar[d]& \\
& & 0 & ~0 & 
\end{tikzcd}.
\end{center}
We have an induced map
\[
\pi^* K \lra 
{\Omega_{\mathbb{P}(g^* \mathcal{H})}^1}|_{\mathbb{P}(F)} (1)= \pi^* \Omega_Y ^1 (1) \oplus \lambda^* \Omega_{\mathbb{P}(\mathcal{H})} ^1(1)\lra
\pi^* \Omega_Y ^1 (1).
\]
By pushing forward to $Y$, we get a map
\[
K\lra \Omega_Y ^1 \otimes F.
\]
We have therefore and induced map
\[
K\otimes F^\vee \lra \Omega_Y ^1 \otimes F\otimes F^\vee \overset{\operatorname{id}\times \operatorname{tr}}{\lra} \Omega_Y ^1.
\]
This map coincides with the differential
\[
d \colon \phi^* \Omega_{G(H, r)} ^1\cong
\SHom(F, K)
\lra \Omega_Y ^1.
\]
To see that these two maps coincide, note that the above diagram is functorial in $Y$. Therefore, it suffices to check the statement for $Y=G(H, r)$. In this case, the differential $d$ is the identity and the claim follows from a careful analysis of the definition of all the maps involved.
As a consequence of this analysis, note that the map $\pi^*K \lra \pi^* \Omega_{G(\mathcal{H}, r)} ^1 (1)$ is induced by the natural pairing 
 \[
K \otimes  \SHom(K, F) \lra F.
 \]

\subsection{Inclusion of $\text{Quot}(\mathcal{H},P_1, P_2)$ in a large flag scheme}\label{Section-inclusion-Grassmannian}
Now we return to our original setting, with $p \colon X\lra S$ denoting a smooth morphism of Noetherian schemes of relative dimension $d$ with $p_*\OO_X \cong \OO_S$, $\mathcal{H}$ a coherent sheaf on $X$, and $P_1$ and $P_2$ a pair of numerical polynomials.

Let $\mathcal{L}$ be a line bundle on $X$ which is very ample relative to $p$. We may assume that the $\mathcal{O}_S$-algebra $\mathcal{S}:= \oplus_{i\geq 0}\mathcal{S}_i$, for $\mathcal{S}_i := p_*(\mathcal{L}^{\otimes i})$, is generated by $\mathcal{S}_1$ and that all summands $\mathcal{S}_i$ are locally free. Let $\widetilde{p} : X\times_S \textnormal{Quot}(\mathcal{H},P_1,P_2) \lra S$ and
$q:\textnormal{Quot}(\mathcal{H},P_1,P_2)\lra S$ denote the structure maps.
For any sheaf $\mathcal{E}$ on $X\times_S \textnormal{Quot}(\mathcal{H},P_1,P_2)$ and for each integer $m$ we define 
\[
\mathcal{E}_m:=\operatorname{pr}_{2,*} (\mathcal{E}\otimes \pr_1^*\mathcal{L}^{\otimes m}).
\]
We also denote $\mathcal{H}_m  := \operatorname{pr}_{2,*}(\pr_1^*(\mathcal{H}\otimes \mathcal{L}^{\otimes m}))$.
Consider the following commutative diagram
on $\operatorname{Quot}(\mathcal{H},P_1,P_2)$
\begin{equation}\label{m-diags}
    \begin{tikzcd}
        0 \ar[r] & \mathcal{K}_m \ar[r] \ar[d,hookrightarrow] & \mathcal{H}_m \ar[r]\ar[d,equals] & \mathcal{F}_m \ar[r]\ar[d] & 0\\
        0 \ar[r] & \mathcal{J}_m \ar[r] & \mathcal{H}_m \ar[r] & \mathcal{G}_m \ar[r] & 0
    \end{tikzcd}.
\end{equation}
Then there is an integer $m_0$ such that for all $m \geq m_0$, we have:
\begin{itemize}
    \item[(i)] the rows of Diagram \ref{m-diags} are exact (and hence the right vertical map is surjective);
    \item[(ii)] $\mathcal{K}_m$ (respectively $\mathcal{J}_m$) generates the graded $ \mathcal{S}$-module $\oplus_{m'\geq m}\mathcal{K}_{m'}$ (respectively $\oplus_{m'\geq m}\mathcal{J}_{m'}$);
    \item[(iii)] both $\mathcal{F}_m$ and $\mathcal{G}_m$ are locally free of ranks $P_1(m)$ and $P_2(m)$, respectively; and
    \item[(iv)] $q^* (p_*(\mathcal{H}\otimes \mathcal{L}^{\otimes m})) \cong \mathcal{H}_m$.
\end{itemize}
By (i), we have that $\mathcal{F}$
and $\mathcal{G}$ are determined by 
$\oplus_{m'\geq m}\mathcal{K}_{m'}$ and 
$\oplus_{m'\geq m}\mathcal{J}_{m'}$
respectively. By (ii) we have that the submodules
$$
    \oplus_{m'\geq m}\mathcal{K}_{m'}  \subseteq \oplus_{m'\geq m}\mathcal{H}_{m'} ~~\text{ and }~~\quad \oplus_{m'\geq m}\mathcal{J}_{m'}  \subseteq \oplus_{m'\geq m}\mathcal{H}_{m'}
$$
are determined by $\mathcal{K}_m \subseteq \mathcal{H}_m$ and $\mathcal{J}_m \subseteq \mathcal{H}_m$ respectively.
By (iii) and (iv)  the submodules $\mathcal{K}_m \subseteq \mathcal{J}_m \subseteq \mathcal{H}_m$ can be viewed as submodules
$$
\mathcal{K}_m \subseteq \mathcal{J}_m \subseteq q^*(p_*(\mathcal{H}\otimes \mathcal{L}^{\otimes m}))
$$
and are determined by a classifying morphism 
\begin{center}
\begin{tikzcd}
\iota_m \colon \textnormal{Quot}(\mathcal{H},P_1,P_2) \ar[r,hookrightarrow] & \mathbf{F}_m := \textnormal{Fl}(p_*(\mathcal{H}\otimes \mathcal{L}^{\otimes m}),P_1(m),P_2(m))
\end{tikzcd}
\end{center}
which is a closed embedding.

\subsection{Equations for $\textnormal{Quot}(\mathcal{H},P_1,P_2)$ in the flag scheme.}

Next, we use our computation of the cotangent sheaf of the flag scheme in \S\ref{flag-cotangent-section} to show we may pick a sequence of integers $m \leq n_0 < \cdots < n_{\ell}$ and a sheaf homomorphism $N'\lra N$, with $N$ locally free, vanishing along the image of the closed immersion
    \begin{center}
    \begin{tikzcd}
        \iota:=\prod_{i=0} ^{\ell} \iota_{n_i} \colon \textnormal{Quot}(\mathcal{H},P_1,P_2) \ar[r,hookrightarrow] & \Pi := 
        \prod_{i=0} ^{\ell} \mathbf{F}_{n_i}
    \end{tikzcd}
    \end{center}
For each $m$, let
\begin{center}
    \begin{tikzcd}
        0 \ar[r] & \mathcal{A}_{m} \ar[r,"\alpha_m"]\ar[d,hookrightarrow] & f_{m} ^*\mathcal{H}_{m} \ar[r,"\beta_m"] \ar[d,equals] & \mathcal{B}_{m} \ar[r] \ar[d,twoheadrightarrow] & 0\\
        0 \ar[r] & \mathcal{C}_{m} \ar[r, "\gamma_m"] & f_{m} ^*\mathcal{H}_{m} \ar[r,"\delta_m"] & \mathcal{D}_{m} \ar[r] & 0
    \end{tikzcd}
\end{center}
be the universal diagram on $\mathbf{F}_{m}$. For any pair of integers $m'\geq m$, there are multiplication maps
\[
\varphi_{m,m'} \colon \pr_1^*f_{m}^*\mathcal{S}_{m'-m}\otimes \pr_1^*\mathcal{A}_m \overset{\textnormal{id} \otimes \pr_1^*(\alpha_m)}{\lra} \pr_1^*f_m^*(\mathcal{S}_{m'-m}\otimes \mathcal{H}_m) \lra (f_m\circ\hspace{1pt} \pr_1)^*\mathcal{H}_{m'} = \pr_2^*f_{m'}^*\mathcal{H}_{m'} \overset{\pr_2^*(\beta_{m'})}{\lra} \operatorname{pr}_2 ^* \mathcal{B}_{m'}
\]
and (similarly)
\[
\psi_{m,m'}: \operatorname{pr}_1 ^* f_{m} ^*
\mathcal{S}_{m'-m}\otimes \operatorname{pr}_1 ^* \mathcal{C}_m \lra \operatorname{pr}_2 ^* \mathcal{D}_{m'}
\]
on $\mathbf{F}_m \times_S \mathbf{F}_{m'}$.

\begin{lemma}\label{vanishing}
There are integers $m\leq n_0 < \cdots < n_\ell$
such that the image of the closed immersion $\iota=\prod_{i=0}^{\ell} \iota_{n_i}$ is an open and closed subscheme of the vanishing locus of
\[
\varphi\oplus \psi:=\bigoplus_{i=1} ^{\ell} 
\varphi_{n_0, n_i}\oplus \psi_{n_0, n_i}.
\]
\end{lemma}

\begin{proof}
Pick $n_0\geq m$. The inclusions
\begin{center}\begin{tikzcd}[row sep = tiny]
\alpha_{n_0} \colon
\mathcal{A}_{n_0}\ar[r] &
f_{n_0} ^*(\mathcal{H}_{n_0})\\
\gamma_{n_0} \colon \mathcal{C}_{n_0} \ar[r] & f_{n_0} ^*(\mathcal{H}_{n_0})\end{tikzcd}\end{center}
induce morphisms
\begin{center}\begin{tikzcd}[row sep = tiny]
\chi_1 \colon f_{n_0}^* \mathcal{S}(-n_0)\otimes 
\mathcal{A}_{n_0}\ar[r] & 
\oplus_{n\geq 0} f_{n_0}^* (\mathcal{H}_n)\\
\chi_2 \colon f_{n_0}^*\mathcal{S}(-n_0) \otimes \mathcal{C}_{n_0}\ar[r] & \oplus_{n\geq 0}f_{n_0}^*(\mathcal{H}_n)
\end{tikzcd}\end{center}
on $\mathbf{F}_{n_0}$ and 
\begin{center}\begin{tikzcd}[row sep = tiny]
\chi_1' \colon \operatorname{pr}_1 ^* \mathcal{L}^{-n_0} 
\otimes \operatorname{pr}_2 ^* \mathcal{A}_{n_0} \ar[r] &
\operatorname{pr}_1 ^* (\mathcal{H})\\
\chi_2'\colon \pr_1^*\mathcal{L}^{-n_0}\otimes \pr_2^*\mathcal{C}_{n_0}\ar[r] & \pr_1^*(\mathcal{H})
\end{tikzcd}\end{center}
on $X\times_S \mathbf{F}_{n_0}$.
Let $C^i:=\operatorname{coker}(\chi_i)$ and $\mathscr{C}^i=\operatorname{coker}(\chi_i')$, for $i = 1,2$ in both cases.
Take the flattening stratification of $\mathscr{C}^i$ for $i=1,2$ over $\mathbf{F}_{n_0}$. The set of pairs of Hilbert polynomials 
$\mathcal{P}=\{(P(\mathscr{C}^1_g),P(\mathscr{C}^2_g))\hspace{3pt}|\hspace{3pt} g\in \mathbf{F}_{n_0}\}$ is finite and there is a decomposition of $\mathbf{F}_{n_0}$ into finitely many locally closed subschemes $\mathbf{F}_{n_0, (P_1',P_2')}$ with $(P_1',P_2')\in \mathcal{P}$ such that a morphism 
$\phi: T\lra \mathbf{F}_{n_0}$ factors through
\[
j \colon \bigsqcup_{(P_1',P_2')\in \mathcal{P}} \mathbf{F}_{n_0, (P_1',P_2')}\lra \mathbf{F}_{n_0}
\]
if and only if $\phi^* C^i$ is flat over $T$ for $i=1,2$. 
Moreover, there is an integer $m_1\geq n_0$ such that
each $\mathbf{F}_{n_0, (P_1',P_2')}$ is characterized as the maximal locally closed subscheme in $\mathbf{F}_{n_0}$
such that ${C^i_n}|_{\mathbf{F}_{n_0,(P_1',P_2')}}$ is locally free 
of rank $P_i'(n)$ for all $n\geq m_1$.
Then $\iota_{n_0} (\operatorname{Quot}(\mathcal{H}, P_1,P_2))=\mathbf{F}_{n_0,(P_1,P_2)}$ as subschemes. It turns out \cite[Lemme 3.7]{Grothendieck.60} that in fact $\mathbf{F}_{n_0, (P_1,P_2)}$ is proper. 
In particular, there is an open neighborhood $U$ of
$\mathbf{F}_{n_0, (P_1,P_2)}$ and an integer $m_2\geq m_1$ such that for all $n\geq m_2$ and all $u\in U$ we
have $P(\mathscr{C}^i_u, n)=\operatorname{dim}_{k(u)}(C^i_n (u))\leq P_i(n)$.
Denote by $Z^i_n\subseteq U$ the maximal closed subscheme of $U$ such that ${C^i_n}|_{Z^i_n}$ is locally free of rank $P_i(n)$ (so that $Z_n := Z^1_n\cap Z^2_n$ is the maximal closed subscheme of $U$ such that $C^i_n\vert_{Z^i_n}$ is locally free of rank $P_i(n)$ for both $i = 1,2$ simultaneously). Then
\[\iota_{n_0}(\text{Quot}(\mathcal{H},P_1,P_2))=\cap_{n\geq m_2} Z_n.\]
Since $\mathbf{F}_{n_0}$ is Noetherian, there will be finitely many integers 
$m_2\leq n_1\leq \cdots \leq n_{\ell}$ so that
\[
\iota_{n_0}(\text{Quot}(\mathcal{H},P_1,P_2))=\cap_{i=1}^{\ell} Z_{n_i}.
\] By our assumptions, for each $n_i$ and every morphism $\phi: T\lra U$, a diagram of surjections as follows
\begin{center}
    \begin{tikzcd}
        \phi^*C^1_{n_i} \ar[r, twoheadrightarrow] & N_1\ar[d,twoheadrightarrow]\\
        \phi^*C^2_{n_i} \ar[r,twoheadrightarrow] & N_2,
    \end{tikzcd}
\end{center}
where for $j = 1,2$ the sheaf $N_j$ is locally free of rank $P_j(n_i)$, must have that the horizontal maps are isomorphisms. This implies that for $j = 1,2$ we have $\textnormal{Grass}(C^j_{n_i}\vert_U,P_j(n_i)) \lra U$ is a closed immersion and the scheme theoretic intersection of the images is $Z_{n_i}$. In fact, the inclusion $Z_{n_i} \lra U$ factors through a closed immersion into $U\times \mathbf{F}_{n_i}$, whose image is precisely the vanishing locus of $\varphi_{n_0, n_i}\vert_{U\times \mathbf{F}_{n_i}}$ and $\psi_{n_0,n_i}\vert_{U\times \mathbf{F}_{n_i}}$. It follows that $\iota(\textnormal{Quot}(\mathcal{H},P_1,P_2)) \subseteq U\times \prod_{i=1}^{\ell}\mathbf{F}_{n_i}$ is the vanishing locus of $\varphi\oplus\psi\vert_{U\times \prod_{i=1}^{\ell}\mathbf{F}_{n_i}}$. Replacing $U$ by $\mathbf{F}_{n_0}$ we get the assertion of the lemma. 
\end{proof}

\begin{corollary}
There is an exact diagram on $\operatorname{Quot}(\mathcal{H}, P_1, P_2)$:
\begin{center}    
\begin{tikzcd}
\oplus_{i=1} ^{\ell} \left(\SHom(\mathcal{F}_{n_i},  \mathcal{K}_{n_0} \otimes S_{n_i - n_0})
\oplus
\SHom( \mathcal{G}_{n_i},  \mathcal{J}_{n_0} \otimes S_{n_i - n_0})\right) \ar[d]\ar[dr, "\Lambda^1 \oplus \Lambda^2"]& \\
\oplus_{i=0} ^{\ell} \Omega_{\mathbf{F}_{n_i}} ^1\ar[d]
&
\oplus_{i=1} ^{\ell} \left( \SHom( \mathcal{F}_{n_i},  \mathcal{K}_{n_i})\oplus \SHom ( \mathcal{G}_{n_i},  \mathcal{J}_{n_i})
\right) \ar[l]\\
\Omega_{\operatorname{Quot}(\mathcal{H}, P_1, P_2)/S} ^1 \ar[d]& \\
~0. &
\end{tikzcd}
\end{center}
\end{corollary}
\begin{proof}
This follows from Lemma \ref{vanishing}, Proposition \ref{flag-cotangent}, and the fact that $\iota^*\mathcal{A}_{n_i} = \mathcal{K}_{n_i}$, $\iota^*\mathcal{B}_{n_i} = \mathcal{F}_{n_i}$, $\iota^*\mathcal{C}_{n_i} = \mathcal{J}_{n_i}$ and $\iota^*\mathcal{D}_{n_i} = \mathcal{G}_{n_i}$. 
\end{proof}

Next, we identify the entries of $\Lambda^1$ and $\Lambda^2$.
Let $Z_{m,m'}\subseteq \mathbf{F}_m \times_S \mathbf{F}_{m'}$
be the vanishing locus of 
$\varphi_{m,m'}\oplus \psi_{m,m'}$. By definition, we have homomorphisms
\begin{align*}
&\operatorname{pr}_1 ^*(f_m ^* \mathcal{S}_{m'-m}\otimes\mathcal{A}_m |_{Z_{m,m'}})\lra \operatorname{pr}_2 ^* \mathcal{A}_{m'}|_{Z_{m,m'}}\\
&\operatorname{pr}_1 ^*( f_m ^* \mathcal{S}_{m'-m}\otimes \mathcal{C}_m |_{Z_{m,m'}})\lra \operatorname{pr}_2 ^* \mathcal{C}_{m'}|_{Z_{m,m'}}\\
&\operatorname{pr}_1 ^*f_m ^* \mathcal{S}_{m'-m}\otimes \pr_2^*\mathcal{B}_{m'} ^\vee |_{Z_{m,m'}}\lra \operatorname{pr}_1 ^* \mathcal{B}_{m} ^\vee |_{Z_{m,m'}} \\
&\operatorname{pr}_1 ^*f_m ^* \mathcal{S}_{m'-m}\otimes \pr_2^*\mathcal{D}_{m'} ^\vee |_{Z_{m,m'}}\lra \operatorname{pr}_1 ^* \mathcal{D}_{m} ^\vee |_{Z_{m,m'}}.
\end{align*}

\begin{lemma}\label{entries}
There is a commutative diagram
\begin{center}
\begin{tikzcd}
    \left.\stackMath\parenVectorstack[c]{{
        \pr_2 ^*\mathcal{B}_{m'}^\vee \otimes \operatorname{pr}_1 ^* \mathcal{A}_m\otimes \pr_1^*f_m^*\mathcal{S}_{m'-m}
    }
    {
        \oplus
    }
    {
        \operatorname{pr}_2 ^* \mathcal{D}_{m'} ^\vee \otimes \operatorname{pr}_1 ^* \mathcal{C}_m\otimes\pr_1^*f_m^*\mathcal{S}_{m'-m}
    }}\right\vert_{Z_{m,m'}} \ar[r] \ar[d] &
    \left.\stackMath\parenVectorstack[c]{{
        \pr_1^*((\mathcal{B}_m^{\vee}\otimes \mathcal{A}_m)\oplus(\mathcal{D}_m^{\vee}\otimes \mathcal{C}_m))
    }
    {
        \oplus
    }
    {
        \pr_2^*((\mathcal{B}_{m'}^{\vee}\otimes \mathcal{A}_{m'})\oplus(\mathcal{D}_{m'}^{\vee}\otimes \mathcal{C}_{m'}))
    }}\right\vert_{Z_{m,m'}} \ar[d]\\
    \mathcal{I}_{Z_{m,m'}}/\mathcal{I}_{Z_{m,m'}} ^2\ar[r] &
    \operatorname{pr}_1 ^* \Omega_{\mathbf{F}_m} ^1 |_{Z_{m,m'}}
    \oplus
    \operatorname{pr}_2 ^* \Omega_{\mathbf{F}_{m'}} ^1 |_{Z_{m,m'}},
\end{tikzcd}
\end{center}
where the top horizontal map is induced by the above multiplications. 
\end{lemma}

Let 
\begin{align*}
&\lambda_{1,i} ' \colon \SHom( \mathcal{F}_{n_i},  \mathcal{S}_{n_i-n_0}\otimes  \mathcal{K}_{n_0})
\lra
\SHom( \mathcal{F}_{n_i},  \mathcal{K}_{n_i})\\
&\lambda_{1,i} '' \colon \SHom(\mathcal{F}_{n_i},  \mathcal{S}_{n_i-n_0}\otimes  \mathcal{K}_{n_0})
\lra
\SHom( \mathcal{F}_{n_0}, \mathcal{K}_{n_0})\\
&\lambda_{2,i} ' \colon \SHom(\mathcal{G}_{n_i},  \mathcal{S}_{n_i-n_0}\otimes \mathcal{J}_{n_0})
\lra
\SHom( \mathcal{G}_{n_i},  \mathcal{J}_{n_i})\\
&\lambda_{2,i} '' \colon \SHom( \mathcal{G}_{n_i},  \mathcal{S}_{n_i-n_0}\otimes  \mathcal{J}_{n_0})
\lra
\SHom( \mathcal{G}_{n_0},  \mathcal{J}_{n_0})
\end{align*}
be the natural morphisms induced by multiplication. Then we see that the entries of $\Lambda^j$ are 
\begin{align*}
&\Lambda_{0i}^j=-\lambda_{j,i}'' \text{ for } 1\leq i\leq l \text{ and } j=1,2,\\
&\Lambda_{ii} ^j=\lambda_{j,i}' \text{ for } 1\leq i\leq l \text{ and } j=1,2,\\
&\Lambda_{ab} ^j=0\text{ otherwise}.
\end{align*}
    

\subsection{Proof of Theorem \ref{cotan-nested}} 
We are now ready to prove Theorem \ref{cotan-nested}. Define
\begin{align*}
&\mathcal{K}' = \operatorname{ker}\left(\operatorname{pr}_2 ^* \mathcal{K}_{n_0} \otimes \pr_1^*\mathcal{L}^{-n_0}
\lra \mathcal{K}
\right)\\
&\mathcal{J}' = \operatorname{ker}\left(\operatorname{pr}_2 ^* \mathcal{J}_{n_0} \otimes \pr_1^*\mathcal{L}^{-n_0}
\lra \mathcal{J}
\right).
\end{align*}
Pick an integer $m_3\geq m_2$ (recall $m_2$ was chosen in the proof of Lemma \ref{vanishing}) such that for all $n\geq m_3$ we have, for $i > 0$,
\begin{align*}
& R^i \pr_{2,*} (\mathcal{K}'\otimes \pr_1^*\mathcal{L}^n) = 0,\\
& R^i \pr_{2,*} (\mathcal{J}'\otimes \pr_1^*\mathcal{L}^n) = 0,\\
& \operatorname{pr}_2 ^* \mathcal{K}_n ' \lra
\mathcal{K}' \otimes \pr_1^*\mathcal{L}^n \text{ is surjective and}\\
& \operatorname{pr}_2 ^* \mathcal{J}_n ' \lra
\mathcal{J}' \otimes \pr_1^*\mathcal{L}^n \text{ is surjective.}
\end{align*}
Consider the short exact sequences 
\begin{align*}
0\lra \mathcal{K}'\otimes \pr_1^*\mathcal{L}^n \lra
\operatorname{pr}_2 ^* \mathcal{K}_{n_0}
\otimes \pr_1^*\mathcal{L}^{n-n_0}\lra
\mathcal{K}\otimes \pr_1^*\mathcal{L}^n
\lra 0,\\
0\lra \mathcal{J}'\otimes\pr_1^*\mathcal{L}^n \lra
\operatorname{pr}_2 ^* \mathcal{J}_{n_0}
\otimes \pr_1^*\mathcal{L}^{n-n_0}\lra
\mathcal{J}\otimes \pr_1^*\mathcal{L}^n
\lra 0.
\end{align*}
Pushing forward along $\pr_2$, one gets exact sequences
\begin{align*}
0\lra \mathcal{K}' _n \lra
\mathcal{K}_{n_0}
\otimes \mathcal{S}_{n-n_0}\lra
\mathcal{K}_n 
\lra 0,\\
0\lra \mathcal{J}'_n \lra
\mathcal{J}_{n_0}
\otimes \mathcal{S}_{n-n_0}\lra
\mathcal{J}_n
\lra 0.
\end{align*}

\begin{corollary}\label{penultimate-corollary}
The cotangent sheaf $\Omega_{\operatorname{Quot}(\mathcal{H}, P_1, P_2)/S}^1$ is isomorphic to the cokernel of the map
\[
\bigoplus_{i=1} ^{\ell}
\SHom(\mathcal{F}_{n_i}, \mathcal{K}_{n_i}')\oplus
\SHom(\mathcal{G}_{n_i}, \mathcal{J}_{n_i}')
\lra
\Omega_{\mathbf{F}_{n_0}} ^1.
\]
\end{corollary}

Consider now the exact sequences
\begin{align*}
\bigoplus_{i=1} ^{\ell} \operatorname{pr}_2 ^* \mathcal{K}_{n_i} ' \otimes \mathcal{L}^{-n_i}\lra
\operatorname{pr}_2 ^* \mathcal{K}_{n_0}\otimes \mathcal{L}^{-n_0}
\lra
\mathcal{K}\lra 0,\\
\bigoplus_{i=1} ^{\ell} \operatorname{pr}_2 ^* \mathcal{J}_{n_i} ' \otimes \mathcal{L}^{-n_i}\lra
\operatorname{pr}_2 ^* \mathcal{J}_{n_0}\otimes \mathcal{L}^{-n_0}
\lra
\mathcal{J}\lra 0.
\end{align*}
By applying $\SExt ^n _{\operatorname{pr}_2}(\mathcal{F}, -\otimes \omega)$
and
$\SExt ^n _{\operatorname{pr}_2}(\mathcal{G}, -\otimes \omega)$, respectively, we get the second row of the following commutative diagram
\begin{center}
\hspace*{-30pt}\begin{tikzcd}
\bigoplus_{i=1} ^{\ell} \SExt ^d _{\operatorname{pr}_2}(\mathcal{G}, \operatorname{pr}_1 ^* \mathcal{K}_{n_i}' (-n_i) \otimes \omega)\ar[r]\ar[d]&
\SExt ^d _{\operatorname{pr}_2}(\mathcal{G}, \operatorname{pr}_1 ^* \mathcal{K}_{n_0} (-n_0) \otimes \omega)\ar[r]\ar[d]&
\SExt ^d _{\operatorname{pr}_2}(\mathcal{G}, \mathcal{K}\otimes \omega)\ar[r]\ar[d]& 0\\
\bigoplus_{i=1}^{\ell}
\parenVectorstack{{
    \SExt ^d _{\operatorname{pr}_2}(\mathcal{F}, \operatorname{pr}_1 ^* \mathcal{K}_{n_i}' (-n_i) \otimes \omega)
    }
    {
        \oplus
    }
    {
        \SExt ^d _{\operatorname{pr}_2}(\mathcal{G}, \operatorname{pr}_1 ^* \mathcal{J}_{n_i}' (-n_i) \otimes \omega)
    }}\ar[r]&
\parenVectorstack{{
    \SExt ^d _{\operatorname{pr}_2}(\mathcal{F}, \operatorname{pr}_1 ^* \mathcal{K}_{n_0} (-n_0) \otimes \omega)
    }
    {
        \oplus
    }
    {
        \SExt ^d _{\operatorname{pr}_2}(\mathcal{G}, \operatorname{pr}_1 ^* \mathcal{J}_{n_0} (-n_0) \otimes \omega)
    }}\ar[r]&
\parenVectorstack{{
    \SExt ^d _{\operatorname{pr}_2}(\mathcal{F}, \mathcal{K} \otimes \omega)
    }
    {
        \oplus
    }
    {
        \SExt ^d _{\operatorname{pr}_2}(\mathcal{G}, \mathcal{J} \otimes \omega)
    }}\ar[r]&0.
\end{tikzcd}
\end{center}
By Grothendieck duality the above diagram may be rewritten as 
\begin{center}
\begin{tikzcd}
\bigoplus_{i=1}^{\ell} \SHom (\mathcal{G}_{n_i}, \mathcal{K}_{n_i}' )\ar[r]\ar[d]&
\SHom (\mathcal{G}_{n_0}, \mathcal{K}_{n_0})\ar[r]\ar[d]&
\SExt ^d _{\operatorname{pr}_2}(\mathcal{G}, \mathcal{K}\otimes \omega)\ar[r]\ar[d]& 0\\
\bigoplus_{i=1} ^{\ell}
\parenVectorstack{{
        \SHom (\mathcal{F}_{n_i}, \mathcal{K}_{n_i}')
    }
    {
        \oplus
    }
    {
        \SHom (\mathcal{G}_{n_i},  \mathcal{J}_{n_i}' )
    }}\ar[r]&
\parenVectorstack{{
        \SHom (\mathcal{F}_{n_0}, \mathcal{K}_{n_0})
    }
    {
        \oplus
    }
    {
        \SHom (\mathcal{G}_{n_0}, \mathcal{J}_{n_0})
    }}\ar[r]&
\parenVectorstack{{
        \SExt ^d _{\operatorname{pr}_2}(\mathcal{F}, \mathcal{K} \otimes \omega)
    }
    {
        \oplus
    }
    {
        \SExt ^d _{\operatorname{pr}_2}(\mathcal{G}, \mathcal{J} \otimes \omega)
    }}\ar[r]&0.
\end{tikzcd}
\end{center}
Now we can witness the desired result because the cokernel of the middle vertical map receives a (composed) morphism from the lower-left term in the diagram, and the cokernel of that map is precisely $\Omega_{\textnormal{Quot}(\mathcal{H},P_1,P_2)/S}^1$ by Corollary \ref{penultimate-corollary}. So by commutativity of the diagram, the right vertical map provides the desired resolution of $\Omega_{\textnormal{Quot}(\mathcal{H},P_1,P_2)/S}^1$.

\subsection{The differential map of the universal family}
We conclude by analyzing the natural pairing on $\operatorname{Quot}(\mathcal{H}, P)$. In this case, the above diagram simplifies to a sequence
\[
\bigoplus_{i=1} ^{\ell} \SHom(\mathcal{F}_{n_i}, \mathcal{K}_{n_i} )\lra
\SHom (\mathcal{F}_{n_0}, \mathcal{K}_{n_0})\lra
\SExt ^d _{\operatorname{pr}_2}(\mathcal{F}, \mathcal{K}\otimes \omega)\lra 0.
\]
The map
\[
\Omega_{G(\mathcal{H}_{n_0}, P(n_0))/S} ^1\vert_{\textnormal{Quot}(\mathcal{H},P)} \lra \Omega_{\operatorname{Quot}(\mathcal{H}, P)} ^1
\lra 0
\]
is simply the differential (pullback) of the inclusion map $\textnormal{Quot}(\mathcal{H},P) \hookrightarrow G(\mathcal{H}_{n_0},P(n_0))$. We also have a diagram
\begin{center}
\begin{tikzcd}
\operatorname{pr}_{2,*} \SHom (\mathcal{K}, \mathcal{F})\ar[r]\ar[d]& 
\SHom(\mathcal{K}_{n_0}, \mathcal{F}_{n_0})
\ar[d]\\
T_{\operatorname{Quot}(\mathcal{H}, P)/S} 
\ar[r]
 &T_{G(\mathcal{H}_{n_0}, P(n_0))/S},
\end{tikzcd}
\end{center}
where the vertical maps are isomorphisms, the bottom map is the differential (pushforward), and the top map is induced by twisting. Therefore, the natural pairing
\[
\mathcal{K}\lra
\SHom\left(\operatorname{pr}_2 ^* T_{\operatorname{Quot}(\mathcal{H},P)/S} , \mathcal{F}\right).
\]
is induced by the pairing on $G(\mathcal{H}_{n_0}, P(n_0))$, which by section \ref{conormal-sheaf-of-sequence} is finally induced by the differential map of the universal family.

\vspace{30pt}

\vspace*{-20pt}

\glssetwidest{*********}

\setglossarystyle{alttree}

\printglossary

\end{document}